\pdfoutput=1
\documentclass[11pt]{article}
\usepackage{graphicx}
\usepackage{amsmath,amsthm,amssymb,amscd}
\usepackage[noblocks]{authblk}
\usepackage{fourier}
\usepackage{listings}
\usepackage{paralist}
\usepackage{enumerate}
\usepackage{color}
\usepackage[margin=1in]{geometry}
\usepackage[square,numbers,sort]{natbib}
\usepackage[small,bf]{caption}
\usepackage[colorlinks,linkcolor=black,bookmarksopen,bookmarksnumbered,
citecolor=black,urlcolor=black]{hyperref}

\graphicspath{{figs/}}
\allowdisplaybreaks
\numberwithin{equation}{section}
\newtheorem{theorem}{Theorem}[section]
\newtheorem*{theorem*}{Theorem}
\newtheorem{proposition}[theorem]{Proposition}
\newtheorem*{proposition*}{Proposition}
\newtheorem{lemma}[theorem]{Lemma}
\newtheorem*{lemma*}{Lemma}

\newtheorem*{claim*}{Claim}

\newtheorem*{axiom*}{Axiom}

\newtheorem*{conjecture*}{Conjecture}

\newtheorem*{corollary*}{Corollary}

\theoremstyle{definition}
\newtheorem{definition}[theorem]{Definition}
\newtheorem*{definition*}{Definition}
\newtheorem{example}[theorem]{Example}
\newtheorem*{example*}{Example}

\newtheorem*{exercise*}{Exercise}
\newtheorem*{recall*}{Recall}

\theoremstyle{remark}

\newtheorem*{note*}{Note}
\newtheorem{remark}[theorem]{Remark}
\newtheorem*{remark*}{Remark}
\newtheorem{notation}[theorem]{Notation}
\newtheorem*{notation*}{Notation}

\newtheorem*{question*}{Question}

\newtheorem*{fact*}{Fact}

\theoremstyle{theorem}

\theoremstyle{definition}

\theoremstyle{remark}

\def\R{\mathbb{R}}

\def\Z{\mathbb{Z}}

\newcommand{\Lforms}[2]{L^#1\Omega^#2\bigl(\abs{K}\bigr)}

\DeclareMathOperator{\im}{im}
\DeclareMathOperator{\grad}{grad}
\DeclareMathOperator{\curl}{curl}

\def\div{\operatorname{div}}

\newcommand{\simplex}[1]{[v_0,\ldots,v_{#1}]}

\newcommand{\hasface}{\succeq}
\newcommand{\faceof}{\preceq}
\newcommand{\hasprprface}{\succ}
\newcommand{\prprfaceof}{\prec}

\DeclareMathOperator{\St}{St}
\DeclareMathOperator{\ClSt}{\overline{St}}

\def\d{\operatorname{d}}

\DeclareMathOperator{\codiff}{\delta} 
\DeclareMathOperator{\dcodiff}{\delta}
\DeclareMathOperator{\dd}{d} 
\DeclareMathOperator{\hodge}{\ast} 
\DeclareMathOperator{\dual}{\star}
\DeclareMathOperator{\laplacian}{\Delta}
\DeclareMathOperator{\deRham}{R}
\DeclareMathOperator{\whitney}{W}

\DeclareMathOperator{\boundary}{\partial}

\newcommand{\cochainBasis}[2]{\bigl(\sigma^#1_#2\bigr)^\ast}

\newcommand{\abs}[1]{\lvert#1\rvert}

\newcommand{\norm}[1]{\lVert#1\rVert}

\newcommand{\eval}[2]{\langle #1,#2 \rangle}

\newcommand{\ainnerproduct}[2]{\langle #1, #2 \rangle}
\newcommand{\aInnerproduct}[2]{\bigl\langle #1, #2 \bigr\rangle}
\newcommand{\pinnerproduct}[2]{( #1, #2 )}
\newcommand{\pInnerproduct}[2]{\bigl( #1, #2 \bigr)}

\def\p{\phantom{-}}

\long\def\symbolfootnote[#1]#2{\begingroup%
\def\thefootnote{\fnsymbol{footnote}}\footnote[#1]{#2}\endgroup}

\def\vertexarray{\mathbb{V}}
\def\simplexarray{\mathbb{S}}
\def\cubearray{\mathbb{C}}
\DeclareMathOperator{\vcurl}{\textbf{curl}}

\title{PyDEC: Software and Algorithms for \\ Discretization of
  Exterior Calculus}
\author{Nathan Bell}
\affil{NVIDIA Corporation, nbell@nvidia.com}
\author{Anil N. Hirani}
\affil{Department of Computer Science,
  University of Illinois at Urbana-Champaign,
  hirani@cs.illinois.edu}
\date{}

\begin{document}

\maketitle

\begin{abstract}
  This paper describes the algorithms, features and implementation of
  PyDEC, a Python library for computations related to the
  discretization of exterior calculus.  PyDEC facilitates inquiry into
  both physical problems on manifolds as well as purely topological
  problems on abstract complexes. We describe efficient algorithms for
  constructing the operators and objects that arise in discrete
  exterior calculus, lowest order finite element exterior calculus
  and in related topological problems.  Our algorithms are formulated
  in terms of high-level matrix operations which extend to arbitrary
  dimension. As a result, our implementations map well to the
  facilities of numerical libraries such as NumPy and SciPy. The
  availability of such libraries makes Python suitable for prototyping
  numerical methods.  We demonstrate how PyDEC is used to solve
  physical and topological problems through several concise examples.

  \paragraph{Categories and Subject Descriptors: } G.4 [Mathematical
  Software]; G.1.8 [Numerical Analysis]: Partial Differential
  Equations -- \emph{Finite element methods, Finite volume methods,
    Discrete exterior calculus, Finite element exterior calculus}; I.3.5
  [Computer Graphics]: Computational Geometry and Object Modeling --
  \emph{Geometric algorithms, languages, and systems, Computational
    topology}

\end{abstract}

\section{Introduction}
\label{sec:introduction}
Geometry and topology play an increasing role in the modern language
used to describe physical problems \cite{AbMaRa1988,Frankel2004}.  A
large part of this language is \emph{exterior calculus} which
generalizes vector calculus to smooth manifolds of arbitrary
dimensions. The main objects are differential forms (which are
anti-symmetric tensor fields), general tensor fields, and vector
fields defined on a manifold. In addition to physical applications,
differential forms are also used in cohomology theory in
topology~\cite{BoTu1982}.

Once the domain of interest is discretized, it may not be smooth and
so the objects and operators of exterior calculus have to be
reinterpreted in this context. For example, a surface in $\R^3$ may be
discretized as a two dimensional simplicial complex embedded in
$\R^3$, i.e., as a triangle mesh. Even when the domain is a simple
domain in space, such as an open subset of the plane or space the
discretization is usually in the form of some mesh. The various
objects of the problem then become defined in a piecewise varying
fashion over such a mesh and so a discrete calculus is required there
as well. After discretizing the domains, objects, and operators, one
can compute numerical solutions of partial differential equations
(PDEs), and compute some topological invariants using the same
discretizations. Both these classes of applications are considered
here.

There have been several recent developments in the discretization of
exterior calculus and in the clarification of the role of algebraic
topology in computations. These go by various names, such as covolume
methods \cite{NiTr2006}, support operator methods \cite{Shashkov1996},
mimetic discretization \cite{HySh1997a,BoHy2006}, discrete exterior
calculus (DEC) \cite{Hirani2003,DeHiLeMa2005, GiBa2010}, compatible
discretization, finite element exterior calculus \cite{ArFaWi2006,
  ArFaWi2010, HoSt2011}, edge and face elements or Whitney forms
\cite{Bossavit1988a,Hiptmair2002a}, and so on. PyDEC provides an
implementation of discrete exterior calculus and lowest order finite
element exterior calculus using Whitney forms.

Within pure mathematics itself, ideas for discretizing exterior
calculus have a long history. For example the de Rham map that is
commonly used for discretizing differential forms goes back at least
to \cite{Rham1955}. The reverse operation of interpolating discrete
differential forms via the Whitney map appears in
\cite{Whitney1957}. A combinatorial (discrete) Hodge decomposition
theorem was proved in \cite{Dodziuk1976} and the idea of a
combinatorial Hodge decomposition dates to \cite{Eckmann1945}. More
recent work on discretization of Hodge star and wedge product is in
\cite{Wilson2007,Wilson2008}.  Discretizations on other types of
complexes have been developed as well \cite{GrHi1999,Sen2003}.

\subsection{Main contributions}
In this paper we describe the algorithms and design of PyDEC, a Python
software library implementing various complexes and operators for
discretization of exterior calculus, and the algorithms and data
structures for those. In PyDEC all the discrete operators are
implemented as sparse matrices and we often reduce algorithms to a
sequence of standard high-level operations, such as sparse
matrix-matrix multiplication~\cite{BaDo1993}, as opposed to more
specialized techniques and ad hoc data structures.  Since these
high-level operations are ultimately carried out by efficient,
natively-compiled routines (e.g. C or Fortran implementations) the
need for further algorithmic optimization is generally unnecessary.

As is commonly done, in PyDEC we implement discrete differential forms
as real valued cochains which will be defined in
Section~\ref{sec:overview}. PyDEC has been used in a
thesis~\cite{Bell2008}, in classes taught at University of Illinois,
in experimental parts of some computational topology
papers~\cite{DeHiKr2010, DuHi2011, HiKaWaWa2011}, in Darcy
flow~\cite{HiNaCh2011}, and in least squares ranking on
graphs~\cite{HiKaWa2011}. The PyDEC source code and examples are
publicly available~\cite{BeHi2008a}. We summarize here our
contributions grouped into four areas.

\medskip\noindent{\bf Basic objects and methods:} 
\begin{inparaenum}[(1)]
\item Data structures for : simplicial complexes of dimension $n$
  embedded in $\R^N$, $N\ge n$; abstract simplicial complexes;
  Vietoris-Rips complexes for points in any dimension; and regular
  cube complexes of dimension $n$ embedded in $\R^n$;
\item Cochain objects for the above complexes;
\item Discrete exterior derivative as a coboundary operator,
  implemented as a method for cochains on various complexes.
\end{inparaenum}

\medskip\noindent{\bf Finite element exterior calculus:}
\begin{inparaenum}[(1)]
\item Fast algorithm to construct sparse mass matrices for Whitney
  forms by eliminating repeated computations; and
\item Assembly of stiffness matrices for Whitney forms from mass
  matrices by using products of boundary and stiffness matrices. Note
  that only the lowest order ($\mathcal{P}_1^{-}$) elements of finite
  element exterior calculus are implemented in PyDEC.
\end{inparaenum}

\medskip\noindent{\bf Discrete exterior calculus:}
\begin{inparaenum}[(1)]
\item Diagonal sparse matrix discrete Hodge star for well-centered
  (circumcenters inside simplices) and Delaunay simplicial complexes
  (with an additional boundary condition);
\item Circumcenter calculation for $k$-simplex in an $n$-dimensional
  simplicial complex embedded in $\R^N$ using a linear system in
  barycentric coordinates; and
\item Volume calculations for primal simplices and circumcentric dual
  cells.
\end{inparaenum}

\medskip\noindent{\bf Examples:}
\begin{inparaenum}[(1)]
\item Resonant cavity curl-curl problem; 
\item Flow in porous medium modeled as Darcy flow, i.e., Poisson's
  equation in first order (mixed) form;
\item Cohomology basis calculation for a simplicial mesh, using
  harmonic cochain computation using Hodge decomposition;
\item Finding sensor network coverage holes by modeling an abstract,
  idealized sensor network as a Rips complex; and
\item Least squares ranking on graphs using Hodge decomposition of
  partial pairwise comparison data.
\end{inparaenum}

\section{Overview of PyDEC}
\label{sec:overview}

One common type of discrete domain used in scientific computing is
triangle or tetrahedral mesh. These and their higher dimensional
analogues are implemented as $n$-dimensional simplicial complexes
embedded in $\R^N$, $N \ge n$. Simplicial complexes are useful even
without an embedding and even when they don't represent a manifold,
for example in topology and ranking problems. Such abstract simplicial
complexes without any embedding for vertices are also implemented in
PyDEC. The other complexes implemented are regular cube complexes and
Rips complexes. Regular cubical meshes are useful since it is easy to
construct domains even in high dimensions whereas simplicial meshing
is hard enough in 3 dimensions and rarely done in 4 or larger
dimensions.  Rips complexes are useful in applications such as
topological calculations of sensor network coverage
analysis~\cite{SiGh2007}. The representations used for these four
types of complexes are described in
Section~\ref{sec:simplicial_complex}-\ref{sec:abstrct_smplcl_cmplx}.
A complex that is a manifold (i.e., locally Euclidean) will be
referred to as \emph{mesh}.

The definitions here are given for simplicial complexes and generalize
to the other types of complexes implemented in PyDEC. In PyDEC we only
consider integer valued chains and real-valued cochains. Also, we are
only interested in finite complexes, that is, ones with a finite
number of cells. Let $K$ be a finite simplicial complex and denote its
underlying space by $\abs{K}$. Give $\abs{K}$ the subspace topology as
a subspace of $\R^N$ (a set $U$ in $\abs{K}$ is open iff $U \cap
\abs{K}$ is open in $\R^N$). For a finite complex this is the same as
the standard way of defining topology for $\abs{K}$ \cite[pages
8-9]{Munkres1984} and $\abs{K}$ is a closed subspace of $\R^N$.

An oriented simplex with vertices $v_0,\dots,v_p$ will be written as
$\simplex{p}$ and given names like $\sigma_i^p$ with the superscript
denoting the dimension and subscript denoting its place in some
ordering of $p$-simplices.  Sometimes the dimensional superscript
and/or the indexing subscript will be dropped. The orientation of a
simplex is one of two equivalence classes of vertex orderings. Two
orderings are equivalent if one is an even permutation of the
other. For example $[v_0,v_1,v_2]$ and $[v_1, v_2, v_0]$ denote the
same oriented triangle while $[v_0,v_2,v_1]$ is the oppositely
oriented one.

A \emph{$p$-chain of $K$} is a function $c$ from oriented
$p$-simplices of $K$ to the set of integers $\Z$, such that
$c(-\sigma) = -c(\sigma)$ where $-\sigma$ is the simplex $\sigma$
oriented in the opposite way. Two chains are added by adding their
values. Thus $p$-chains are formal linear combinations (with integer
coefficients) of oriented $p$-dimensional simplices. The space of
$p$-chains is denoted $C_p(K)$ and it is a free abelian group.  See
\cite[page 21]{Munkres1984}. Free
abelian groups have a basis and one does not need to impose a vector
space structure. For example, a basis for $C_p(K)$ is the set of
integer valued functions that are 1 on a $p$-simplex and 0 on the
rest, with one such basis element corresponding to each $p$-simplex.
These are called \emph{elementary chains} and the one corresponding to
a $p$-simplex $\sigma^p$ will also be referred to as $\sigma^p$.  The
existence of this basis and the addition and negation of chains is the
only aspect that is important for this paper.  The intuitive way to
think of chains is that they play a role similar to that played by the
domains of integration in the smooth theory. The negative sign allows
one to talk about orientation reversal and the integer coefficient
allows one to say how many times integration is to be done on that
domain.

Sometimes we will need to refer to a \emph{dual mesh} which will in
general be a cell complex obtained from a subdivision of the given
complex $K$. We'll refer to the dual complex as $\dual K$. For a
discrete Hodge star diagonal matrix of DEC, the dual mesh is the one
obtained from circumcentric subdivision of a well-centered or Delaunay
simplicial complex and such a Hodge star is described in
Section~\ref{sec:metric}.

Homomorphisms from the $p$-chain group $C_p(K)$ to $\R$ are called
\emph{$p$-cochains of $K$} and denoted $C^p(K;\R)$. This set is an
abelian group and also a vector space over $\R$. Similarly the dual
$p$-cochains are denoted $C^p(\dual K; \R)$ or $D^p(\dual K; \R)$.
The discretization map from space of smooth $p$-forms to $p$-cochains
is called the \emph{de Rham} map $\deRham : \Omega^p(K) \to C^p(K;\R)$
or $\deRham : \Omega^p(K) \to C^p(\dual K;\R)$. See \cite{Rham1955,
  Dodziuk1976}. For a smooth $p$-form $\alpha$, the de Rham map is
defined as $\deRham : \alpha \mapsto (c \mapsto \int_c \alpha)$ for
any chain $c \in C_p(K)$. We will denote the evaluation of the cochain
$\deRham(\alpha)$ on a chain $c$ as $\eval{\deRham(\alpha)}{c}$. A
basis for $C^p(K;\R)$ is the set of \emph{elementary cochains}. The
elementary cochain $(\sigma^p)^\ast$ is the one that takes value 1 on
elementary chain $\sigma^p$ and 0 on the other elementary chains. Thus
the vector space dimension of $C^p(K;\R)$ is the number of
$p$-simplices in $K$. We'll denote this number by $N_p$. Thus $N_0$
will be the number of vertices, $N_1$ the number of edges, $N_2$ the
number of triangles and so on.

Like most of the numerical analysis literature mentioned in
Section~\ref{sec:introduction} we assume that the smooth forms are
either defined in the embedding space of the simplicial complex, or on
the complex itself, or can be obtained by pullback from the manifold
that the complex approximates. In contrast, most mathematics
literature quoted including \cite{Whitney1957,Dodziuk1976} uses
simplicial complex defined on the smooth manifold as a ``curvilinear''
triangulation.  In the applied literature, the complex approximates
the manifold. Many finite element papers deal with open subsets of the
plane or $\R^3$ so they are working with triangulations of a manifold
with piecewise smooth boundaries. Surface finite element methods have
been studied outside of exterior calculus~\cite{DeDz2007}.  A
variational crimes methodology is used for finite element exterior
calculus on simplicial approximations of manifolds
in~\cite{HoSt2011}. In the computer graphics literature,
piecewise-linear triangle mesh surfaces embedded in $\R^3$ are common
and convergence questions for operators on such surfaces have been
studied \cite{HiPoWa2006}. In light of all of these, PyDEC's framework
of using simplicial or other approximations of manifolds is
appropriate.

Operators such as the discrete exterior derivative ($\d$) and Hodge
star ($\hodge$) can be implemented as sparse matrices. At each
dimension, the exterior derivative can be easily determined by the
incidence structure of the given simplicial mesh. For DEC the Hodge
star is a diagonal matrix whose entries are determined by the ratios
of primal and dual volumes. Care is needed for dual volume calculation
when the mesh is not well-centered. For finite element exterior
calculus we implement Whitney forms. The corresponding Hodge star is
the mass matrix which is sparse but not diagonal. One of the stiffness
matrices can be obtained from it by combining it with the exterior
derivative.

Once the matrices implementing the basic operators have been
determined, they can be composed together to obtain other operators
such as the codifferential ($\codiff$) and Laplace-deRham
($\laplacian$).  While this composition could be performed manually,
i.e. the appropriate set of matrices combined to form the desired
operation, it is prone to error.  In PyDEC this composition is handled
automatically.  For example, the function $\d(.)$ which implements the
exterior derivative, looks at the dimension of its argument to
determine the appropriate matrix to apply.  The same method can be
applied to the codifferential function $\codiff(.)$, which then makes
their composition $\codiff(\d(.))$ work automatically.  This
automation eliminates a common source of error and makes explicit
which operators are being used throughout the program.

PyDEC is intended to be fast, flexible, and robust.  As an interpreted
language, Python by itself is not well-suited for high-performance
computing.  However, combined with numerical libraries such as NumPy
and SciPy one can achieve fast execution with only a small amount of
overhead.  The NumPy array structure, used extensively throughout
SciPy and PyDEC, provides an economical way of storing N-dimensional
arrays (comparable to C or Fortran) and exposes a C API for
interfacing Python with other, potentially lower-level
libraries~\cite{WaCoVa2011}. In this way, Python can be used to
efficiently ``glue'' different highly-optimized libraries together
with greater ease than a purely C, C++, or Fortran implementation
would permit~\cite{Oliphant2007}.  Indeed, PyDEC also makes extensive
use of the \texttt{sparse} module in SciPy which relies on
natively-compiled C++ routines for all performance-sensitive
operations, such as sparse matrix-vector and matrix-matrix
multiplication.  PyDEC is therefore scalable to large data sets and
capable of solving problems with millions of elements~\cite{Bell2008}.

Even large-scale, high-performance libraries such as Trilinos provide
Python bindings showing that Python is useful beyond the prototyping
stage.  We also make extensive use of Python's built-in unit testing
framework to ensure PyDEC's robustness.  For each non-trivial
component of PyDEC, a number of examples with known results are used
to check for consistency.

\subsection{Previous work}

Discrete differential forms now appear in several finite element
packages such as FEMSTER \cite{CaRiWh2005}, DOLFIN \cite{LoWe2010} and
{deal.II}~\cite{BaHaKa07}. These libraries support arbitrary order
conforming finite element spaces in two and three dimensions. In
contrast, for finite elements PyDEC supports simplicial and cubical
meshes of arbitrary dimension, albeit with lowest order elements. In
addition, PyDEC also supports the operators of discrete exterior
calculus and complexes needed in topology. We note that
Exterior~\cite{LoMa2008}, an experimental library within the
FEniCS~\cite{LoMaWe2012} project, realizes the framework developed by
Arnold et al.~\cite{ArFaWi2009} which generalizes to arbitrary order
and dimension. Exterior uses symbolic methods and supports integration
of forms on the standard simplex. PyDEC supports mass and stiffness
matrices on simplicial and cubical complexes. The discovery of lower
dimensional faces in a complex and the computation of all the boundary
matrices is also implemented in PyDEC.

The other domain where PyDEC is useful is in computational
topology. There are several packages in this domain as well, and again
PyDEC has a different set of features and aims from these.
In~\cite{KaMiKoMr2004} efficient techniques are developed for finding
meaningful topological structures in cubical complexes, such as
digital images. In addition to simplicial and cubical manifolds, PyDEC
also provides support for abstract simplicial complexes such as the
Rips complex of a point set. The Applied and Computational Topology
group at Stanford University has been the source for several packages
for computational topology. These include various versions of PLEX
such as JPlex and javaPlex which are designed for persistent homology
calculations. Another package from the group is Dionysus, a C++
library that implements persistent homology and
cohomology~\cite{EdLeZo2002, ZoCa2005} and other interesting
topological and geometric algorithms. In contrast, we view the role of
PyDEC in computational topology as providing a tool to specify and
represent different types of complexes, compute their boundary
matrices, and compute cohomology representatives with or without
geometric information.

\section{Simplicial Complex Representation}
\label{sec:simplicial_complex}
Before detailing the algorithms used to implement discretizations of
exterior calculus, we discuss the representation of various complexes,
starting in this section with simplicial complexes.  Consider the
triangle mesh shown in Figure~\ref{fig:mesh_example_filled} with
vertices and faces enumerated as shown.  This example mesh is
represented by arrays
\begin{align*}
\vertexarray = \begin{bmatrix}
0 & 0 \\
1 & 0 \\
2 & 0 \\
1 & 1 \\
2 & 1
\end{bmatrix},
&&
\simplexarray_2 = \begin{bmatrix}
0 & 1 & 3 \\
1 & 2 & 3 \\
2 & 4 & 3
\end{bmatrix},
\end{align*}
where the subscript $2$ denotes the dimension of the simplices.  The
$i$-th row of $\vertexarray$ contains the spatial coordinates of the
$i$-th vertex.  Likewise the $i$-th row of simplex array
$\simplexarray_2$ contains the indicies of the vertices that form the
$i$-th triangle.  The indices of each simplex in $\simplexarray_2$ in
this example are ordered in a manner that implies a counter-clockwise
orientation for each.  For an $n$-dimensional discrete manifold, or
mesh, arrays $\vertexarray$ and $\simplexarray_n$ suffice to describe
the computational domain.

\begin{figure}
\begin{center}
     \includegraphics[height=1in]{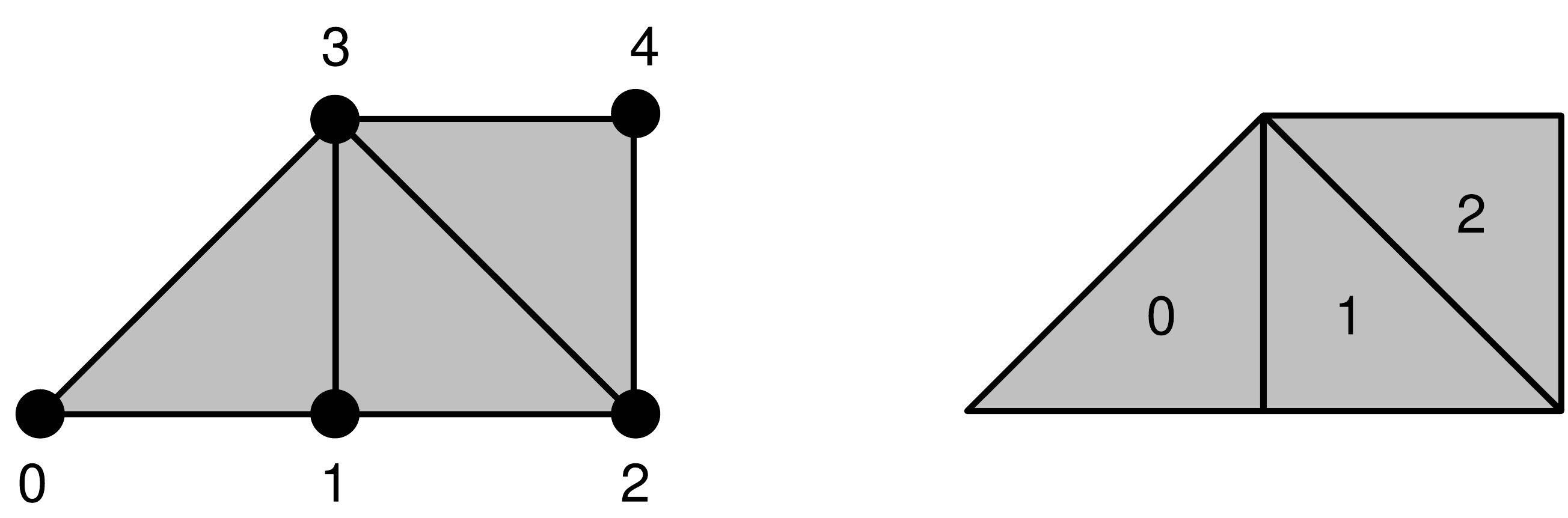}
\end{center}
\caption{Simplicial mesh with enumerated vertices and simplices.}
\label{fig:mesh_example_filled}
\end{figure}

In addition to $\vertexarray$ and $\simplexarray_n$, an
$n$-dimensional simplicial complex is comprised by its $p$-dimensional
faces, $\simplexarray_p, 0 \leq p < n$.  In the case of
Figure~\ref{fig:mesh_example_filled}, these are
\begin{align*}
\simplexarray_0 = \begin{bmatrix}
0\\
1\\
2\\
3\\
4
\end{bmatrix},
&&
\simplexarray_1 = \begin{bmatrix}
0 & 1\\
0 & 3\\
1 & 2\\
1 & 3\\
2 & 3\\
2 & 4\\
3 & 4
\end{bmatrix},
\end{align*}
which correspond to the vertices (0-simplices) and oriented edges
(1-simplices) of the complex.  A graphical representation of this
simplicial complex is shown in Figure~\ref{fig:mesh_example_oriented}.
Since the orientation of the lower ($< n$) dimensional faces is
arbitrary, we use the convention that face indices will be in sorted
order.  Furthermore, we require the rows of $\simplexarray$ to be
sorted in lexicographical order.  As pointed out in
Section~\ref{sec:exterior_derivative} and~\ref{sec:whtny_innrprdct},
these conventions facilitate efficient construction of differential
operators and stiffness matrices.

\begin{figure}
\begin{center}
     \includegraphics[height=1in]{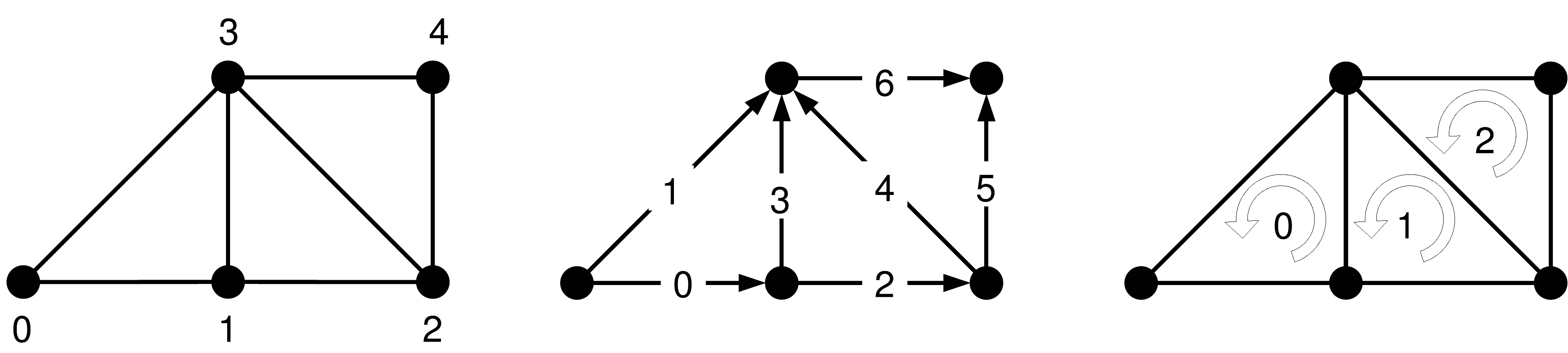}
\end{center}
\caption{Simplicial complex with oriented edges and triangles.}
\label{fig:mesh_example_oriented}
\end{figure}

\section{Regular Cube Complex Representation}
\label{sec:cube_complex}

PyDEC provides a regular cube complex of dimension $n$ embedded in
$\R^n$ for any $n$. As mentioned earlier, in dimension higher than 3,
constructing simplicial manifold complexes is hard. In fact, even
construction of good tetrahedral meshes is still an active area in
computational geometry. This is one reason for using regular cube
complexes in high dimensions. Moreover, for some applications, like
topological image analysis or analysis of voxel data, the regular cube
complex is a very convenient framework~\cite{KaMiKoMr2004}.

A regular cube complex can be easily specified by an $n$-dimensional
array of binary values (bitmap) and a regular $n$-dimensional cube is
placed where the bit is on. For example the cube complex shown in
Figure~\ref{fig:cube_mesh} can be created by specifying the bitmap
array
\[
\begin{bmatrix}
  0 & 1\\ 1 & 1
\end{bmatrix}\, .
\]

A bitmap suffices to describe the top level cubes, but a cube array
(like simplex array) is used during construction of differential
operators and for computing faces. In this paper we describe the
construction of exterior derivative, Hodge star and Whitney forms on
simplicial complexes. For cube complexes we describe only the
construction of exterior derivative and lower dimensional
faces. However, the other operators and objects are also implemented
in PyDEC for such complexes. For example, Whitney-like elements for
hexahedral grids are described in~\cite{BoRo2002} and are implemented
in PyDEC.

\begin{figure}
\begin{center}
     \includegraphics[height=1.25in]{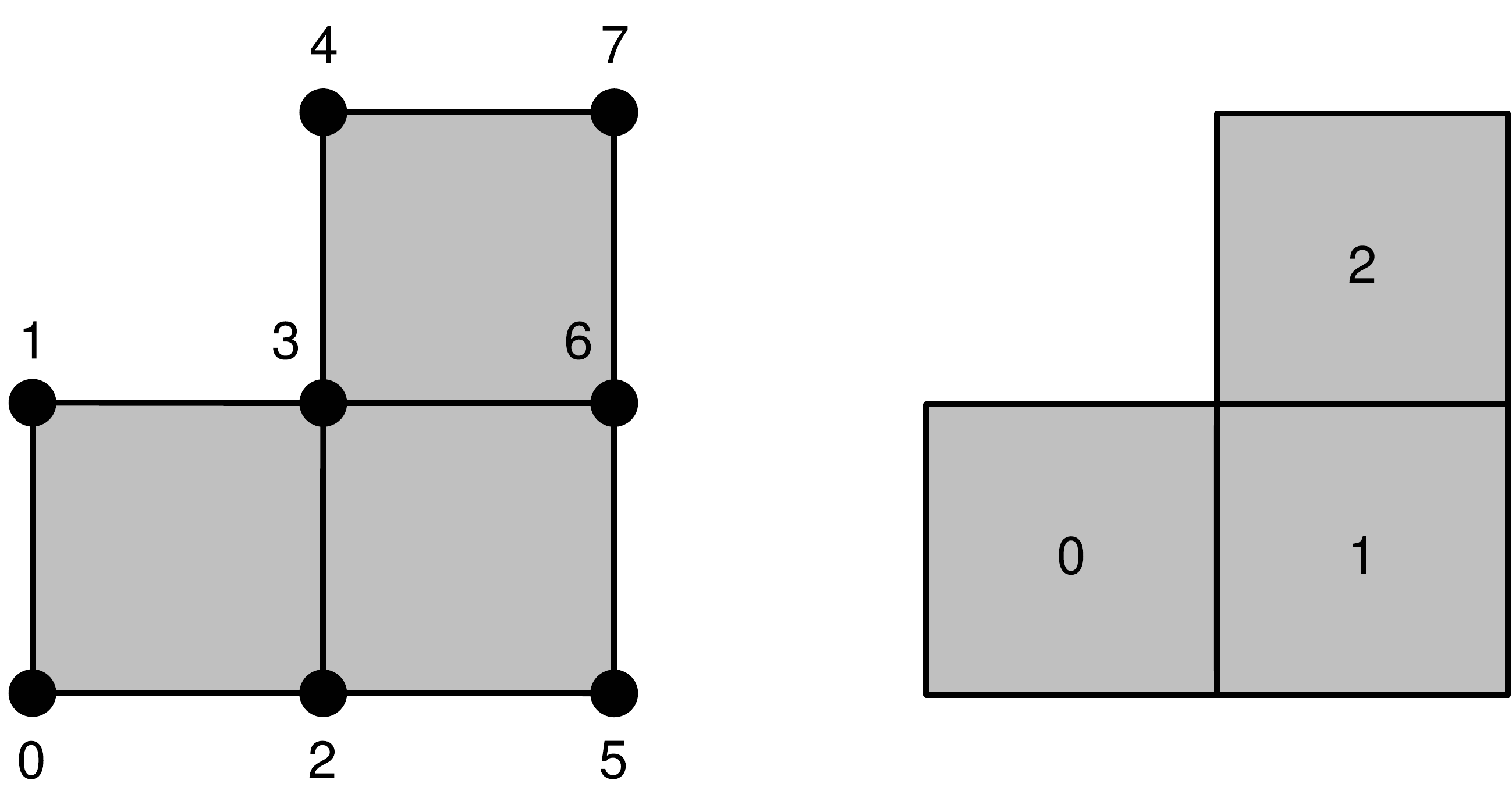}
\end{center}
\caption{Regular cube mesh with enumerated vertices and faces.}
\label{fig:cube_mesh}
\end{figure}

\begin{figure}
\begin{center}
     \includegraphics[height=1.25in]{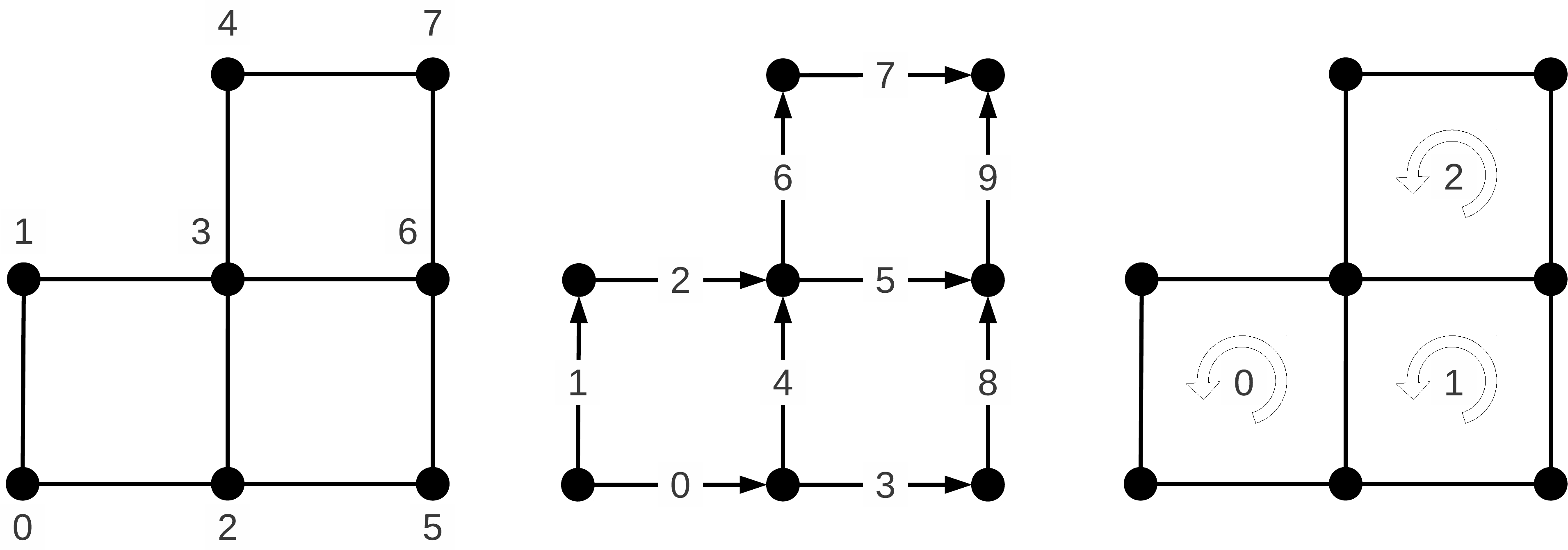}
\end{center}
\caption{Regular cube complex with oriented edges and faces.}
\label{fig:cube_oriented}
\end{figure}

Converting a bitmap representation of a mesh into a cube array
representation is straightforward.  For example, the cube array
representation of the mesh in Figure~\ref{fig:cube_mesh} is
\[
\cubearray_2 = 
\begin{bmatrix}
  0 & 0 & & 0 & 1 \\
  1 & 0 & & 0 & 1 \\
  1 & 1 & & 0 & 1
\end{bmatrix}.
\]
As with the simplex arrays, the rows of $\cubearray_2$ correspond to
individual two-dimensional cubes in the mesh.  The two left-most
columns of $\cubearray_2$ encode the origins of each two-dimensional
cube, namely $(0,0)$, $(1,0)$ and $(1,1)$.  The remaining two columns
encode the coordinate directions spanned by the cube.  Since
$\cubearray_2$ represents the top-level elements, all cubes span both
the $x$ (coordinate $0$) and $y$ (coordinate $1$) dimensions.  In
general, the first $n$ columns of $\cubearray_k$ encode the origin or
corner of a cube while the remaining $k$ columns identify the
coordinate directions swept out by the cube.  We note that the cube
array representation is similar to the cubical notation used by
Sen~\cite{Sen2003}.

The edges of the mesh in Figure~\ref{fig:cube_mesh} are represented by
the cube array
\[
\cubearray_1 = 
\begin{bmatrix}
  0 & 0 & & 0 \\
  0 & 0 & & 1 \\
  0 & 1 & & 0 \\
  1 & 0 & & 0 \\
  1 & 0 & & 1 \\
  1 & 1 & & 0 \\
  1 & 1 & & 1 \\
  1 & 2 & & 0 \\
  2 & 0 & & 1 \\
  2 & 1 & & 1
\end{bmatrix}\,
\]
where again the first two columns encode the origin of each edge and
the last column indicates whether the edge points in the $x$ or $y$
direction.  For example, the row $[0,0,0]$ corresponds to edge $0$
in Figure~\ref{fig:cube_oriented} which begins at $(0,0)$ and
extends one unit in the $x$ direction. Similarly the row $[2, 1, 1]$
encodes an edge starting at $(2,1)$ extending one unit in the $y$
direction. Since zero-dimensional cubes (points) have no spatial extent
their cube array representation
\[
\cubearray_0 = 
\begin{bmatrix}
  0 & 0 \\
  0 & 1 \\
  1 & 0 \\
  1 & 1 \\
  1 & 2 \\
  2 & 0 \\
  2 & 1 \\
  2 & 2 \\
\end{bmatrix}\,
\]
contains only their coordinate locations.

The cube array provides a convenient representation for regular cube
complexes.  While a bitmap representation of the top-level cubes is
generally more compact, the cube array representation generalizes
naturally to lower-dimensional faces and is straightforward to
manipulate.

\section{Rips Complex Representation}
\label{sec:rips_complex}

The \emph{Rips complex}, or Vietoris-Rips complex of a point set is
defined by forming a simplex for every subset of vertices with
diameter less than or equal to a given distance $r$.  For example, if
pair of vertices $(v_i,v_j)$ are no more than distance $r$ apart, then
the Rips complex contains an edge (1-simplex) between the vertices.
In general, a set of $p \geq 2$ vertices forms a $(p-1)$-simplex when
all pairs of vertices in the set are separated by at most $r$.

\begin{figure}
  \begin{center}
    \includegraphics[width=0.7\textwidth]{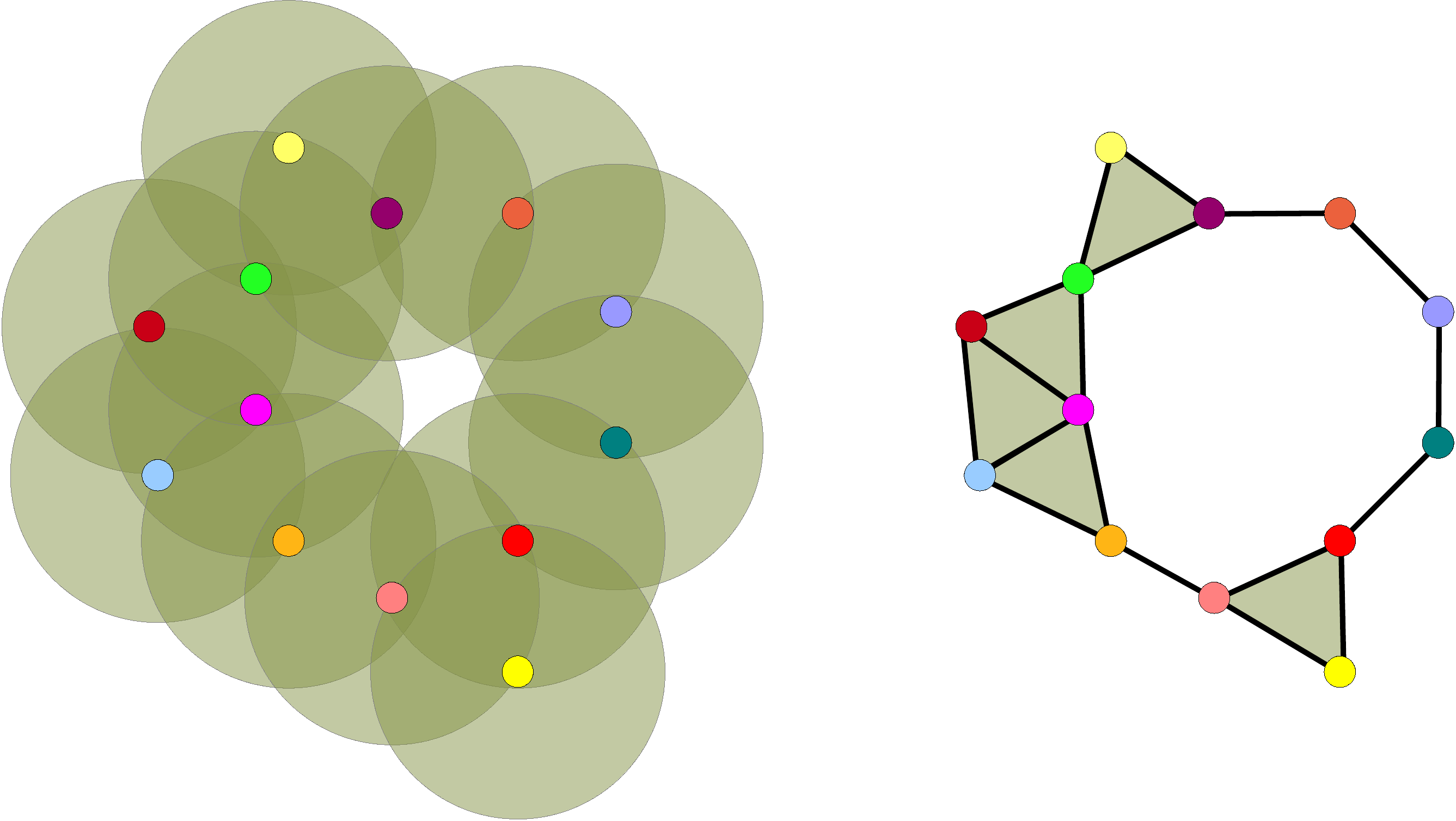} 
  \end{center}
\caption{Broadcast radii and Rips complex for a sensor network.}
\label{fig:sensor_network_illustration}
\end{figure}

In recent work, certain sensor network coverage problems have been
shown to reduce to finding topological properties of the network's
Rips complex, at least for an abstract model of sensor
networks~\cite{SiGh2007}.  Such coordinate-free methods
rely only on pairwise communication between nodes and do not require
the use of positioning devices.  These traits are especially important
in the context of ad-hoc wireless networks with limited per-node
resources.  Figure \ref{fig:sensor_network_illustration} depicts a
planar sensor network and its associated Rips complex.

In this section we describe an efficient method for computing the Rips
complex for a set of points.  Although we consider only the case of
points embedded in Euclidean space, our methodology applies to more
general metric spaces.  Indeed, only the construction of the
$1$-skeleton of the Rips complex requires metric information.  The
higher-dimensional simplices are constructed directly from the
$1$-skeleton.

We compute the $1$-skeleton of the Rips complex with a kD-Tree data
structure.  Specifically, for each vertex $v_i$ we compute the set of
neighboring vertices $\{ v_j : \norm{v_j - v_i} \leq r\}$.  The
hierarchical structure of the kD-Tree allows such queries to be
computed efficiently.

The $1$-skeleton of the Rips complex is stored in an array
$\simplexarray_1$, using the convention discussed in Section
\ref{sec:simplicial_complex}.  Additionally, the (oriented) edges of
the $1$-skeleton are used to define $\mathbb{E}$, a directed graph
stored in a sparse matrix format.  Specifically, $\mathbb{E}(i,j) = 1$
if $[i,j]$ is an edge of the Rips complex, and zero otherwise.  For
the Rips complex depicted in Figure
\ref{fig:simple_rips_complex_example},
\begin{align*}
\simplexarray_1 = \begin{bmatrix}
0 & 1 \\
0 & 2 \\
0 & 3 \\
1 & 3 \\
2 & 3
\end{bmatrix},
&&
\mathbb{E} = \begin{bmatrix}
0 & 1 & 1 & 1 \\
0 & 0 & 0 & 1 \\
0 & 0 & 0 & 1 \\
0 & 0 & 0 & 0
\end{bmatrix},
\end{align*}
are the corresponding simplex array and directed graph respectively.

\begin{figure}
%     2
%     | \
%     |  3
%     | / \
%     0-----1
  \begin{center}
    \includegraphics[width=0.7\textwidth]
    {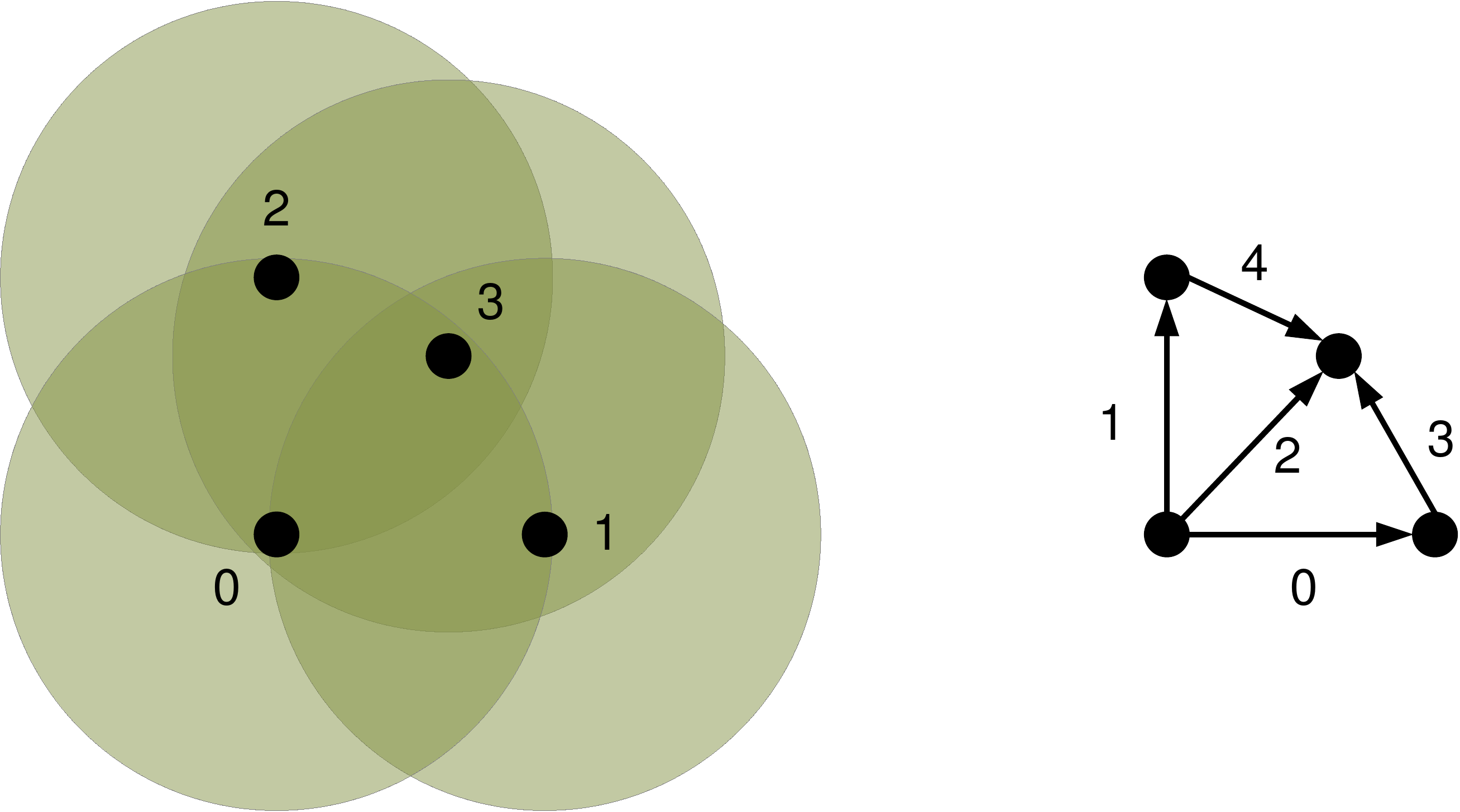} 
  \end{center}
  \caption{Five directed edges form the $1$-skeleton of the Rips
    complex.}
  \label{fig:simple_rips_complex_example}
\end{figure}

The arrays of higher dimensional simplices $\simplexarray_2,
\simplexarray_3, \ldots$ can be computed as follows.  Let
$\mathbb{F}_p$ denote the (sparse) matrix whose rows are identified
with the $p$-simplices as specified by $\simplexarray_p$.  Each row of
$\mathbb{F}_p$ encodes the vertices which form the corresponding
simplex.  Specifically, $\mathbb{F}_p(i,j)$ takes the value $1$ if the
$i$-th simplex contains vertex $j$ and zero otherwise.  For the
example shown in Figure \ref{fig:simple_rips_complex_example},
\begin{align*}
\mathbb{F}_1 = \begin{bmatrix}
1 & 1 & 0 & 0 \\
1 & 0 & 1 & 0 \\
1 & 0 & 0 & 1 \\
0 & 1 & 0 & 1 \\
0 & 0 & 1 & 1
\end{bmatrix},
\end{align*}
encodes the edges stored in $\simplexarray_1$.  Once $\mathbb{F}_p$ is
constructed we compute the sparse matrix-matrix product $\mathbb{F}_p
\,\mathbb{E}$.  For our example the result is
\begin{align*}
\mathbb{F}_1\, \mathbb{E} = 
\begin{bmatrix}
0 & 1 & 1 & 2\\
0 & 1 & 1 & 2\\
0 & 1 & 1 & 1\\
0 & 0 & 0 & 1\\
0 & 0 & 0 & 1
\end{bmatrix}.
\end{align*}
Like $\mathbb{F}_p$, the product $\mathbb{F}_p\, \mathbb{E}$ is a
matrix that relates the $p$-simplices to the vertices: the matrix
entry $(i,j)$ of $\mathbb{F}_p\, \mathbb{E}$ counts the number of
directed edges that exist from the vertices of simplex $i$ to vertex
$j$. When the value of $(i,j)$ entry of $\mathbb{F}_p\, \mathbb{E}$ is
equal to $p+1$, we form a $p$-simplex of the Rips complex by
concatenating simplex $i$ with vertex $j$.  In the example, matrix
entries $(0,3)$ and $(1,3)$ of $\mathbb{F}_1\,\mathbb{E}$ are equal to
$2$ which implies that the $2$-skeleton of the Rips complex contains
two simplices, formed by appending vertex $3$ to the $1$-simplices
$[0, 1]$ and $[0, 2]$, or
\begin{align*}
\simplexarray_2 = \begin{bmatrix}
0 & 1 & 3 \\
0 & 2 & 3
\end{bmatrix},
\end{align*}
in array format.  This process may be applied recursively to develop
higher dimensional simplices $\simplexarray_3, \simplexarray_4,
\ldots$ as required by the application. Thus our algorithm computes
simplices of the Rips complex with a handful of sparse and dense
matrix operations.

\section{Abstract Simplicial Complex Representation}
\label{sec:abstrct_smplcl_cmplx}

In Section~\ref{sec:rips_complex} we saw an example of a simplicial
complex which was not a manifold complex
(Figure~\ref{fig:sensor_network_illustration}). Rips complexes
described in Section~\ref{sec:rips_complex} demonstrate one way to
construct such complexes in PyDEC, starting from locations of
vertices. There are other applications, for example in topology, where
we would like to create a simplicial complex which is not necessarily
a manifold. In addition we would like to do this without requiring
that the location of vertices be given. For example, in topology,
surfaces are often represented as a polygon with certain sides
identified. One way to describe such an object is as an \emph{abstract
  simplicial complex} \cite[Section 3]{Munkres1984}. This is a
collection of finite nonempty sets such that if a set is in the
collection, so are all the nonempty subsets of it.
Figure~\ref{fig:asc_examples} shows two examples of abstract simplicial
complexes created in PyDEC.
\begin{figure}[ht]
  \centering
  \includegraphics[scale=0.4,trim=1.5in 1in 1.5in 2in, clip]
  {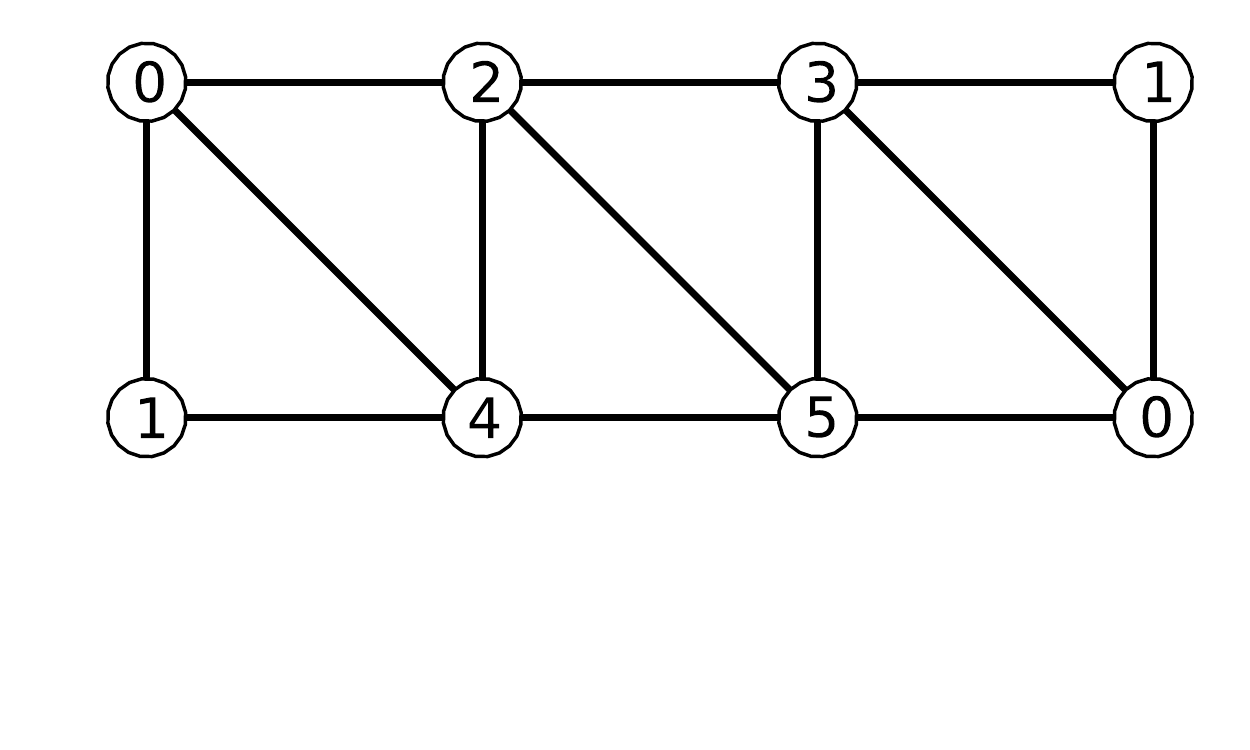}
  \includegraphics[scale=0.4,trim=1.5in 1in 1.5in 1in, clip]
  {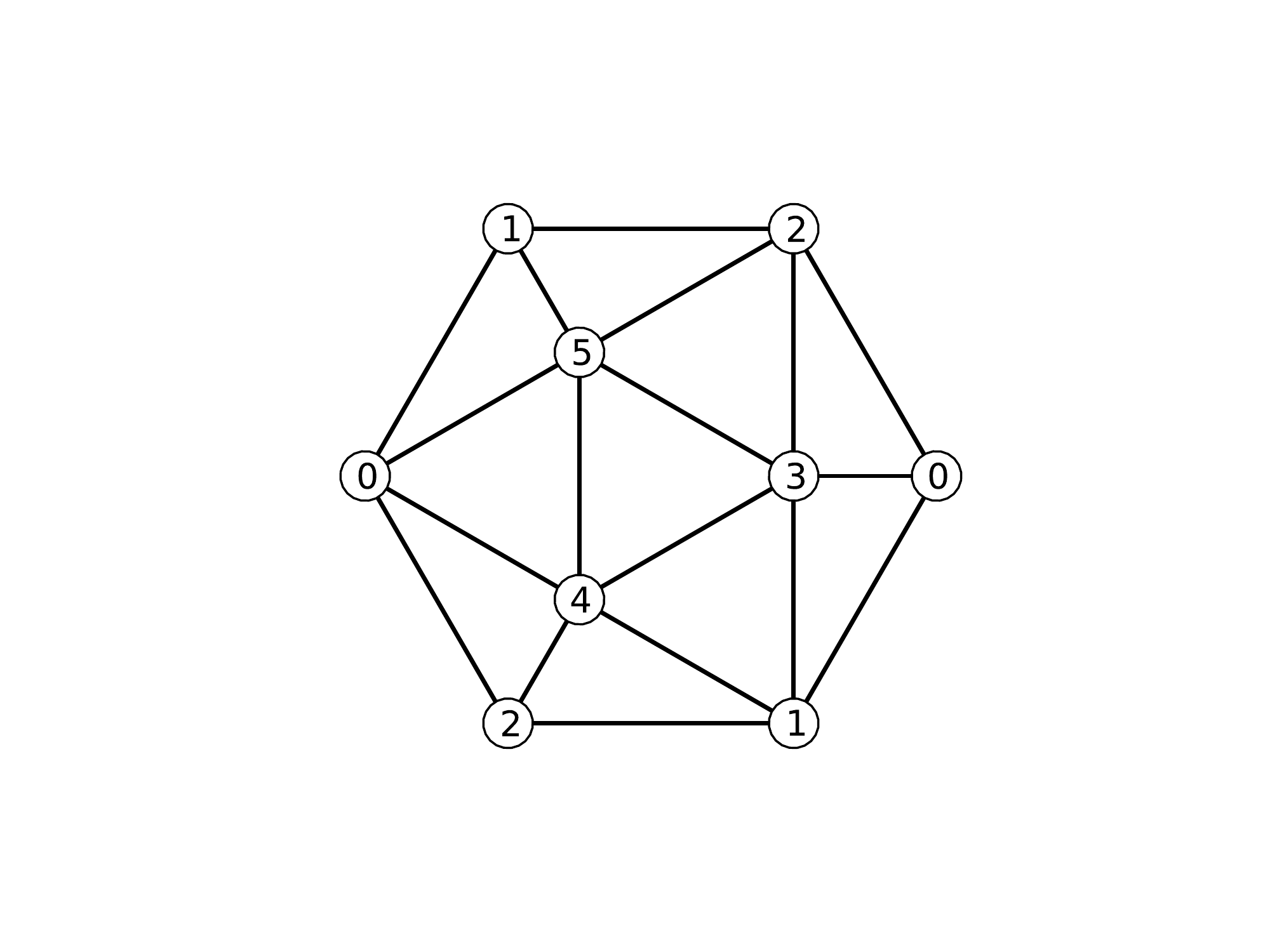}
  \caption{Examples of abstract simplicial complexes. The one of the
    left represents the triangulation of a M\"obius strip and the one
    on the right that of a projective plane.}
  \label{fig:asc_examples}
\end{figure}

In PyDEC, abstract simplicial complexes are created by specifying a
list of arrays. Each array contains simplices of a single dimension,
specified as an array of vertex numbers. Lower dimensional faces of a
simplex need not be specified explicitly. For example the M\"obius
strip triangulation shown in Figure~\ref{fig:asc_examples} can be
created by giving the array
\[
\begin{bmatrix}
0 & 1 & 3\\0 & 3 & 5\\3 & 2 & 5\\5 & 2 & 4\\2 & 0 & 4\\0 & 1 & 4
\end{bmatrix}
\]
as input to PyDEC. Abstract simplicial complexes need not be a
triangulation of a manifold. For example one consisting of 2 triangles
with an extra edge attached and a standalone vertex may be created
using a list consisting of the arrays
\[
\begin{bmatrix}
  5
\end{bmatrix}\qquad
\begin{bmatrix}
  1 & 4
\end{bmatrix}\qquad
\begin{bmatrix}
  0 & 1 & 2\\ 1 & 2 & 3
\end{bmatrix}
\]
as input.

The boundary matrices of a simplicial complex encode the connectivity
information and can be computed from a purely combinatorial
description of a simplicial complex. The locations of the vertices are
not required. Thus the abstract simplicial complex structure is all
that is required to compute these matrices as will be described in the
next section.

\section{Discrete Exterior Derivative}
\label{sec:exterior_derivative}

Given a manifold $M$, the exterior derivative $\d : \Omega^p(M) \to
\Omega^{p+1}$ which acts on differential $p$-forms, generalizes the
derivative operator of calculus. When combined with metric dependent
operators Hodge star, sharp, and flat appropriately, the vector
calculus operators div, grad and curl can be generated from $\d$. But
$\d$ itself is a metric independent operators whose definition does
not require any Riemannian metric on the
manifold. See~\cite{AbMaRa1988} for details. The discrete exterior
derivative (which we will also denote as $\dd$) in PyDEC is defined as
is usual in the literature, as the coboundary operator of algebraic
topology~\cite{Munkres1984}. Thus
\begin{equation*}
 \eval{\dd_p a}{c} = \eval{a}{\boundary_{p+1} c}\, ,
\end{equation*}
for arbitrary $p$-cochain $a$ and $(p+1)$-chain $c$. Recall that the
boundary operator on cochains, $\boundary_p : C_p(K) \to C_{p-1}(K)$
is defined by extension of its definition on an oriented simplex.
The boundary operator on a $p$-simplex $\sigma^p = \simplex{p}$ is
given in terms of its $(p-1)$-dimensional faces ($p+1$ in number) as
\begin{equation}\label{eq:boundary}
  \boundary_p \sigma^p = \sum_{i=0}^p 
  (-1)^i \bigl[v_0,\ldots,\widehat{v_i},\ldots,v_p\bigr]\, ,
\end{equation}
where $\widehat{v_i}$ means that $v_i$ is omitted.  Therefore, given
an $n$-dimensional simplicial complex represented by
$\simplexarray_0,\dots,\simplexarray_n$, for the discrete exterior
derivative, it suffices to compute $\boundary_0, \dots, \boundary_n$.
As is usual in algebraic topology, in PyDEC we compute matrix
representations of these in the elementary chain basis. Boundary
matrices are useful in finite elements since their transposes are the
coboundary operators. They are also useful in computational topology
since homology and cohomology groups are the quotient groups of kernel
and image of boundary matrices~\cite{Munkres1984}.

For the complex pictured in Figure~\ref{fig:mesh_example_oriented}
the boundary operators are
\begin{align*}
\boundary_0 = \begin{bmatrix}
0 & 0 & 0 & 0 & 0
\end{bmatrix}\, ,
&&
\boundary_1 = \begin{bmatrix}
-1 & -1 &  \p0 &  \p0 &  \p0 &  
\p0 &  \p0\\
\p1  &  \p0 & -1 & -1 &  \p0 &
\p0 &  \p0\\
\p0  &  \p0 &  \p1 &  \p0 & 
-1 & -1 &  \p0\\
\p0  &  \p1 &  \p0 &  \p1 &  
\p1 &  \p0 & -1\\
\p0  &  \p0 &  \p0 &  \p0 &  
\p0 &  \p1 &  \p1
\end{bmatrix}\, ,
\\
\boundary_2 = \begin{bmatrix}
 \p1 &  \p0 &  \p0\\
-1 &  \p0 &  \p0\\
 \p0 &  \p1 &  \p0\\
 \p1 & -1 &  \p0\\
 \p0 &  \p1 & -1\\
 \p0 &  \p0 &  \p1\\
 \p0 &  \p0 & -1
\end{bmatrix}.
&&
\end{align*}
In the following we describe an algorithm that takes as input
$\simplexarray_n$ and computes both $\simplexarray_{n-1}$ and
$\boundary_{n-1}$.  This procedure is applied recursively to produce
all faces of the complex and their boundary operator at that
dimension.

The first step of the algorithm converts a simplex array
$\simplexarray_n$ into a canonical format. In the canonical format
each simplex (row of $\simplexarray_n$) is replaced by the simplex
with sorted indices and the relative parity of the original ordering
to the sorted order.  For instance, the simplex $(1,3,2)$ becomes
$((1,2,3), -1)$ since an odd number of transpositions (namely one) are
needed to transform $(1,3,2)$ into $(1,2,3)$.  Similarly, the
canonical format for simplex $(2,1,4,3)$ is $((1,2,3,4),+1)$ since an
even number of transpositions are required to sort the simplex
indices.  Since the complex dimension $n$ is typically small (i.e. $<
10$), a simple sorting algorithm such as insertion sort is employed at
this stage.  We denote the aforementioned process
\texttt{canonical\_format} $( \simplexarray_n ) \to \simplexarray^+_n$
where the rightmost column of $\simplexarray^+_n$ contains the simplex
parity.  Applying \texttt{canonical\_format} to $\simplexarray_2$ in
our example yields
\begin{equation*}
\simplexarray_2 = \begin{bmatrix}
0 & 1 & 3 \\
1 & 2 & 3 \\
2 & 4 & 3
\end{bmatrix} 
\to
\begin{bmatrix}
0 & 1 & 3 & &  \p1 \\
1 & 2 & 3 & &  \p1\\
2 & 3 & 4 & & -1
\end{bmatrix} 
= \simplexarray^+_2
\end{equation*}

Once a simplex array $\simplexarray_n$ has been transformed into
canonical format, the $(n-1)$-dimensional faces $\simplexarray_{n-1}$
and boundary operator $\boundary_n$ are readily obtained.  We denote
this process 
\[
\texttt{boundary\_faces}(\simplexarray^+_n) \to
\simplexarray_{n-1},\boundary_n \, .
\]
In order to establish the correspondence between the $n$-dimensional
simplices and their faces, we first enumerate the simplices by adding
another column to $\simplexarray^+_n$ to form $\simplexarray^{++}_n$.
For example,
\begin{equation*}
\simplexarray^+_2 = \begin{bmatrix}
0 & 1 & 3 & &  \p1 \\
1 & 2 & 3 & &  \p1\\
2 & 3 & 4 & & -1
\end{bmatrix} 
\to
\simplexarray^{++}_2 = \begin{bmatrix}
0 & 1 & 3 & &  \p1 & & 0 \\
1 & 2 & 3 & &  \p1 & & 1 \\
2 & 3 & 4 & & -1 & & 2 
\end{bmatrix}.
\end{equation*}
The formula~\eqref{eq:boundary} is applied to $\simplexarray^{++}_n$
in a columnwise fashion by excluding the $i$-th column of simplex
indices, multiplying the parity column by $(-1)^i$, and carrying the
last column over unchanged.  For example,
\begin{equation*}
\simplexarray^{++}_2 = \begin{bmatrix}
0 & 1 & 3 & &  \p1 & & 0 \\
1 & 2 & 3 & &  \p1 & & 1 \\
2 & 3 & 4 & & -1 & & 2
\end{bmatrix}
\to
\begin{bmatrix}
1 & 3 & &  \p1 & & 0 \\
2 & 3 & &  \p1 & & 1 \\
3 & 4 & & -1 & & 2 \\
0 & 3 & & -1 & & 0 \\
1 & 3 & & -1 & & 1 \\
2 & 4 & &  \p1 & & 2 \\
0 & 1 & &  \p1 & & 0 \\
1 & 2 & &  \p1 & & 1 \\
2 & 3 & & -1 & & 2
\end{bmatrix}
\end{equation*}
The resultant array is then sorted by the first $n$ columns in
lexicographical order, allowing the unique faces to then be extracted.
\begin{equation*}
\begin{bmatrix}
1 & 3 & &  \p1 & & 0 \\
2 & 3 & &  \p1 & & 1 \\
3 & 4 & & -1 & & 2 \\
0 & 3 & & -1 & & 0 \\
1 & 3 & & -1 & & 1 \\
2 & 4 & &  \p1 & & 2 \\
0 & 1 & &  \p1 & & 0 \\
1 & 2 & &  \p1 & & 1 \\
2 & 3 & & -1 & & 2
\end{bmatrix}
\to
\begin{bmatrix}
0 & 1 & &  \p1 & & 0 \\
0 & 3 & & -1 & & 0 \\
1 & 2 & &  \p1 & & 1 \\
1 & 3 & &  \p1 & & 0 \\
1 & 3 & & -1 & & 1 \\
2 & 3 & &  \p1 & & 1 \\
2 & 3 & & -1 & & 2 \\
2 & 4 & &  \p1 & & 2 \\
3 & 4 & & -1 & & 2 \\
\end{bmatrix}
\to
\begin{bmatrix}
0 & 1 \\
0 & 3 \\
1 & 2 \\
1 & 3 \\
2 & 3 \\
2 & 4 \\
3 & 4 \\
\end{bmatrix}.
\end{equation*}

Furthermore, a Compressed Sparse Row (CSR)~\cite{Saad2003} sparse
matrix representation of $\boundary_n$ as
\[
\boundary_n = (\texttt{ptr},\texttt{indices},\texttt{data})
\] 
is obtained from the sorted matrix.  For $\boundary_2$ these are
\begin{align*}
\texttt{ptr} &=
\begin{bmatrix}
0 & 1 & 2 & 3 & 5 & 7  & 8 & 9
\end{bmatrix}\, , \\
\texttt{indices} &=
\begin{bmatrix}
\p0 & \p0 & \p1 & \p0 & \p1 & \p1 & \p2 & \p2 & \p2  
\end{bmatrix}\, , \\
\texttt{data} &=
\begin{bmatrix}
\p1 & -1 & \p1 & \p1 & -1 & \p1 & -1 & \p1 & -1
\end{bmatrix}\,
\end{align*}
where \texttt{indices} and \texttt{data} correspond to the fourth and
third rows of the sorted matrix.

This process is then applied to $\simplexarray_{n-1}$ and so on down
the dimension. Since the lower dimensional simplex array rows are
already sorted, those arrays are already in canonical format. 
Thus a single algorithm generates the lower dimensional faces as well
as the boundary matrices at all the dimensions. The boundary matrices,
and hence the coboundary operators, are generated in a convenient
sparse matrix format.

\subsection{Generalization to Abstract Complexes}

The boundary operators and faces of an abstract simplicial complex are
computed with a straightforward extension of the
\texttt{boundary\_faces} algorithm.  Recall from
Section~\ref{sec:abstrct_smplcl_cmplx} that an abstract simplicial
complex is specified by a \emph{list} of simplex arrays of different
dimensions, where the lower-dimensional simplex arrays represent
simplices that are not a face of any higher-dimensional simplex.
Generalizing the previous scheme to the case of abstract simplicial
complexes is accomplished by (1) augmenting the set of computed faces
with the user-specified simplices and (2) modifying the computed
boundary operator accordingly.

Consider the abstract simplicial complex represented by the simplex
arrays
\[
\simplexarray_0 = [5] \qquad\qquad
\simplexarray_1 = [1, 4] \qquad\qquad
\simplexarray_2 = \begin{bmatrix} 0 & 1 & 2 \\ 1 & 2 &3 \end{bmatrix}
\]
which consists of two triangles, an edge, and an isolated vertex.
Applying \texttt{boundary\_faces} to $\simplexarray_{2}$ produces
an array of face edges and corresponding boundary operator
\[
\texttt{boundary\_faces}(\simplexarray_2) \to 
\simplexarray_{1},\boundary_2 \ = 
\begin{bmatrix}
 0 & 1 \\
 0 & 2 \\
 1 & 2 \\
 1 & 3 \\
 2 & 3 \\
\end{bmatrix}, 
\begin{bmatrix}
  \p1 & \p0 \\
   -1 & \p0 \\
  \p1 & \p1 \\
  \p0 &  -1 \\
  \p0 & \p1 \\
\end{bmatrix},
\]
which includes all but the user-specified edge $[1,4]$.
User-specified simplices are then incorporated into the simplex array
in a three-stage process:
\begin{inparaenum}[(1)]
\item user-specified simplices are concatenated to the computed face
  array;
\item the rows of the combined simplex array are sorted
  lexicographically; 
\item redundant simplices (if any) are removed from the sorted array.
\end{inparaenum}
Upon completion, the augmented simplex array contains the union of the
face simplices and the user-specified simplices.  Continuing the
example, the edge $[1,4]$ is incorporated into $\simplexarray_1$ as
follows
\[
\begin{bmatrix}
 0 & 1 \\
 0 & 2 \\
 1 & 2 \\
 1 & 3 \\
 2 & 3 \\
\end{bmatrix}, 
\begin{bmatrix}
 1 & 4 \\
\end{bmatrix}
\to
\begin{bmatrix}
 0 & 1 \\
 0 & 2 \\
 1 & 2 \\
 1 & 3 \\
 2 & 3 \\
 1 & 4 \\
\end{bmatrix}
\to
\begin{bmatrix}
 0 & 1 \\
 0 & 2 \\
 1 & 2 \\
 1 & 3 \\
 1 & 4 \\
 2 & 3 \\
\end{bmatrix}
\to
\begin{bmatrix}
 0 & 1 \\
 0 & 2 \\
 1 & 2 \\
 1 & 3 \\
 1 & 4 \\
 2 & 3 \\
\end{bmatrix} = \simplexarray_1
\]

In the final stage of the procedure, the computed boundary operator
($\boundary_2$ in the example) is updated to reflect the newly
incorporated simplices.  Since the new simplices do not lie in the
boundary of any higher-dimensional simplex, we may simply add empty
rows into the sparse matrix representation of the boundary operator
for each newly added simplex.  Therefore, the boundary operator update
procedure amounts to a simple remapping of row indices.  In the
example, the addition of the edge $[1,4]$ into the fifth row of the
simplex array requires the addition of an empty row into the boundary
operator at the corresponding position,
\[
\begin{bmatrix}
  \p1 & \p0 \\
   -1 & \p0 \\
  \p1 & \p1 \\
  \p0 &  -1 \\
  \p0 & \p1 \\
\end{bmatrix}
\to
\begin{bmatrix}
  \p1 & \p0 \\
   -1 & \p0 \\
  \p1 & \p1 \\
  \p0 &  -1 \\
  \p0 & \p0 \\
  \p0 & \p1 \\
\end{bmatrix} = \boundary_2.
\]

The Rips complex of Section~\ref{sec:rips_complex} does have the
location information for the vertices. However, ignoring those, such a
complex is an abstract simplicial complex. Thus the boundary matrices
for a Rips complex can be computed as described above. In practice
some efficiency can be obtained by ordering the computation
differently, so that the matrices are built as the complex is being
built from the edge skeleton of the Rips complex. That is how it is
implemented in PyDEC.
                                                                      
\subsection{Boundary operators and faces for cubical complexes}
\label{subsec:cbbndry}

The algorithm used to compute the faces and boundary operator of a
given cube array ($\cubearray_p \rightarrow \cubearray_{p-1},
\boundary_p$) is closely related to the procedure discussed in
Section~\ref{sec:exterior_derivative} for simplex arrays.  Consider a
general $p$-cube in $n$-dimensions, denoted by the pair
$\bigl[(c_0,\ldots, c_{n-1}) (d_0,\ldots,d_{p-1})\bigr]$
where $(c_0,\ldots, c_{n-1})$ are the coordinates of cube's origin and
$(d_0,\ldots,d_{p-1})$ are the directions which the $p$-cube spans.
Note that the values $[c_0, \ldots, c_{n-1}, d_0, \ldots, d_{p-1}]$
correspond exactly to a row of the cube array representation
introduced in Section~\ref{sec:cube_complex}.  Using this notation,
the boundary of a $p$-cube is given by the expression
\begin{multline}\label{eq:cube_boundary}
 \boundary_p \bigl[(c_0,\ldots, c_{n-1}) (d_0,\ldots,d_{p-1})\bigr] = 
  \sum_{i=0}^{p-1} (-1)^i \bigl(\bigl[(c_0, \ldots, c_{d_i} + 1,
  \ldots, 
  c_{n-1}) (d_0,\ldots,\widehat{d_i},\ldots,d_{p-1})\bigr] 
  - \\
  \bigl[(c_0, \ldots, c_{d_i} + 0, \ldots, c_{n-1}) 
  (d_0,\ldots,\widehat{d_i},\ldots,d_{p-1})\bigr]\bigr)\,
\end{multline}
where $\widehat{d_i}$ denotes the omission of the $i$-th spanning
direction and $c_{d_i}$ is the corresponding coordinate.  For example,
the boundary of a square centered at the location $(10,20)$ is
\begin{equation}
  \boundary_2 \bigl[(10, 20) (0, 1)\bigr] = 
   \bigl[(11,20) (1)\bigr] - \bigl[(10,20) (1)\bigr] - 
   \bigl[(10,21) (0)\bigr] + \bigl[(10,20) (0)\bigr].
\end{equation}

The canonical format for a $p$-cube is the one where the spanning
directions are specified in ascending order.  For instance, the
$2$-cube $\bigl[(10, 20) (0, 1)\bigr]$ is in the canonical format
because $d_0 < d_1$.  As with simplices, each cube has a unique
canonical format, through which duplicates are easily identified.
Since the top-level cube array $\cubearray_n$ is generated from a
bitmap it is already in the canonical format and no reordering of
indices or parity tracking is necessary.

Applying Equation~\ref{eq:cube_boundary} to a $p$-cube array with $N$
members generates a collection $2 N$ oriented faces.  In the mesh
illustrated in Figure~\ref{fig:cube_mesh} the three squares in
$\cubearray_2$ are initially expanded into
\[
\cubearray_2 = 
\begin{bmatrix}
  0 & 0 & & 0 & 1 \\
  1 & 0 & & 0 & 1 \\
  1 & 1 & & 0 & 1
\end{bmatrix}
\to
\begin{bmatrix}
0 & 0 & & 1 & &  -1 & & 0\\
1 & 0 & & 1 & &  -1 & & 1\\
1 & 1 & & 1 & &  -1 & & 2\\
1 & 0 & & 1 & & \p1 & & 0\\
2 & 0 & & 1 & & \p1 & & 1\\
2 & 1 & & 1 & & \p1 & & 2\\
0 & 0 & & 0 & & \p1 & & 0\\
1 & 0 & & 0 & & \p1 & & 1\\
1 & 1 & & 0 & & \p1 & & 2\\
0 & 1 & & 0 & &  -1 & & 0\\
1 & 1 & & 0 & &  -1 & & 1\\
1 & 2 & & 0 & &  -1 & & 2\\
\end{bmatrix} = \cubearray_1^+
\]
where the fourth column of $\cubearray_1^+$ encodes the orientation of
the face and the fifth column records the $2$-cube to which each face
belongs.  Sorting the rows of $\cubearray_1^+$ in lexicographical
order
\[
\begin{bmatrix}
0 & 0 & & 1 & &  -1 & & 0\\
1 & 0 & & 1 & &  -1 & & 1\\
1 & 1 & & 1 & &  -1 & & 2\\
1 & 0 & & 1 & & \p1 & & 0\\
2 & 0 & & 1 & & \p1 & & 1\\
2 & 1 & & 1 & & \p1 & & 2\\
0 & 0 & & 0 & & \p1 & & 0\\
1 & 0 & & 0 & & \p1 & & 1\\
1 & 1 & & 0 & & \p1 & & 2\\
0 & 1 & & 0 & &  -1 & & 0\\
1 & 1 & & 0 & &  -1 & & 1\\
1 & 2 & & 0 & &  -1 & & 2\\
\end{bmatrix} 
\to
\begin{bmatrix}
0 & 0 & & 0 & \p1 & & 0\\
0 & 0 & & 1 &  -1 & & 0\\
0 & 1 & & 0 &  -1 & & 0\\
1 & 0 & & 0 & \p1 & & 1\\
1 & 0 & & 1 &  -1 & & 1\\
1 & 0 & & 1 & \p1 & & 0\\
1 & 1 & & 0 & \p1 & & 2\\
1 & 1 & & 0 &  -1 & & 1\\
1 & 1 & & 1 &  -1 & & 2\\
1 & 2 & & 0 &  -1 & & 2\\
2 & 0 & & 1 & \p1 & & 1\\
2 & 1 & & 1 & \p1 & & 2\\
\end{bmatrix}
\to
\begin{bmatrix}
  0 & 0 & & 0 \\
  0 & 0 & & 1 \\
  0 & 1 & & 0 \\
  1 & 0 & & 0 \\
  1 & 0 & & 1 \\
  1 & 1 & & 0 \\
  1 & 1 & & 1 \\
  1 & 2 & & 0 \\
  2 & 0 & & 1 \\
  2 & 1 & & 1
\end{bmatrix}
=
\cubearray_1
\]
allows the unique faces to be extracted.  Lastly, a sparse matrix
representation of the boundary operator is obtained from the sorted
cube array in the same manner as for simplices.

\section{Review of Whitney Map and Whitney Forms}
\label{sec:whitney}

In this section we review and collect some material, most of which is
well-known in DEC and finite element exterior calculus research
communities. It is included here partly to fix notation. In this
section we also give the monomials based definition of inner product
of differential forms. This is not the way inner product of forms is
usually defined in most textbooks, \cite{Morita2001} being one
exception we know of. The monomial form leads to an efficient
algorithm for computation of stiffness and mass matrices for Whitney
forms given in Section~\ref{sec:whtny_innrprdct}.

The basic function spaces that are useful with exterior calculus are
the space of square integrable $p$-forms on a manifold and Sobolev
spaces derived from that. Let $M$ be a \emph{Riemannian manifold}, a
manifold on which a smoothly varying inner product is defined on the
tangent space at each point.  Let $g$ be its \emph{metric}, a smooth
tensor field that defines the inner product on the tangent space at
each point on $M$.

For differential forms on such a manifold $M$, the space of square
integrable forms is denoted $L^2 \Omega^p(M)$. One can then define the
spaces $H\Omega^p(M)$ which generalize the spaces $H(\div)$ and
$H(\curl)$ used in mixed finite element methods~\cite{ArFaWi2010}. To
define $L^2 \Omega^p(M)$ one has to define an inner product on the
space of forms which is our starting point for this section. All these
function spaces have been discussed in~\cite{ArFaWi2006}
and~\cite{ArFaWi2010}. The definitions and properties of Whitney map
and Whitney forms is in \cite{Dodziuk1976}, the geometric analysis
background is in~\cite{Jost2005} and the basic definition of inner
products on forms is in~\cite[page 411]{AbMaRa1988}.

\subsection{Inner product of forms}
To define the spaces $L^2\Omega^p(M)$ and $H\Omega^p(M)$ more
precisely, we recall the definitions related to inner products of
forms. We will need the exterior calculus operators wedge product and
Hodge star which we recall first. For a manifold $M$ the \emph{wedge
  product} $\wedge: \Omega^p(M) \times \Omega^q(M) \to
\Omega^{p+q}(M)$ is an operator for building $(p+q)$-forms from
$p$-forms and $q$-forms. It is defined as the skew-symmetrization of
the tensor product of the two forms involved. For a Riemannian
manifold of dimension $n$, the Hodge star operator $\hodge:\Omega^p(M)
\to \Omega^{n-p}(M)$ is an isomorphism between the spaces of $p$ and
$(n-p)$-forms. For more details, see \cite[page 394]{AbMaRa1988} for
wedge products and \cite[page 411]{AbMaRa1988} for Hodge star.

\begin{definition}
  \label{def:wdghdg_ip}
  Given two smooth $p$-forms $\alpha$, $\beta \in \Omega^p(M)$ on a
  Riemannian manifold $M$, their \emph{pointwise inner product} at
  point $x \in M$ is defined by
  \begin{equation} \label{eq:wdghdg_ip}
  \ainnerproduct{\alpha(x)}{\beta(x)}\,\mu = \alpha(x) \wedge \hodge
  \beta(x) \, ,
  \end{equation}
  where $\mu = \hodge 1$ is the volume form associated with the metric
  induced by the inner product on $M$. 
\end{definition}

The pointwise inner products of forms can be defined in another way,
which will be more useful to us in computations. The second definition
given below in Definition~\ref{def:monomials_ip} is equivalent to the
one given above in Definition~\ref{def:wdghdg_ip}. The operator
$\sharp$ (the \emph{sharp} operator) used below is an isomorphism
between 1-forms and vector fields and is defined by $g(\alpha^\sharp,
X) = \alpha(X)$ for given 1-form $\alpha$ and all vector fields
$X$. See~\cite{AbMaRa1988} for details.

\begin{definition}
\label{def:monomials_ip}
Let $\alpha_1,\dots,\alpha_p$ and $\beta_1,\dots,\beta_p$ be 1-forms
on a Riemannian manifold $M$. By analogy with polynomials we'll call
$p$-forms of the type $\alpha_1 \wedge \dots \wedge \alpha_p$ and
$\beta_1 \wedge \dots \wedge \beta_p$ \emph{monomial
  $p$-forms}. Define the following operator at a point $x \in M$:
\begin{equation} \label{eq:monomials_ip}
\aInnerproduct{\alpha_1 \wedge\dots\wedge\alpha_p}
{\beta_1\wedge\dots\wedge\beta_p} := \det\bigl[
  g\bigl(\alpha_i^\sharp,\beta_j^\sharp\bigr)\bigr]\, ,
\end{equation}
where $\bigl[g\bigl(\alpha_i^\sharp,\beta_j^\sharp\bigr)\bigr]$ is the
matrix obtained by taking $1 \le i,j \le p$.  In the equation above,
all the 1-forms are evaluated at the point $x \in M$. Extend the
operation in \eqref{eq:monomials_ip} bilinearly pointwise to the space
of all $p$-forms. It can be shown that this defines a \emph{pointwise
  inner product} of $p$-forms equivalent to the one defined in
\eqref{eq:wdghdg_ip}. Note that if $\alpha_i = \beta_i$ for all $i$,
the expression on the right in~\eqref{eq:monomials_ip} is the Gram
determinant.
\end{definition}

\begin{remark}\label{rem:whch_ip_dfntn}
  Note that unlike~\eqref{eq:wdghdg_ip} the definition
  in~\eqref{eq:monomials_ip} does not involve wedge product and Hodge
  star explicitly. This is an advantage of the latter form since a
  discrete wedge product is not available in PyDEC.  The RHS
  of~\eqref{eq:monomials_ip} does involve the sharp operator, but as
  we will see in the next section, this is easy to interpret for the
  purpose of discretization in this context.
\end{remark}

\begin{definition} \label{def:L2innrprdctfrms} The pointwise
  innerproduct in~\eqref{eq:wdghdg_ip}, or equivalently
  in~\eqref{eq:monomials_ip}, induces an $L^2$ \emph{inner product on
    $M$} as
  \begin{equation} \label{eq:forms_ip}
  \pinnerproduct{\alpha}{\beta}_{L^2} = 
  \int_M \ainnerproduct{\alpha(x)}{\beta(x)}\,\mu \, .
  \end{equation}
  The space of $p$-forms obtained by completion of $\Omega^p(M)$ under
  this inner product is the Hilbert space of \emph{square integrable
    $p$-forms} $L^2\Omega^p(M)$.  The other useful space mentioned at
  the beginning of this section is the Sobolev space $H \Omega^p(M) :=
  \{\alpha \in L^2 \Omega^p(M) \;|\; \d\alpha \in L^2
  \Omega^{p+1}(M)\}$.
\end{definition}

\subsection{Whitney map and Whitney forms}
Let $K$ be an $n$-dimensional manifold simplicial complex embedded in
$\R^N$ and $\abs{K}$ the underlying space.  The metric on the
interiors of the top dimensional simplices of $K$ will be the one
induced from the embedding Euclidean space $\R^N$. As is usual in
finite element methods, finite dimensional subspaces of the function
spaces described in the previous paragraph are used in the numerical
solution of PDEs. The finite dimensional spaces can be obtained by
``embedding'' the space of cochains into these spaces by using an
interpolation. For example, to embed $C^p(K;\R)$ into $\Lforms2p$ one
can use the Whitney map $\whitney:C^p(K;\R)\to \Lforms2p$, which will
be reviewed in this subsection. The image
$\whitney\bigl(C^p(K;\R)\bigr)$ is a linear vector subspace of
$\Lforms2p$ and is the space of \emph{Whitney $p$-forms}
\cite{Whitney1957,Dodziuk1976,Bossavit1988a} and is denoted
$\mathcal{P}_1^- \Omega^p(\abs{K})$ in~\cite{ArFaWi2009,
  ArFaWi2010}. (We use $\Omega^p$ instead of $\Lambda^p$ used
in~\cite{ArFaWi2010}.) The embedding of cochains is analogous to how
scalar values at discrete sample points would be interpolated to get a
piecewise affine function. In the scalar case also, the space of such
functions is a vector subspace of square integrable functions.  In
fact, $\whitney\bigl(C^0(K;\R)\bigr)$, the space of Whitney $0$-forms
is the space of continuous piecewise affine functions on
$\abs{K}$. The Whitney map for $p>0$ is actually built from
barycentric coordinates which are the building blocks of piecewise
linear interpolation.  Thus the embedding of $C^p(K;\R)$ into
$\Lforms2p$ for $p>0$ can be considered to be a generalization of the
embedding of $C^0(K;\R)$ into $\Lforms20$. Thus Whitney forms enable
only low order methods. However, arbitrary degree polynomial spaces
suitable for use in finite element exterior calculus have been
discovered~\cite{ArFaWi2006, ArFaWi2009}. These however are not yet a
part of PyDEC.

The space of Whitney $p$-forms is the space of piecewise smooth
differential $p$-forms obtained by applying the Whitney map to
$p$-cochains. It can be thought of as a method for interpolating
values given on $p$-simplices of a simplicial complex. For example,
inside a tetrahedron Whitney forms allow the interpolation of numbers
on edges or faces to a smooth 1-form or 2-form respectively.  As
mentioned above, for 0-cochains, i.e. scalar functions sampled at
vertices, the interpolation is the one obtained using the standard
scalar piecewise affine basis functions on each simplex, that is the
barycentric coordinates corresponding to each vertex of the
simplex. We recall the definition of barycentric coordinates followed
by the definition of the Whitney map.

\begin{definition}\label{def:barycentric_coords}
  Let $\sigma^p = \simplex{p}$ be a $p$-simplex embedded in
  $\R^N$. The affine functions $\mu_i:\R^N \to \R$, $i=0,\ldots,p$,
  which when restricted to $\sigma^p$ take the value 1 on vertex $v_i$
  and 0 on the others, are called the \emph{barycentric coordinates}
  in $\sigma^p$.
\end{definition}

\begin{definition}
  Let $\sigma^p$ be an oriented $p$-simplex $[v_{i_0},\ldots,v_{i_p}]$
  in an $n$-dimensional manifold complex $K$, and
  $\bigl(\sigma^p\bigr)^\ast$ the corresponding elementary
  $p$-cochain.  We define
  \begin{equation} \label{eq:whitney_map}
  \whitney\left(\left(\sigma^p\right)^{\ast}\right) :=
  p! \sum_{k=0}^p (-1)^k \mu_{i_k}\,d\mu_{i_0}\wedge\cdots
  \wedge\widehat{d\mu_{i_k}}\wedge\cdots\wedge d\mu_{i_p}\, ,
  \end{equation}
  where $\mu_{i_{k}}$ is the barycentric coordinate function with
  respect to vertex $v_{i_{k}}$ and the notation
  $\widehat{d\mu_{i_{k}}}$ indicates that the term $d\mu_{i_{k}}$ is
  omitted from the wedge product.  The \emph{Whitney map}
  $\whitney:C^p(K;\R)\to\Lforms2p$ is the above map $\whitney$
  extended to all of $C^p(K;\R)$ by requiring that $\whitney$ be a
  linear map.  $\whitney(\left(\sigma^p\right)^{\ast})$ is called the
  \emph{Whitney form} corresponding to $\sigma^p$, and for a general
  cochain $c$, $\whitney(c)$ is called the Whitney form
  corresponding to $c$.
\end{definition}

For example, the Whitney form corresponding to the edge $[v_0,v_1]$ is
$\whitney([v_0,v_1]^\ast) = \mu_0 \d \mu_1 - \mu_1 \d \mu_0 \, , $ and
the Whitney form corresponding to the triangle $[v_1, v_2, v_3]$ in a
tetrahedron $[v_0, v_1, v_2, v_3]$ is 
\[
\whitney([v_1, v_2, v_3]^\ast)
= 2\,(\mu_1 \d \mu_2 \wedge \d \mu_3 - \mu_2 \d \mu_1 \wedge \d \mu_3
+ \mu_3 \d \mu_1 \wedge \d \mu_2)\, .
\]

\begin{remark}\label{rem:mnml_cmbntns}
If we were using local coordinate charts on a manifold then at any
point a $p$-form would be a linear combination of monomials.  Note
from \eqref{eq:whitney_map} that the Whitney form
$\whitney(\sigma^\ast)$ is a sum of monomials with
coefficients. Thus Whitney forms allow us to treat forms at a point as
a linear combination of monomials even though we are not using local
coordinate charts.
\end{remark}

We emphasize again that this section was a review of known
material. We have tried to present this material in a manner which
makes it easier to explain the examples of Section~\ref{sec:examples}
and the construction of mass matrix for Whitney forms described in the
next section.

\section{Whitney Inner Product of Cochains}
\label{sec:whtny_innrprdct}

Given a manifold simplicial complex $K$, an inner product between two
$p$-cochains $a$ and $b$ can be defined by first embedding these
cochains into $\Lforms2p$ using Whitney map and then taking the $L^2$
inner product of the resulting Whitney forms \cite{Dodziuk1976}. 

\begin{definition}\label{def:whtnyip}
Given two $p$-cochains $a,b \in C^p(K;\R)$, their \emph{Whitney inner
  product} is defined by 
\begin{equation}\label{eq:cochains_ip}
  \pinnerproduct{a}{b} := \pinnerproduct{\whitney a}{\whitney b}_{L^2}
  =  \int_{\abs{K}} \ainnerproduct{\whitney a}{\whitney b} \mu\, ,
\end{equation}
using the $L^2$ inner product on forms given
in~\eqref{eq:forms_ip}. The matrix for Whitney inner product of
$p$-forms in the elementary $p$-cochain basis will be denoted $M_p$.
That is, $M_p$  is a square matrix of order $N_p$ (the number of
$p$-simplices in $K$) such that the entry in row $i$ and column $j$ is 
%\begin{equation} \label{eq:Mpij}
$M_p(i,j) = \pInnerproduct{(\sigma_i^p)^\ast}{(\sigma_j^p)^\ast}$,
%\end{equation}
where $(\sigma_i^p)^\ast$ and $(\sigma_j^p)^\ast$ are the elementary
$p$-cochains corresponding to the $p$-simplices $\sigma_i^p$ and
$\sigma_j^p$ with index number $i$ and $j$ respectively. 
\end{definition}
The integral in~\eqref{eq:cochains_ip} is the sum of integrals over
each top dimensional simplex in $K$. Inside each such simplex the
inner product of smooth forms applies since the Whitney form in each
simplex is smooth all the way up to the boundary of the simplex. The
interior of each top dimensional simplex is given an inner product
that is induced from the standard inner product of the embedding space
$\R^N$.

\begin{remark}
Given cochains $a,b\in C^p(K;\R)$ we will refer to their
representations in the elementary cochain basis also as $a$ and
$b$. Then the matrix representation of the Whitney inner product of
$a$ and $b$ is $\pinnerproduct{a}{b} = a^T M_p b$.
\end{remark}

The inner product of cochains defined in this way is a key concept
that connects exterior calculus to finite element methods and
different choices of the inner product lead to different
discretizations of exterior calculus. This is because the inner
product matrix $M_p$ is the mass matrix of finite element methods
based on Whitney forms. The details of the efficient computation of
$M_p$ for any $p$ and $n$ will be given in Section~\ref{subsec:mssmtrx}
and~\ref{subsec:mssmtrx_algrthm}.

Recall that for a Riemannian manifold $M$, if $\codiff_{p+1} :
\Omega^{p+1}(M) \to \Omega^p(M)$ is the codifferential, then the
Laplace-deRham operator on $p$-forms is $\laplacian_p :=
\d_{p-1}\codiff_p + \codiff_{p+1} \d_p$. For a boundaryless $M$ the
codifferential $\codiff_{p+1}$ is the adjoint of the exterior
derivative $d_p$. In case $M$ has a boundary, we have instead that
\begin{equation} \label{eq:adjoint}
\pinnerproduct{\d_p\alpha}{\beta} =
\pinnerproduct{a}{\codiff_{p+1}\beta} + \int_{\partial
  M}\alpha\wedge\hodge\beta \, .
\end{equation}
See \cite[Exercise 7.5E]{AbMaRa1988} for a derivation of the
above. Now consider Poisson's equation $\laplacian_p u = f$ on
$p$-forms defined on a $p$-dimensional simplicial manifold complex
$K$. For simplicity, we'll consider the weak form of this using smooth
forms rather than Sobolev spaces of forms. See~\cite{ArFaWi2006,
  ArFaWi2010} for a proper functional analytic treatment. We will also
assume that the correct boundary conditions are satisfied, so that the
boundary term in~\eqref{eq:adjoint} is 0. In our simple treatment, the
weak form of the Poisson's equation is to find a $u \in
\Omega^p(\abs{K})$ such that $\pinnerproduct{\codiff_p u}{\codiff_p
  v}_{L^2} + \pinnerproduct{\d_p u}{\d_p v}_{L^2} =
\pinnerproduct{f}{v}_{L^2}$.  Thus it is clear that a Galerkin
formulation using Whitney forms $\mathcal{P}_1^- \Omega^p$ will
require the computation of a term like $\pinnerproduct{\d_p \whitney
  a}{\d_p \whitney b}_{L^2}$ for cochains $a$ and $b$. By the
commuting property of Whitney forms $\d_p W = W \d_p$ (where the
second $\d_p$ is the coboundary operator on cochains) we have that the
above inner product is equal to $\pinnerproduct{\whitney \d_p
  a}{\whitney \d_p b}_{L^2}$. (See~\cite{Dodziuk1976} for a proof of
the commuting property.) Now, by definition of the Whitney inner
product of cochains in~\eqref{eq:cochains_ip} this is equal to
$\pinnerproduct{\dd_p a}{\dd_p b}$ in the inner product on
$(p+1)$-cochains. The matrix form of this inner product can be
obtained from the mass matrix $M_{p+1}$ as $\d_p^T\, M_{p+1}\,
\d_p$. This is what we mean when we say that the stiffness matrix can
be computed easily from the mass matrix. The term on the right in the
weak form will use the mass matrix $M_p$. Since codifferential of
Whitney forms is 0, the first term in the weak form has to be handled
in another way, as described in~\cite{ArFaWi2010}.

\begin{remark}\label{rem:stffnss}
  Exploiting the aforementioned commutativity of Whitney forms to
  compute the stiffness matrix represents a significant simplification
  to our software implementation.  While computing the stiffness
  matrix directly from the definition is possible, it is a complex
  operation which requires considerable programmer effort, especially
  if the performance of the implementation is important.  In contrast,
  our formulation requires no additional effort and has the
  performance of the underlying sparse matrix-matrix multiplication
  implementation, an optimized and natively-compiled routine.  All of
  the complex indexing, considerations of relative orientation,
  mappings between faces and indices, etc. is reduced to a simple
  linear algebra expression. The lower dimensional faces are oriented
  lexicographically and the orientation information required in
  stiffness matrix assembly is implicit in the boundary matrices.
\end{remark}

\subsection{Computing barycentric differentials}
Given that the Whitney form $\whitney(\sigma^\ast)$ in
\eqref{eq:whitney_map} is built using wedges of differentials of
barycentric coordinates, it is clear that the algorithm for computing
an inner product of Whitney forms involves computation of the
gradients or differentials of the barycentric coordinates. The
following lemma shows how these are computed using simple linear
algebra operations.

\begin{lemma}\label{lem:barycentric_grads}
  Let $\sigma^p = \simplex{p}$ be a $p$-simplex embedded in $\R^N$, $p
  \le N$ where the vertices $v_i \in \R^N$ are given in some basis for
  $\R^N$. Let $X\in \R^{N \times p}$ be a matrix whose $j$-th column
  consists of the components of $\d \mu_j$ in the dual basis, for $j =
  1,\ldots,p$ . Let $V_0\in \R^{N \times p}$ be a matrix whose $j$-th
  column is $v_j - v_0$, for $j = 1,\ldots,p$. Then $X^T =
  \bigl(V_0^TV_0\bigr)^{-1}V_0^T = V_0^+$, the pseudoinverse of $V_0$.
\end{lemma}

% \begin{proof}[First proof]
%   Note that $(\d \mu_i)(v_i - v_0) = 1$ for all $i = 1,\ldots,p$. This
%   is because along the vector $v_i - v_0$ the function $\mu_i$ changes
%   from 0 at vertex $v_0$ to 1 at vertex $v_i$. Similarly,
%   $(\d\mu_i)(v_j - v_0) = 0$ when $i\ne j$ for all $i,j=1,\ldots,p$
%   since for such $i$ and $j$, $\mu_i$ is 0 at vertices $v_0$ and
%   $v_j$. This yields the matrix equation $V_0^T X = I$ where
%   $I\in\R^{p\times p}$ is the identity matrix. We are interested in
%   obtaining the components of all the $\d \mu_i$ in
%   $\operatorname{span}(V_0)$, i.e., we'd like $X=V_0Y$ where $Y\in
%   \R^{p\times p}$ is a matrix that has to be found. Thus we have that
%   $V_0^TV_0Y=I$. Since the simplices are assumed to be non-degenerate,
%   the columns of $V_0$ are linearly independent, which implies that
%   $V_0^TV_0$ is invertible. Thus $Y = (V_0^TV_0)^{-1}$, which implies
%   that $X=V_0Y=V_0(V_0^TV_0)^{-1}$. But this implies that $X^T=
%   (V_0^TV_0)^{-1}V_0^T = V_0^+$.
% \end{proof}

\begin{proof}
  Let $\zeta = [\mu_1,\dots,\mu_p]^T$ be the vector of barycentric
  coordinates (other than $\mu_0$) with respect to $\sigma^p$, for a
  point $x = [x_1,\dots,x_N]^T \in \R^N$. Then by definition of
  barycentric coordinates and simplices, $V_0 \zeta = x - v_0$ is the
  linear least squares system for the barycentric coordinates. Thus
  $\zeta = V_0^+ (x -v_0)$ which implies that $\d \zeta = X^T =
  V_0^+$.
\end{proof}

\begin{remark}
  The use of normal equations in the solution of the least squares
  problem in the above proposition suffers from the well-known
  condition squaring problem. This is only likely to be a problem if
  the simplices are nearly degenerate. In that case one can just use
  an orthogonalization method to compute a QR factorization and use
  that to solve the least squares problem. Notice that in typical
  physical problems $V_0$ will typically be $2\times 2$, $3\times 2$
  or $3 \times 3$ matrix so any of these methods are easy to
  implement.
\end{remark}

Once the components for $\d \mu_i$ have been obtained for $i =
1,\dots,p$, the components of $\d \mu_0$ can be obtained by noting
that $\d \mu_0 + \cdots + \d \mu_p = 0$ which follows from the fact
that the barycentric coordinates sum to 1. Also note that the
components of the gradients $\nabla \mu_i$ will be the same as those
of $\d \mu_i$ if the standard metric of Euclidean space is used for
the embedding space $\R^N$ which is the case in all of PyDEC.

\subsection{Whitney inner product matrix} 
\label{subsec:mssmtrx}
We will now use the inner product of forms in~\eqref{eq:monomials_ip}
and the cochains inner product defined in~\eqref{eq:cochains_ip} to
give a formula for the computation of $M_p$, the Whitney inner product
matrix for $p$-cochains. We will also refer to this as the Whitney
mass matrix. 

\begin{notation}
  Given simplices $\sigma$ and $\tau$ the notation $\sigma \hasface
  \tau$ or $\tau \faceof \sigma$ means $\tau$ is a face of
  $\sigma$. Note that this means $\tau$ can be equal to $\sigma$ since
  any simplex is its own face. For proper inclusion we
  use $\tau \prprfaceof \sigma$ or $\sigma \hasprprface \tau$ to
  indicate that $\tau$ is a proper face of $\sigma$. The use of this
  notation simplifies the expression of summations over various classes
  of simplices in a complex. For example, given two fixed
  $p$-simplices $\sigma_i^p$ and $\sigma_j^p$
  \[
  \underset{\sigma^n \hasface \sigma^p_i, \sigma^p_j}
  {\sum_{\sigma^n}} 
  \]
  is read as ``sum over all $n$-simplices $\sigma^n$ which have
  $\sigma_i^p$ and $\sigma_j^p$ as faces'' .  Another notation used in
  the proof of the proposition below is the \emph{star} of a simplex
  $\sigma$, written $\St(\sigma)$ (not to be confused with the dual
  star $\star \sigma$). This star $\St(\sigma)$ is the union of the
  interiors of all simplices of the complex that have $\sigma$ as a
  face. That includes $\sigma$ also. The closure of this open set is
  called the \emph{closed star} and written $\ClSt{\sigma}$. This is
  the union of simplices that contain~$\sigma$.
\end{notation}

\begin{proposition} \label{prop:whtny_ip} Let $\sigma_i^p =
  [v_{i_0},\dots,v_{i_p}]$ and $\sigma_j^p = [v_{j_0},\dots,v_{j_p}]$
  be oriented $p$-simplices in an $n$-dimensional manifold simplicial
  complex $K$, with $0 \le p \le n$. Then the row $i$, column
  $j$ entry of the Whitney inner product matrix $M_p$ is given by
  \[
  M_p(i,j) = (p!)^2 \underset{\sigma^n \hasface \sigma^p_i, \sigma^p_j}
  {\sum_{\sigma^n}} 
  \sum_{k,l=0}^p(-1)^{k+l} c_{kl}\int_{\sigma^n}\mu_{i_k}\mu_{j_l} \, \mu\, ,
  \]
  where $c_{kl} = 1$ for $p=0$, and for $p>0$
  \[
  c_{kl} = \det
  \begin{bmatrix}
    \aInnerproduct{\d \mu_{i_0}}{\d \mu_{j_0}} & \dots & 
    \widehat{\aInnerproduct{\d \mu_{i_0}}{\d \mu_{j_l}}} &
    \dots & \aInnerproduct{\d \mu_{i_0}}{\d \mu_{j_p}}\\
    \vdots & & \vdots & & \vdots\\
    \widehat{\aInnerproduct{\d \mu_{i_k}}{\d \mu_{j_0}}} & 
    \dots & 
    \widehat{\aInnerproduct{\d \mu_{i_k}}{\d \mu_{j_l}}} & 
    \dots &
    \widehat{\aInnerproduct{\d \mu_{i_k}}{\d \mu_{j_l}}}\\
    \vdots & & \vdots & & \vdots\\
    \aInnerproduct{\d \mu_{i_p}}{\d \mu_{j_0}} & \dots &
    \widehat{\aInnerproduct{\d \mu_{i_p}}{\d \mu_{j_l}}} &
    \dots & \aInnerproduct{\d \mu_{i_p}}{\d \mu_{j_p}}
  \end{bmatrix}\, ,
  \]
  the hats indicating the deleted terms. Here $\mu$ is the volume form
  corresponding to the standard inner product in $\R^N$ and
  $\mu_{i_k}$ and $\mu_{j_l}$ are the barycentric coordinates
  corresponding to vertices $i_k$ and $j_l$.
\end{proposition}

\begin{proof} See Appendix. \end{proof} 

For the $p=n$ case a simpler formulation is given in
Proposition~\ref{prop:p_equal_n}.  The above proposition shows that
computation of the Whitney inner product matrix involves computations
of inner products of differentials of barycentric coordinates. Since
the only metric implemented in PyDEC is the standard one inherited
from the embedding space $\R^N$,
\[
\aInnerproduct{\d \mu_i}{\d \mu_j} = g((\d \mu_i)^\sharp,(\d
\mu_j)^\sharp) = \nabla\mu_i \cdot \nabla\mu_j\, .
\] 

\begin{example} \label{ex:ttrhdrn_whtny} Consider the simplicial
  complex corresponding to a tetrahedron $\sigma^3$ embedded in $\R^3$
  for which we want to compute $M_2$, the Whitney inner product matrix
  for 2-cochains. Here $N=n=3$ and $p=2$ and $M_2$ is of order $N_2 =
  4$, the number of triangles in the complex. Label the vertices as
  $0, 1, 2, 3$. Then by PyDEC's lexicographic numbering scheme, the
  edges numbered 0 to 5 are $[0,1]$, $[0,2]$, $[0,3]$, $[1,2]$,
  $[1,3]$, $[2,3]$ and the triangles numbered 0 to 3 are $[0,1,2]$,
  $[0,1,3]$, $[0,2,3]$, and $[1,2,3]$.  We will describe the
  computation of the row 0, column 3 entry and the row 0, column 1
  entry of $M_2$. The $(0,3)$ entry corresponds to the inner product
  of cochains $[v_0,v_1,v_2]^\ast$ and $[v_1,v_2,v_3]^\ast$ since, in
  the lexicographic ordering and naming convention of PyDEC these are
  $\sigma^2_0$ and $\sigma^2_3$ respectively. Thus we are computing
  \[
  \pInnerproduct{\cochainBasis 2 0}{\cochainBasis 2 3} = 
  \pInnerproduct{\whitney \cochainBasis 2 0}
  {\whitney \cochainBasis 2 3}_{L^2}\, .
  \]
  The corresponding Whitney forms are
  \begin{align*}
    \whitney \cochainBasis{2}{0} &= 2! \, 
    \bigl(\mu_0\, \d \mu_1 \wedge  \d \mu_2 - 
    \mu_1 \d \mu_0 \wedge \d \mu_2 + 
    \mu_2 \d \mu_0 \wedge \d \mu_1\bigr)\\
    \whitney \cochainBasis{2}{3} &= 2! \, 
    \bigl(\mu_1\, \d \mu_2 \wedge  \d \mu_3 - 
    \mu_2 \d \mu_1 \wedge \d \mu_3 + 
    \mu_3 \d \mu_1 \wedge \d \mu_2\bigr)\, .
  \end{align*}
  Then $\pInnerproduct {\whitney \cochainBasis 2 0}{\whitney
    \cochainBasis 2 3}_{L^2}/(2!)^2$ is
  %\begin{multline}\label{expr:ip03full}
  \begin{equation}\label{expr:ip03full}
  \int_{\sigma^3} \mu_0\mu_1
  \aInnerproduct{\d \mu_1 \wedge \d \mu_2}
  {\d\mu_2\wedge \d\mu_3}\, \mu -\\
  \int_{\sigma^3} \mu_0\mu_2
  \aInnerproduct{\d \mu_1 \wedge \d \mu_2}
  {\d\mu_1\wedge \d\mu_3}\, \mu + 
  \dots\, ,
  \end{equation}
  %\end{multline}
  where $\mu$ is just $\d x\, \wedge\, \d y\, \wedge\, \d z$, the
  standard volume form in $\R^3$.  Each term like 
  $\aInnerproduct{\d\mu_1 \wedge \d \mu_2} {\d\mu_2\wedge \d\mu_3}$
  is
  \[
  \det
  \begin{bmatrix}
    \ainnerproduct{\d\mu_1}{\d \mu_2} &
    \ainnerproduct{\d\mu_1}{\d \mu_3} \\
    \ainnerproduct{\d\mu_2}{\d \mu_2} &
    \ainnerproduct{\d\mu_2}{\d \mu_3}
  \end{bmatrix}\, ,
  \]
  which in the notation of Prop.~\ref{prop:whtny_ip} is 
  \[
  \det
  \begin{bmatrix}
    \widehat{\ainnerproduct{\d\mu_0}{\d\mu_1}} &
    \widehat{\ainnerproduct{\d\mu_0}{\d\mu_2}} &
    \widehat{\ainnerproduct{\d\mu_0}{\d\mu_3}} \\
    \widehat{\ainnerproduct{\d\mu_1}{\d\mu_1}} &
    \ainnerproduct{\d\mu_1}{\d \mu_2} &
    \ainnerproduct{\d\mu_1}{\d \mu_3} \\
    \widehat{\ainnerproduct{\d\mu_2}{\d\mu_1}} &
    \ainnerproduct{\d\mu_2}{\d \mu_2} &
    \ainnerproduct{\d\mu_2}{\d \mu_3}        
  \end{bmatrix}\, .
  \]
  Using a shorthand notation for matrices like above, the $2\times 2$
  matrices whose determinants need to be computed for calculating the
  $(0,3)$ entry of $M_2$ are given below.
  \[
  \begin{array}{ccc}
  \begin{bmatrix}
    \widehat{01} & \widehat{02} & \widehat{03}\\
    \widehat{11} & 12 & 13\\
    \widehat{21} & 22 & 23
  \end{bmatrix} &
   \begin{bmatrix}
     \widehat{02} & \widehat{01} & \widehat{03}\\
     \widehat{12} & 11 & 13\\
     \widehat{22} & 21 & 23
   \end{bmatrix} &
 \begin{bmatrix}
     \widehat{03} & \widehat{01} & \widehat{02}\\
     \widehat{13} & 11 & 12\\
     \widehat{23} & 21 & 22
   \end{bmatrix}\\
   & & \\
  \begin{bmatrix}
    \widehat{11} & \widehat{12} & \widehat{13}\\
    \widehat{01} & 02 & 03\\
    \widehat{21} & 22 & 23
  \end{bmatrix} &
   \begin{bmatrix}
     \widehat{12} & \widehat{11} & \widehat{13}\\
     \widehat{02} & 01 & 03\\
     \widehat{22} & 21 & 23
   \end{bmatrix} &
 \begin{bmatrix}
     \widehat{13} & \widehat{11} & \widehat{12}\\
     \widehat{03} & 01 & 02\\
     \widehat{23} & 21 & 22
   \end{bmatrix}\\
   & & \\
  \begin{bmatrix}
    \widehat{21} & \widehat{22} & \widehat{23}\\
    \widehat{01} & 02 & 03\\
    \widehat{11} & 12 & 13
  \end{bmatrix} &
   \begin{bmatrix}
     \widehat{22} & \widehat{21} & \widehat{23}\\
     \widehat{02} & 01 & 03\\
     \widehat{12} & 11 & 13
   \end{bmatrix} &
 \begin{bmatrix}
     \widehat{23} & \widehat{21} & \widehat{22}\\
     \widehat{03} & 01 & 02\\
     \widehat{13} & 11 & 12
   \end{bmatrix}
  \end{array}
  \]
Removing the deleted rows and columns the above matrices are given
below as the actual $2\times 2$ matrices.
  \begin{equation} \label{expr:ip03}
  \begin{array}{ccc}
  \begin{bmatrix}
    12 & 13\\
    22 & 23
  \end{bmatrix} &
   \begin{bmatrix}
    11 & 13\\
    21 & 23
   \end{bmatrix} &
 \begin{bmatrix}
    11 & 12\\
    21 & 22
   \end{bmatrix}\\
   & & \\
  \begin{bmatrix}
    02 & 03\\
    22 & 23
  \end{bmatrix} &
   \begin{bmatrix}
    01 & 03\\
    21 & 23
   \end{bmatrix} &
 \begin{bmatrix}
    01 & 02\\
    21 & 22
   \end{bmatrix}\\
   & & \\
  \begin{bmatrix}
    02 & 03\\
    12 & 13
  \end{bmatrix} &
   \begin{bmatrix}
    01 & 03\\
    11 & 13
   \end{bmatrix} &
 \begin{bmatrix}
    01 & 02\\
    11 & 12
   \end{bmatrix}
  \end{array}
  \end{equation}
  The $2\times 2$ matrices whose determinants are needed in computing
  $\pInnerproduct{\cochainBasis 2 0}{\cochainBasis 2 1}$, i.e., entry
  $(0,1)$ of $M_2$ are given below.
  \begin{equation}\label{expr:ip01}
  \begin{array}{ccc}
  \begin{bmatrix}
    11 & 13\\
    21 & 23
  \end{bmatrix} &
   \begin{bmatrix}
    10 & 13\\
    20 & 23
   \end{bmatrix} &
 \begin{bmatrix}
    10 & 11\\
    20 & 21
   \end{bmatrix}\\
   & & \\
  \begin{bmatrix}
    01 & 03\\
    21 & 23
  \end{bmatrix} &
   \begin{bmatrix}
    00 & 03\\
    20 & 23
   \end{bmatrix} &
 \begin{bmatrix}
    00 & 01\\
    20 & 21
   \end{bmatrix}\\
   & & \\
  \begin{bmatrix}
    01 & 03\\
    11 & 13
  \end{bmatrix} &
   \begin{bmatrix}
    00 & 03\\
    10 & 13
   \end{bmatrix} &
 \begin{bmatrix}
    00 & 01\\
    10 & 11
   \end{bmatrix}
  \end{array}
  \end{equation}
\end{example}
Recall that each number in these matrices is shorthand for an inner
product of two barycentric differentials. For example, the entry 12
stands for $\aInnerproduct{\d \mu_1}{\d \mu_2} = g((\d
\mu_1)^\sharp,(\d \mu_2)^\sharp) = \nabla\mu_1 \cdot \nabla\mu_2$.

\begin{proposition} \label{prop:p_equal_n}
  For $p=n$, $M_p$ is a diagonal matrix with $M_p(i,i) =
  1/\abs{\sigma_i^n}$, where $\abs{\sigma_i^n}$ is the volume of the
  simplex.
\end{proposition}
\begin{proof}
  For any $n$-simplex $\sigma^n$, the Whitney form $\whitney
  (\sigma^n)^\ast$ is 0 on other $n$-simplices and so $M_p$ is
  diagonal. Furthermore, it is a constant coefficient volume form on
  $\sigma^n$ with $\int_{\sigma^n} \whitney (\sigma^n)^\ast =
  1$. See~\cite{Dodziuk1976} for proofs of these properties. Thus it
  must be that $\whitney (\sigma^n)^\ast = \mu / \abs{\sigma^n}$ where
  $\mu$ is the volume form on the simplex. Thus 
  \[
  \aInnerproduct{\whitney (\sigma^n)^\ast}{\whitney (\sigma^n)^\ast}
  \, \mu = \whitney (\sigma^n)^\ast \wedge \hodge \whitney
  (\sigma^n)^\ast = \frac{\mu}{\abs{\sigma^n}^2} \, .
  \]
  Thus $\pInnerproduct{\whitney (\sigma^n)^\ast}{\whitney
    (\sigma^n)^\ast}_{L^2}$ is $\int_{\sigma^n} \mu /
  \abs{\sigma^n}^2$ which is $1/\abs{\sigma^n}$.
\end{proof}

\subsection{Algorithm for Whitney inner product matrix}
\label{subsec:mssmtrx_algrthm}

We motivate our algorithm for Whitney mass matrix computation by
making some observations about Example~\ref{ex:ttrhdrn_whtny}. The
first, and obvious observation is that matrix $M_p$ is symmetric,
being an inner product matrix. Thus only the diagonal entries and
those above (or below) the diagonal need be computed.  A more
interesting efficiency comes from the structure of the entries of the
matrix collections, such as ones shown in~\eqref{expr:ip03}
and~\eqref{expr:ip01}. Note that many entries repeat in the shorthand
collection of matrices in~\eqref{expr:ip03} and~\eqref{expr:ip01}. For
example the entry 12 appears 4 times by itself in the matrix
collection~\eqref{expr:ip03}. Moreover, due to the symmetry of inner
product, the entry 12 corresponds to the same result as the entry 21
and 21 appears 4 times as well. That entry also appears 4 times in the
collection~\eqref{expr:ip01}. Thus it is clear that a saving in
computational time can be achieved by doing such calculations only
once. That is, $\ainnerproduct{\d \mu_1}{\d \mu_2} = \ainnerproduct{\d
  \mu_2}{\d \mu_1}$ need only be computed once for the tetrahedron.

The determinants of all the matrices in a collection such
as~\eqref{expr:ip03} are needed to plug into an expression
like~\eqref{expr:ip03full} to obtain a single entry (in this case row
0, column 3) of the Whitney inner product matrix for $p$-cochains
($p=2$ in this case), whose size ($4\times 4$ in this case) depends on
the number of $p$-simplices in the simplicial complex. Thus reusing
repeated inner products of barycentric differentials can add up to a
substantial saving in computational expense when all the unique
entries of $M_p$ are computed. These savings are quantified later in
this subsection.

Another useful point to note in the example calculation is that the
collection~\eqref{expr:ip01} of matrices can be obtained from the
collection~\eqref{expr:ip03} by keeping the first digit in each entry
same and making the substitutions $1\rightarrow 0$; $2\rightarrow 1$;
and $3\rightarrow 3$ in the second digit. The first digits in the two
collections are the same because both correspond to the triangle
$\sigma^2_0$. The substitution above works for the second digit
because $\sigma^2_3 = [v_1, v_2, v_3]$ and $\sigma^2_1 = [v_0, v_1,
v_3]$. This suggests the use of a template simplex for creating a
template collection of matrices whose determinants are needed. The
actual instances of the collections can then be obtained by using the
vertex numbers in a given simplex. This is another idea that is used
in the algorithm implemented in PyDEC. The algorithm takes as input
a manifold simplicial $n$-complex $K$, embedded in $\R^N$ and $0
\le p \le n$. The output is $M_p$, an $N_p \times N_p$ matrix
representation of inner product on $C^p(K;\R)$ using elementary
cochain basis. If a naive algorithm, which does not take into
account the duplications in determinant calculations were to be used,
the number of operations required in the mass matrix calculation are 
\[
N_n \times \dfrac{\dbinom{n+1}{p+1}^2+\dbinom{n+1}{p+1}}{2} 
\times \binom{n}{p}^2 \times Np^2 \times (O(p!) \text{ or } O(p^3)) \, .
\]
The last term is written as $O(p!)$ or $O(p^3)$ because a determinant
can be computed using the formula for determinant or by LU
factorization. For low values of $p$ (i.e. $\le$ about 5) the formula
will likely be better.

According to the above formula, for example, for $n=3, p=2$, the
number of determinants required in a naive implementation of mass
matrix calculation would be 
\[
\dfrac{\dbinom{4}{3}^2+\dbinom{4}{3}}{2} 
\times \binom{3}{2}^2 = 10 \times 9 = 90\, .
\]
But there are only 21 unique determinants needed for $n=3, p=2$. Our
algorithm computes the unique
determinants first and the operation count is
\[
N_n \times \dfrac{\dbinom{n+1}{p}^2+\dbinom{n+1}{p}}{2} 
\times Np^2 \times (O(p!) \text{ or } O(p^3)) \, .
\]
\begin{figure}[p]
  \centering
  \begin{tabular}[h]{ccc}
    \includegraphics[scale=0.3]{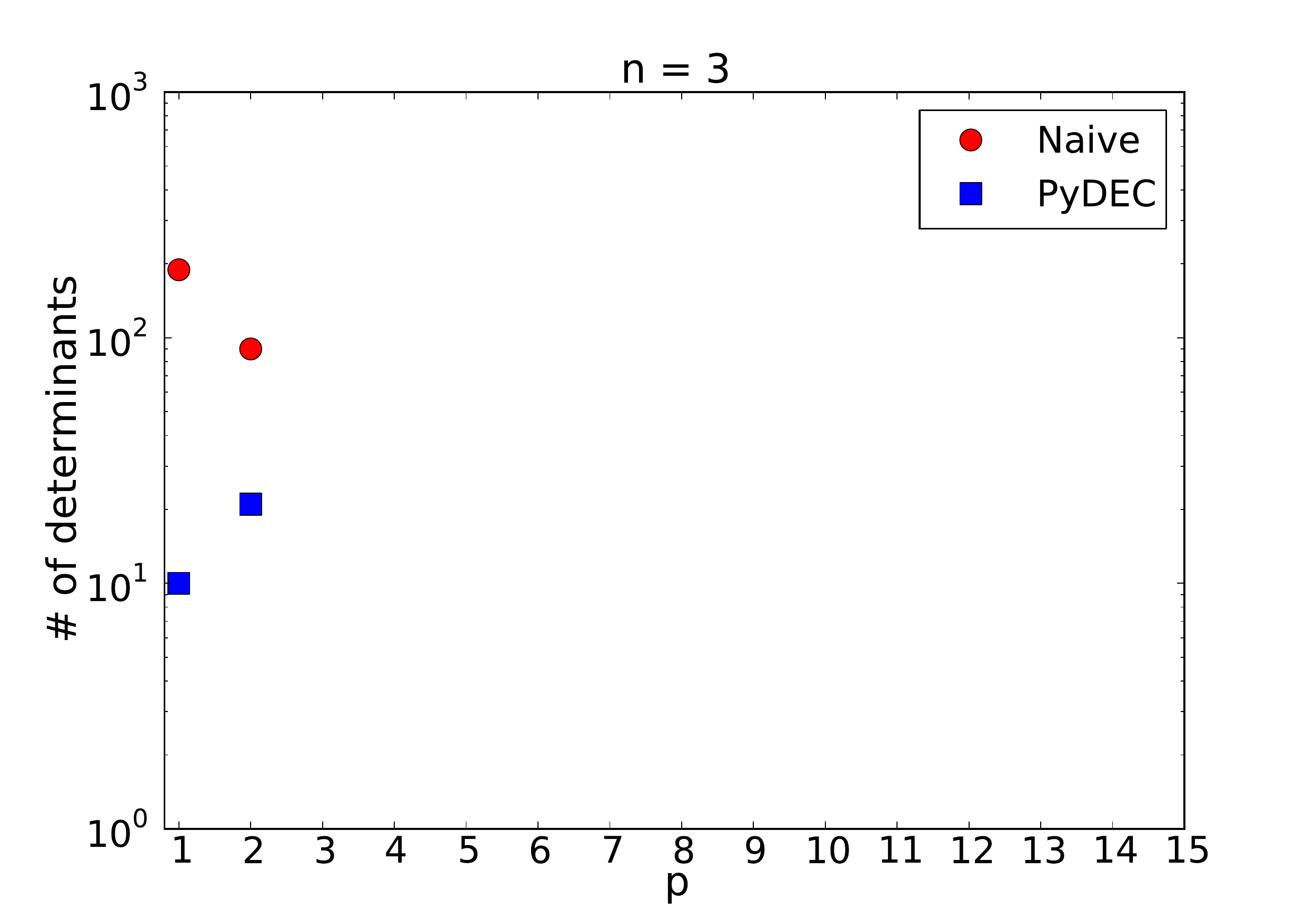} &
    \includegraphics[scale=0.3]{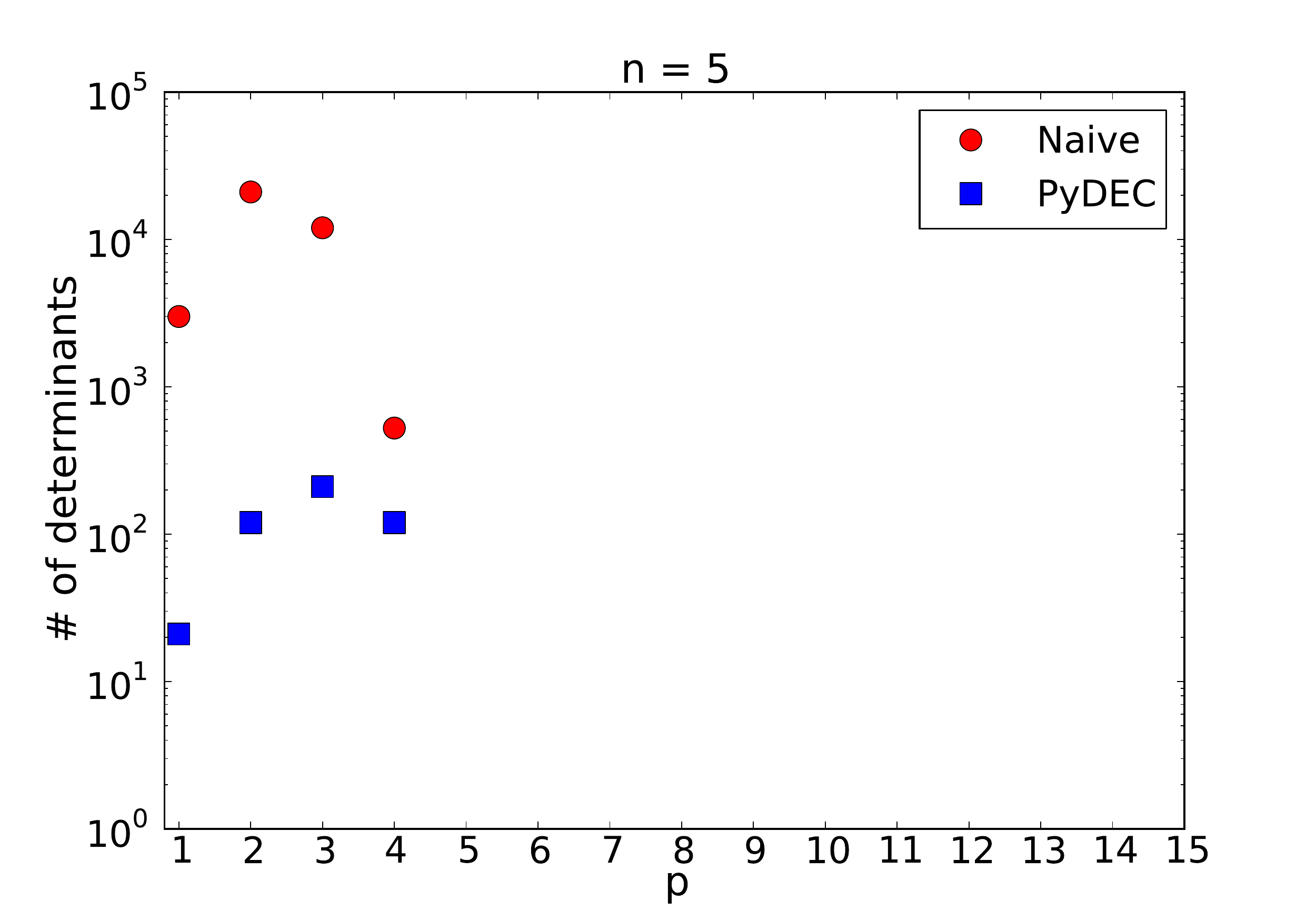}\\
    \includegraphics[scale=0.3]{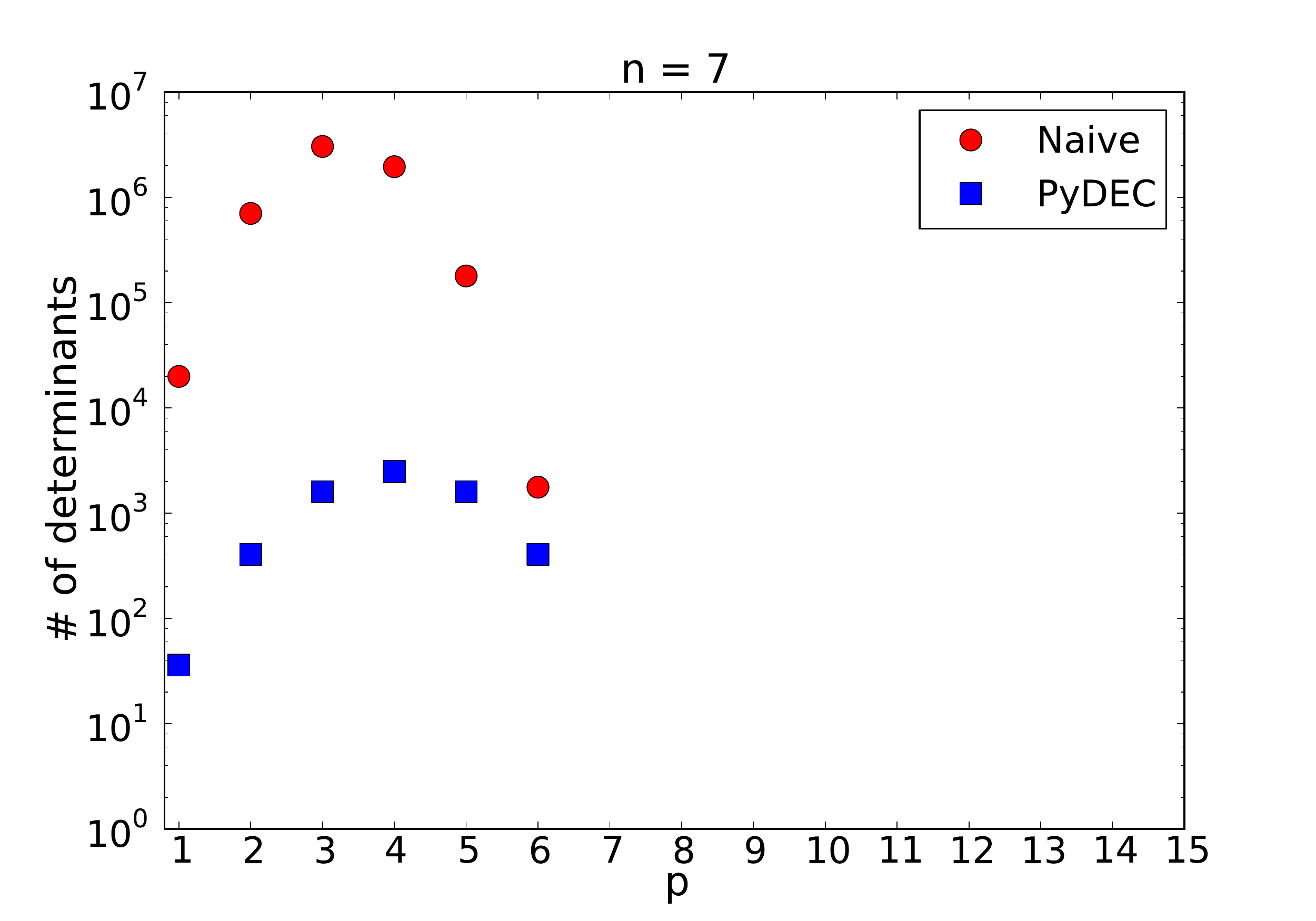} &
    \includegraphics[scale=0.3]{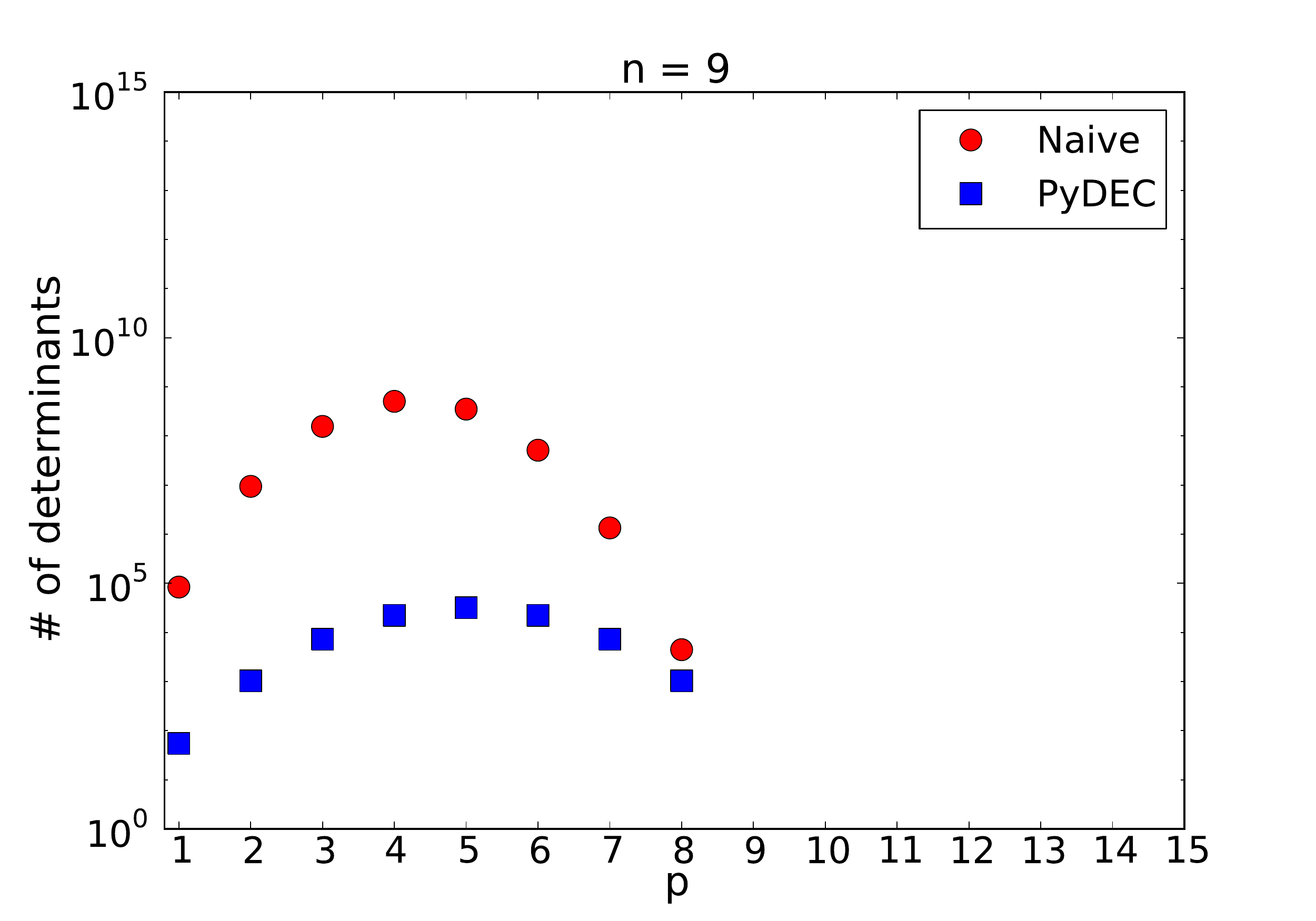}\\
    \includegraphics[scale=0.3]{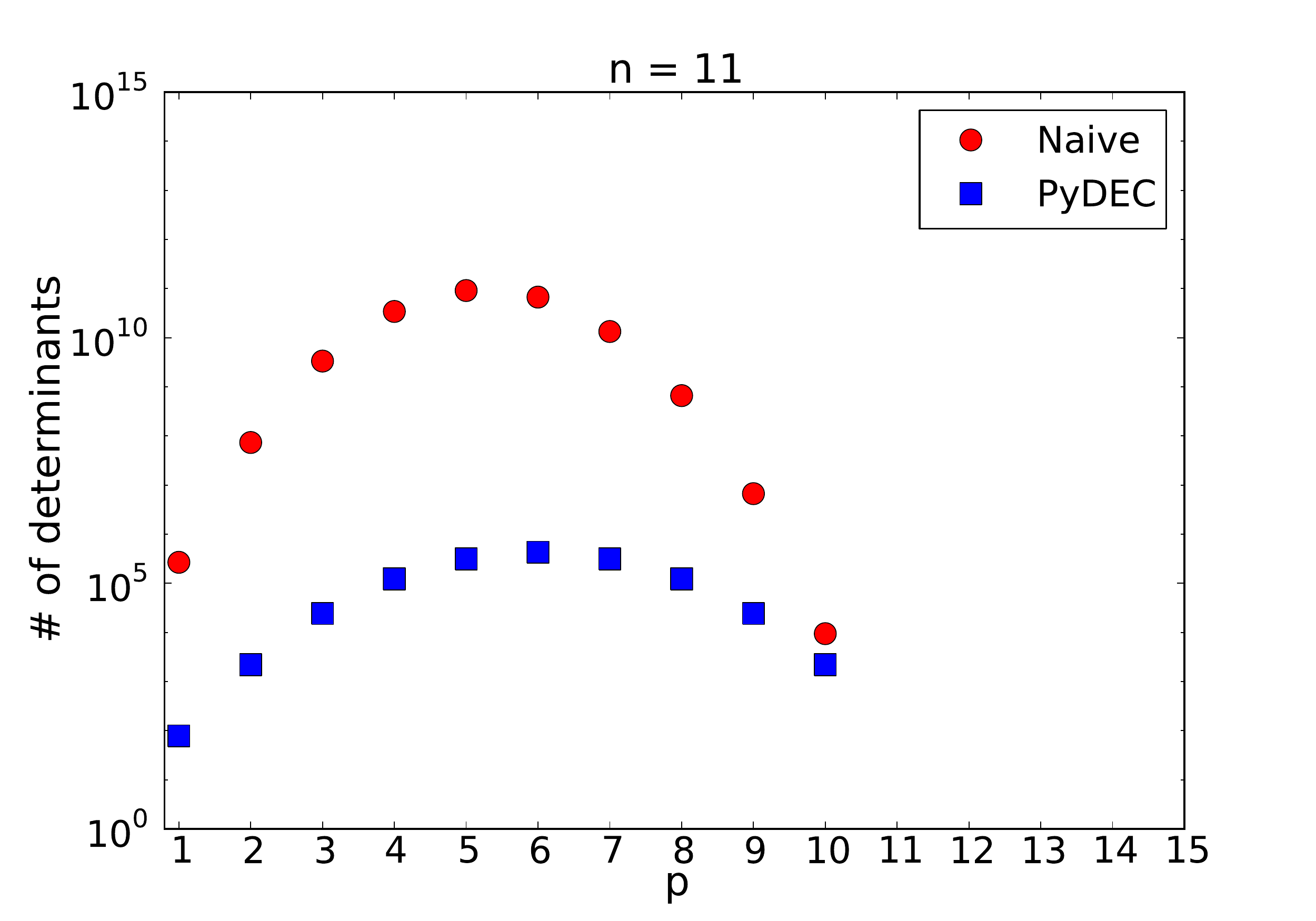} &
    \includegraphics[scale=0.3]{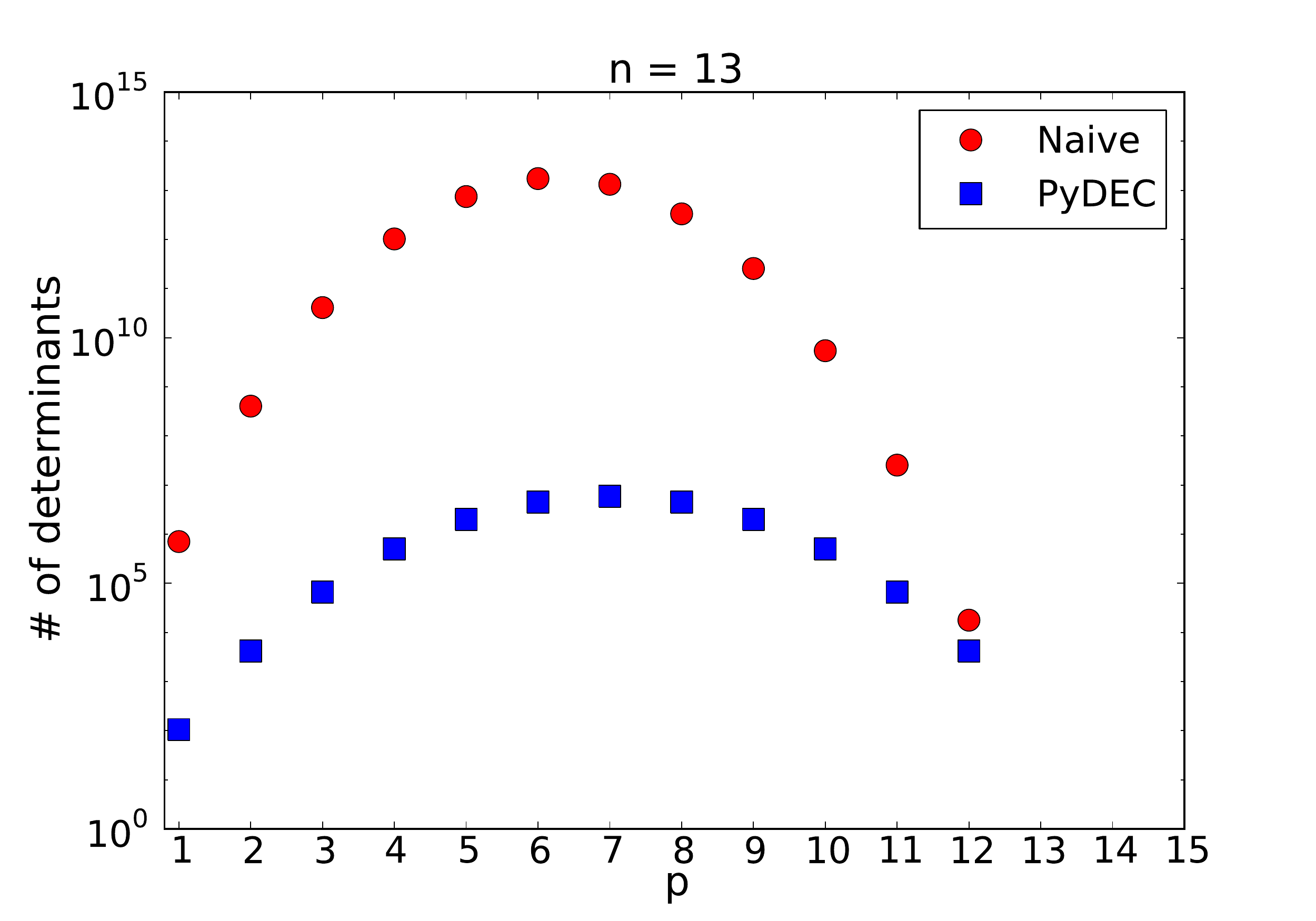} 
  \end{tabular}
  \caption{Comparison between PyDEC and a naive algorithm for
    computing Whitney mass matrix. The figure shows the number of
    determinant computations needed by the two algorithms, for various
    values of $p$, the Whitney form dimension, and $n$, the simplicial
    complex dimension. The embedding dimension is not relevant in
    these calculations. For $p=n$ case we use the shortcut described
    in Proposition~\ref{prop:p_equal_n} so that case is not shown.}
  \label{fig:dtrmnnt_cnts}
\end{figure}

Figure~\ref{fig:dtrmnnt_cnts} shows a comparison of determinant
counts for our algorithm compared with a naive algorithm that does the
duplicate work that our PyDEC algorithm avoids. Note that for any $n$,
the most advantage is gained intermediate values of $p$. The savings
that the PyDEC implementation provides over a naive algorithm are
several orders of magnitude, especially for moderately large $n$ and
higher. For $p=n$ case, in PyDEC we use the shortcut described in
Proposition~\ref{prop:p_equal_n}.

\section{Metric Dependent Operators}
\label{sec:metric}
We now describe the PyDEC implementations of some metric dependent
exterior calculus operators. The simplicial complex $K$ is now
supposed to be an approximation of a Riemannian $n$-manifold $M$. The
metric implemented in PyDEC is the one induced from an embedding space
$\R^N$. The main metric dependent operator is the Hodge star which
enables the discretization of codifferential and Laplace-deRham
operators. The sharp and flat, which are isomorphisms between 1-forms
and vector fields, are not implemented.

For the DEC Hodge star, the implementation is using the circumcentric
dual as in~\cite{Hirani2003, DeHiLeMa2005} and the other operators are
then simply defined in terms of the exterior derivative and the Hodge
star. For PyDEC's implementation of low order finite element exterior
calculus, we define the Hodge star to be the Whitney mass matrix
described in Section~\ref{sec:whtny_innrprdct}. The other operators
are defined by analogy with DEC even though the dual mesh concept is
not part of finite element exterior calculus.  Extensive experimental
justification for this approach can be seen in its effectiveness in
numerical experiments in~\cite{Bell2008} and in~\cite{HiKaWaWa2011}.

For the definitions in this section we will need two cochain complexes
of real-valued cochains. One will be on the simplicial complex $K$ and
for brevity we'll call this space of $p$-cochains $C^p(K)$ instead of
$C^p(K;\R)$. The other cochain complex is on the circumcentric dual
cell complex $\dual K$ and we'll denote the $(n-p)$-dimensional
cochains as $D^{n-p}(\dual K)$. At each dimension, these will be
connected by discrete Hodge star operators to be defined below. Since
the exterior derivative is the coboundary operator, the matrix
representation for the exterior derivative on the dual mesh is the
boundary operator. The matrix form for the DEC Hodge star
on $p$-cochains will be denoted $\hodge_p: C^p(K) \to D^{n-p}(\dual
K)$.  One box of the primal and dual complexes is shown below.
\begin{equation*}
  \begin{CD}
    C^p(K) @>\dd_p>> C^{p+1}(K) \\
    @VV\hodge_p V    @VV\hodge_{p+1} V \\
    D^{n-p}(\dual K) @<\dd_p^T< < D^{n-p-1}(\dual K)
  \end{CD}
\end{equation*}

As described in~\cite{Hirani2003, DeHiLeMa2005} and other references,
the DEC Hodge star is defined by
\[
\frac{\eval{\hodge_p \sigma^\ast_i}{\dual \sigma_j}}{\abs{\dual\sigma_j}} = 
\frac{\eval{\sigma^\ast_i}{\sigma_j}}{\abs{\sigma_j}}\, ,
\]
for $p$-simplices $\sigma_i$ and $\sigma_j$. Here $\sigma^\ast$ is the
elementary cochain corresponding to $\sigma$ and
$\eval{\sigma^\ast}{\tau}$ stands for the evaluation of the cochain
$\sigma^\ast$ on the elementary chain $\tau$. Thus the matrix
representation of the DEC Hodge star $\hodge_p$ is as a diagonal
matrix with $\hodge_p(i,i) =
\abs{\dual\sigma_i}/\abs{\sigma_i}$. In~\cite{Hirani2003,
  DeHiLeMa2005} this was defined for well-centered meshes. For the
codimension 1 Hodge star the definition extends to Delaunay meshes
with a slight additional condition for boundary simplices. This
extension involves computing the volume of $\dual\sigma$ taking into
account signs. Consider a codimension 1 simplex $\sigma$ shared by
simplices $L$ and $R$.  For the portion of $\dual\sigma$ corresponding
to $L$, the sign is positive if the circumcenter and remaining vertex
of $L$ are on the same side of $\sigma$. Similarly for $R$. (For
surface meshes and higher dimensional analogs the circumcenter
condition above is one way to define a Delaunay-like condition.)  If
$\sigma$ is a codimension 1 face of top dimensional $\tau$ and is on
domain boundary then the circumcenter of $\tau$ and vertex opposite to
$\sigma$ should be on the same side. The smooth Hodge star on
$p$-forms satisfies $\hodge\hodge = (-1)^{p(n-p)}$. In the discrete
setting $\hodge\hodge$ is written as $\hodge_p^{-1}\hodge_p$ or
$\hodge_p\hodge_p^{-1}$ and this is defined to be $(-1)^{p(n-p)}\, I$
where $I$ is the identity matrix.

In the smooth theory, the codifferential $\codiff_{p+1} :
\Omega^{p+1}(M) \to \Omega^p(M)$ is defined as $\codiff_{p+1} =
(-1)^{np+1} \hodge \d \hodge$ and so we define the discrete
codifferential $\dcodiff_{p+1} : C^p(K) \to C^{n-p}(K)$ as
$\dcodiff_{p+1} := (-1)^{np+1} \hodge_p^{-1} \dd_p^T \hodge_{p+1}$.
For finite element exterior calculus implemented in PyDEC, we take
this to be the definition, without reference to a dual mesh. If we now
take $\hodge_p$ to be the Whitney mass matrix then $\dd_p$ and
$\dcodiff_{p+1}$ are adjoints (up to sign) with respect to the Whitney
inner product on cochains as shown in~\cite{HiKaWaWa2011}. We will
call the use of Whitney mass matrix as $\hodge_p$ to be a
\emph{Whitney Hodge star} matrix.

In the discrete setting the Laplace-deRham operator is implemented in
the weak form. For $0<p<n$ the discrete definition is $\laplacian_p :=
\d_p\hodge_{p+1}\d_p + (-1)^{(p-1)(n-p+1)}
\hodge_p\d_{p-1}\hodge^{-1}_{p-1}\d^T_{p-1}\hodge_p$, with the
appropriate term dropped for the $p=0$ and $p=n$ cases. The above
expression involves inverses of the Hodge star, which is easy to
compute for DEC Hodge star since that is a diagonal matrix. For a
Whitney Hodge star see~\cite{Bell2008, HiKaWaWa2011} for various
approaches to avoiding explicitly forming the inverse Whitney mass
matrix in computations.

\subsection{Circumcenter Calculation}
\label{subsec:circumcenter}
Circumcentric duality is used in DEC. To compute the DEC Hodge star, a
basic computational step is the computation of the circumcenter of a
simplex. We give here a linear system for computing the circumcenter
using barycentric coordinates.

The circumcenter of a simplex is the unique point that is equidistant
from all vertices of that simplex.  In the case that a simplex (or
face) is not of the same dimension as the embedding (e.g. a triangle
embedded in $\R^4$), we choose the point that lies in the affine space
spanned by the vertices of the simplex.  In either case we can write
the circumcenter in terms of barycentric coordinates of the simplex.

Let $\sigma^p$ be the $p$-simplex defined by the points $\{ v_0, v_1,
\ldots v_p\}$ in $\R^N$.  Let $R$ denote the circumradius and $c$ the
circumcenter of simplex, which can be written in barycentric
coordinates as $c = \sum_j b_j v_j$ where $b_j$ is the barycentric
coordinate for the circumcenter corresponding to $v_j$.  For each
vertex $i$ we have
\[
\left\Vert v_i - \sum_{j=0}^p b_j v_j \right\Vert^2 - R^2 = 0 \, ,
\]
which can be rewritten as
\[
v_i \cdot v_i - 2 v_i \cdot \left(\sum_{j=0}^p b_j v_j\right) + 
\left\Vert \sum_{j=0}^p b_j v_j\right\Vert^2 - R^2 = 0\, .
\]
Here the norm and the dot product are the standard ones on $\R^N$.
Rearranging the above yields
\[
2 v_i \cdot \left(\sum_{j=0}^p b_j v_j\right) - 
\left(\left\Vert\sum_{j=0}^p b_j v_j\right\Vert^2 - R^2\right) = 
v_i \cdot v_i\, .
\]
The second term on the left hand side is some scalar which is unknown,
but is the same for every equation. So we can replace it by the
unknown $Q$ and write
\[
2 v_i \cdot \left(\sum_{j=0}^p b_j v_j\right) + Q = v_i \cdot v_i \, .
\]
With the additional constraint that barycentric coordinates sum to
one, we have a linear system with $p + 2$ unknowns ($b_0 \ldots b_p$
and $Q$) and $p + 2$ equations with the following matrix form
\[ 
\begin{pmatrix}
 2 v_0 \cdot v_0 & 2 v_0 \cdot v_1 & \dots & 2 v_0 \cdot v_p &1\\
 2 v_1 \cdot v_0 & 2 v_1 \cdot v_1 & \dots & 2 v_1 \cdot v_p &1\\
\vdots & \vdots & \ddots & \vdots & \vdots  \\
 2 v_p \cdot v_0 & 2 v_p \cdot v_1 & \dots & 2 v_p \cdot v_p &1\\
 1 & 1 & \dots & 1 & 0\\
\end{pmatrix} 
\begin{pmatrix}
b_0 \\
b_1 \\
\vdots \\
b_p \\
Q \\
\end{pmatrix} 
= 
\begin{pmatrix}
v_0 \cdot v_0 \\
v_1 \cdot v_1 \\
\vdots \\
v_p \cdot v_p \\
1 \\
\end{pmatrix}
\]
The solution to this yields the barycentric coordinates from which the
circumcenter $c$ can be located. Another quantity required for DEC
Hodge star is the unsigned volume of a simplex. This can be computed
by the well-known formula $\sqrt{\det V^T V}/p!$ where $V$ is the $p$
by $N$ matrix with rows formed by the vectors $\{ v_1 - v_0, v_2 -
v_0, \ldots v_p - v_0\}$.

\section{Examples} \label{sec:examples}

In the domains for which PyDEC is intended, it is often possible to
easily translate the mathematical formulation of a problem into a
working program. To make this point, and to demonstrate a variety of
applications of PyDEC, we give 5 examples from different fields. The
first example (Section~\ref{subsec:cavity}) is a resonant cavity
eigenvalue problem in which Whitney forms work nicely while the nodal
piecewise linear Lagrange vector finite element $\mathcal{P}_1^2$
fails when directly applied. The second is Darcy flow
(Section~\ref{subsec:darcy}), which is an idealization of the steady
flow of a fluid in a porous medium. We solve it here using DEC. The
third problem (Section~\ref{subsec:cohomology}) is computation of a
basis for the cohomology group of a mesh with several holes. This is
achieved in our code here by Hodge decomposition of cochains, again
using DEC. Next example (Section~\ref{subsec:sensor}) is an
idealization of the sensor network coverage problem. Some randomly
located idealized sensors in the plane are connected into a Rips
complex based on their mutual distances. Then a harmonic cochain
computation reveals the possibility of holes in coverage. The last
example (Section~\ref{subsec:ranking}) involves the ranking of
alternatives by a least squares computation on a graph.

None of these problem is original and they all have been treated in
the literature by a variety of techniques. We emphasize that we are
including these just to demonstrate the capabilities of PyDEC. We have
included the relevant parts of the Python code in this paper. The full
working programs are available with the PyDEC
package~\cite{BeHi2008a}.

\subsection{Resonant cavity curl-curl problem}
\label{subsec:cavity}

An electromagnetic resonant cavity is an idealized box made of a
perfect conductor and containing no enclosed charges in which
Maxwell's equations reduce to an eigenvalue problem.  Several authors
have popularized this example as one of many striking examples that
motivate finite element exterior calculus. See for
instance~\cite{ArFaWi2010}. The use of $\mathcal{P}_1^2$ finite
element space, i.e. piecewise linear, Lagrange finite elements with 2
components, yields a corrupted spectrum.  On the other hand, the use
of $\mathcal{P}_1^-$ elements, i.e., Whitney 1-forms yields the
qualitatively correct spectrum. For detailed analysis and background
see~\citep{BoFeGaPe1999,ArFaWi2010}.

Let $M \subset \R^2$ be a square domain with side length $\pi$. We
first give the equation in vector calculus notation and then in the
corresponding exterior calculus notation. In the former, the resonant
cavity problem is to find vector fields $E$ and eigenvalues $\lambda
\in \R$ such that
\[
\vcurl\, \curl E = \lambda E \quad \text{ on } M \quad \text{ and }
\quad E_\parallel = 0
\text{ on } \partial M \, ,
\]
where $E_\parallel$ is the tangential component of $E$ on the
boundary. Here $\vcurl \phi = (\partial \phi / \partial y, -\partial
\phi / \partial x)$ and $\curl v = \partial v_2 / \partial x
- \partial v_1 / \partial y$ for scalar function $\phi$ and vector
field $v = (v_1, v_2)$. Note that for $\lambda \ne 0$ this equation is
equivalent to the pair of equations $\Delta E = \lambda E$ and $\div E
= 0$. This is because the vector Laplacian $\Delta = \vcurl \circ
\curl - \grad \circ \div$ and $\div \circ \curl = 0$.

Now we give the equation in exterior calculus notation so the
transition to PyDEC will be easier. Let $u \in \Omega^1(M)$ be the unknown
electric field 1-form and $i : \partial M \hookrightarrow M$ the
inclusion map. Then the above vector calculus equation is equivalent
to 
\begin{alignat*}{2}
 \delta_2 \d_1 u & = \lambda\, u &&\quad \text{in } M\\
   i^\ast u &= 0 && \quad \text{on } \boundary M\, .
\end{alignat*}
The pullback $i^\ast u$ by inclusion map means restriction of $u$ to
the boundary, i.e., allowing only vectors tangential to the boundary
as arguments to $u$. As usual, we will seek $u$ not in $\Omega^1(M)$
but in $H\Omega^1(M)$ subject to boundary conditions. Define the
vector space $V = \{v\,\vert\, v \in H\Omega^1(M), i^\ast v = 0 \text{
  on } \partial M\}$. 

To express the PDE in weak form, we seek a $(u,\lambda)$ in $V\times
\R$ such that $\pinnerproduct{\codiff_2 \d_1 u}{v}_{L^2} = \lambda
\pinnerproduct{u}{v}_{L^2}$ for all $v \in V$. By the properties of
the codifferential, the expression on the left is equal to
$\pinnerproduct{\d_1 u}{\d_1 v}_{L^2} - \int_{\partial M} u \wedge
\hodge \d_1 v$. But the boundary term is 0 because $u$ is in $V$. Thus
the weak form is to find a $(u,\lambda) \in V\times \R$ such that
$\pinnerproduct{\d_1 u}{\d_1 v}_{L^2} = \lambda
\pinnerproduct{u}{v}_{L^2}$ for all $v \in V$. 

Taking the Galerkin approach of looking for a solution in a finite
dimensional subspace of $V$ here we pick the space of Whitney 1-forms,
that is, $\mathcal{P}_1^-\Omega^1$ as the finite dimensional
subspace. We define these over a triangulation of $M$ which we will
call $K$.  The Whitney map $\whitney : C^1(K;\R) \to
L^2\Omega^1(\abs{K})$ is an injection with its image
$\mathcal{P}_1^-(K)$. Thus an equivalent formulation is over
cochains. Using the same names for the variables, we seek a
$(u,\lambda) \in C^1(K;\R) \times \R$ such that $\pinnerproduct{\d_1
  \whitney u}{\d_1 \whitney v}_{L^2} = \lambda \pinnerproduct{\whitney
  u}{\whitney v}_{L^2}$ for all 1-cochains $v\in C^1(K;\R)$. Since the
Whitney map commutes with the exterior derivative and coboundary
operator, and using the definition of cochain inner product, the above
is same as $\pinnerproduct{\d_1 u}{\d_1 v} = \lambda
\pinnerproduct{u}{v}$ where now the inner product is over cochains and
$\d_1$ is the coboundary operator. In matrix notation, using
$\hodge_1$ and $\hodge_2$ to stand for the Whitney mass matrices $M_1$
and $M_2$, the generalized eigenvalue problem is to find $(u,\lambda)
\in C^1(K;\R) \times \R$ such that
\[ 
\d_1^T \, \hodge_2 \, \d_1 u = \lambda \, \hodge_1 u \, .
\]

We now translate this equation into PyDEC code. Once the appropriate
modules have been imported, a simplicial complex object \texttt{sc} is
created after reading in the mesh files. Now the main task is to find
matrix representations for the stiffness matrix $\d_1^T \, \hodge_2 \,
\d_1$and the mass matrix $\hodge_1$. This is accomplished by the
following two lines, where \texttt{K} is the stiffness matrix :
\begin{lstlisting}[firstline=25,lastline=26]
K = sc[1].d.T * whitney_innerproduct(sc,2) * sc[1].d
M = whitney_innerproduct(sc,1)
\end{lstlisting} The
boundary conditions can be imposed by simply removing the edges that
lie on the boundary. The indices of such edges is easily determined
and stored in the list \texttt{non\_boundary\_indices} which is used
below to impose the boundary conditions :
\begin{lstlisting}[firstline=34,lastline=35]
K = K[non_boundary_indices,:][:,non_boundary_indices]
M = M[non_boundary_indices,:][:,non_boundary_indices]
\end{lstlisting} Now
all that remains is to solve the eigenvalue problem. To simplify the
code and because the matrix size is small, we use the dense eigenvalue
solver \texttt{scipy.linalg.eig}
\begin{lstlisting}[firstline=39,lastline=39]
eigenvalues, eigenvectors = eig(K.todense(), M.todense())
\end{lstlisting} Some
of the resulting eigenvalues are displayed in the left part of
Figure~\ref{fig:cavity}. The 1-cochain $u$ which is the eigenvector
corresponding to one of these eigenvalues is shown as a vector field
in the right part of Figure~\ref{fig:cavity}. The visualization as a
vector field is achieved by interpolating the 1-cochain $u$ using the
Whitney map and then sampling the vector field $(\whitney u)^\sharp$
at the barycenter. This is achieved by the PyDEC command:
\begin{lstlisting}[firstline=54,lastline=54]
bases, arrows = simplex_quivers(sc,all_values)
\end{lstlisting}
where \texttt{all\_values} contains both the known and the computed
values of the 1-cochain. There is no sharp operator in PyDEC. But
since PyDEC only implements the Riemannian metric from the embedding
space of simplices the transformation from 1-form to vector field just
involves using the components of the Whitney 1-form as the vector
field components.

\begin{figure}[t]
  \centering
  \includegraphics[height=2in,trim=0.25in 0.4in 0.25in 0.5in, clip]
%  {pydec/cavity/eigenvals50_notitle}
  {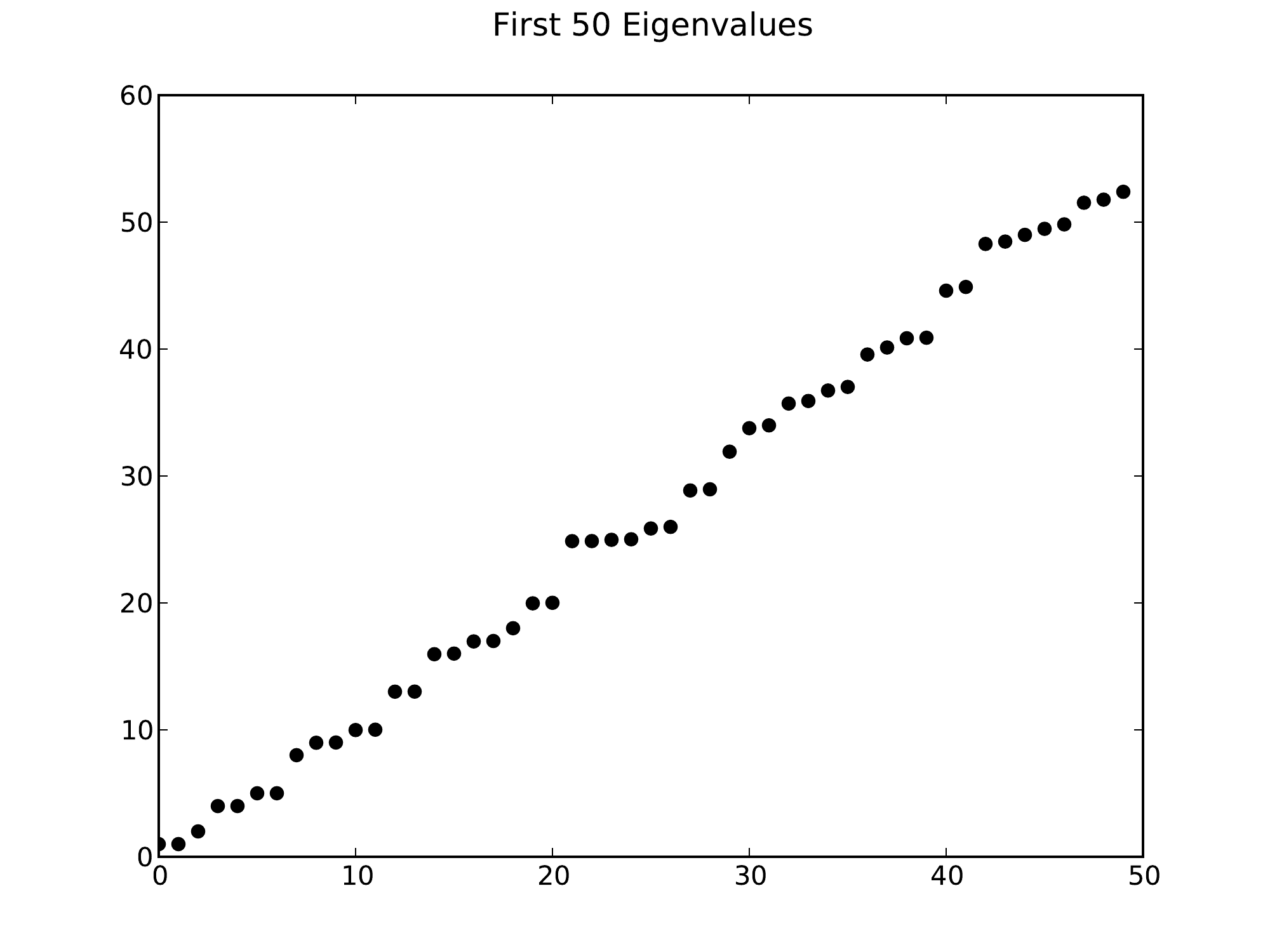}
  \includegraphics[height=2.2in,trim=1.1in 0.8in 1in 0.75in, clip]
%  {pydec/cavity/eigenvec3_notitle}
  {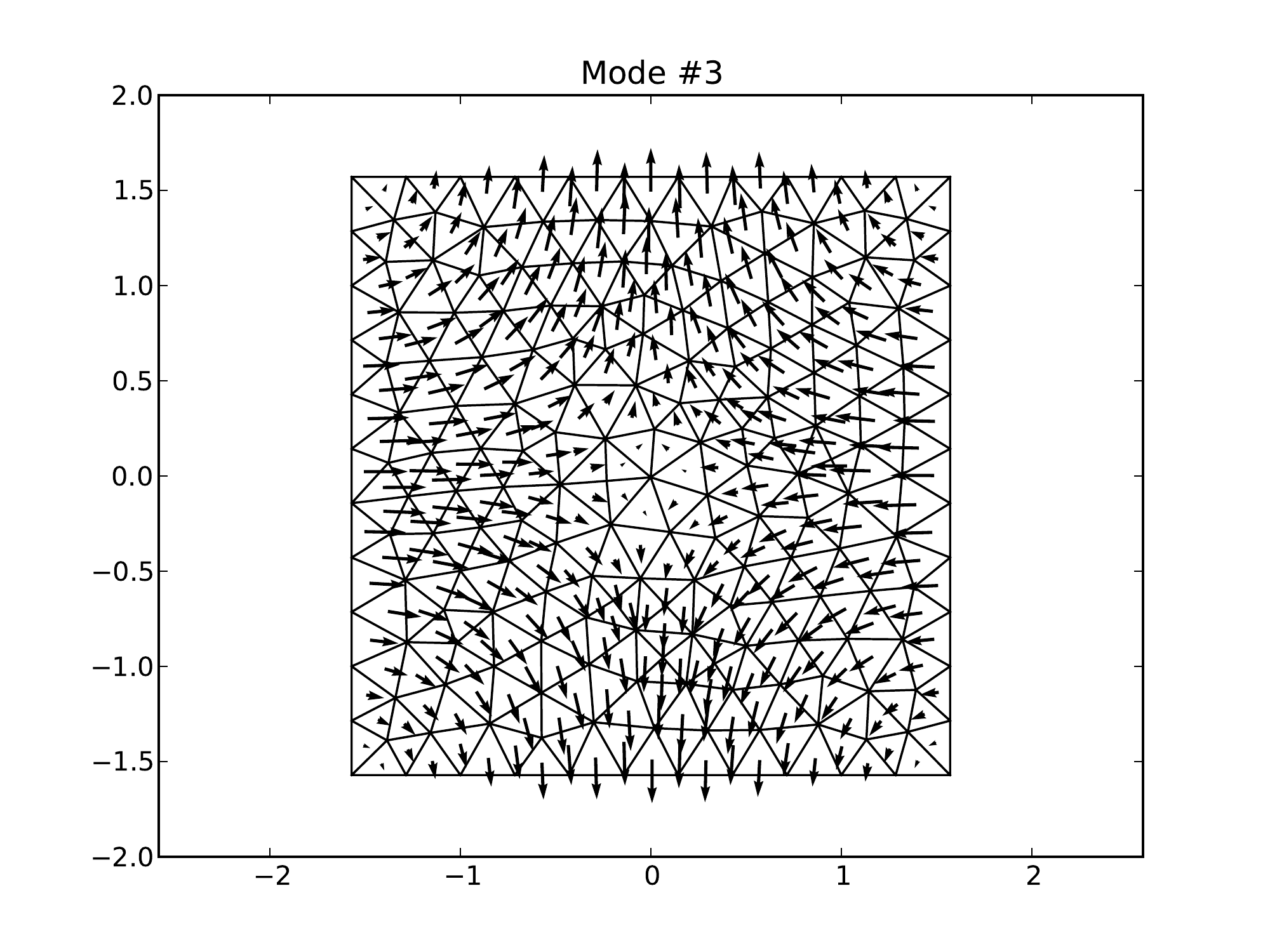}
  \caption{The first 50 nonzero eigenvalues for the resonant cavity
    problem of Section~\ref{subsec:cavity} and the eigenvector
    corresponding to one of these eigenvalues. The eigenvector is a
    1-cochain which is visualized as a vector field by first
    interpolating it using a Whitney map. See
    Section~\ref{subsec:cavity} for details.}
  \label{fig:cavity}
\end{figure}

\subsection{Darcy flow or Poisson's in mixed form}
\label{subsec:darcy}
We give here a brief description of the equations of Darcy flow and
their PyDEC implementation. For more details
see~\cite{HiNaCh2011}. The resonant cavity example in
Section~\ref{subsec:cavity} was implemented using finite element
exterior calculus. For variety we use a DEC implementation for Darcy
flow.

Darcy flow is a simple model of steady state flow of an incompressible
fluid in a porous medium. It models the statement that flow is from
high to low pressure. For a fixed pressure gradient, the velocity is
proportional to the permeability $\kappa$ of the medium and inversely
proportional to the viscosity $\mu$ of the fluid. Let the domain be
$M$, a polygonal planar domain. Assuming that there are no sources of
fluid in $M$ and there is no other force acting on the fluid, the
equations of Darcy flow are
% \begin{alignat}{2}
%   v + \frac{\kappa}{\mu} \nabla p &= 0
%   &&\quad \text{in } M \notag \, ,\\
%   \div v &= 0
%   &&\quad \text{in } M \label{eq:vfdrcy}\, ,\\
%   v \cdot \hat{n} &= \psi 
%   &&\quad \text{on } \partial M\, ,\notag
% \end{alignat}
\begin{equation}\label{eq:vfdrcy}
  v + \frac{\kappa}{\mu} \nabla p = 0 \quad\text{and}
  \quad \div v = 0 \quad \text{in } M \quad \text{with}\quad
  v \cdot \hat{n} = \psi \quad \text{on } \partial M.
\end{equation}
where $\kappa > 0$ is the coefficient of permeability of the medium
$\mu > 0$ is the coefficient of (dynamic) viscosity of the fluid,
$\psi:\partial M \rightarrow \mathbb{R}$ is the prescribed normal
component of the velocity across the boundary, and $\hat{n}$ is the
unit outward normal vector to $\partial M$. For consistency
$\int_{\partial M} \psi \, d\Gamma = 0$, where $d\Gamma$ is the
measure on $\partial M$. Since $\div \circ \grad = \Delta$, the
simplified Darcy flow equations above are equivalent to Laplace's
equation.

Let $K$ be a simplicial complex that triangulates $M$. Instead of
velocity and pressure, we will use flux and pressure as the primary
unknowns. The flux through the edges is $f=\hodge v^\flat$ and thus it
will be a primal 1-cochain. Although PyDEC does not implement a flat
operator, this is not an issue here because we never solve for $v$,
and make $f$ itself one of the unknowns. This implies that the
pressure $p$ will be a dual 0-cochain since $\hodge \d p$ has to be of
the same type as $f$. The choice to put flux on primal edges and
pressures on circumcenters can be reversed, as shown in a dual
formulation in~\cite{GiBa2010a}. In exterior calculus notation, the
PDE in~\eqref{eq:vfdrcy} is
% \begin{align*}
% -(\mu/k)\hodge f + \d p &= 0\, ,\\
% \d f &= 0\, ,
% \end{align*}
$-(\mu/k)\hodge f + \d p = 0$ and $\d f = 0$, which, when discretized,
translates to the matrix equation
\[
\begin{bmatrix}
-(\mu/k)\hodge_1 & \d_1^T\\
\d_1 & 0
\end{bmatrix} 
\begin{bmatrix}
f \\ p
\end{bmatrix} =
\begin{bmatrix}
  0\\0
\end{bmatrix}\, .
\]
In PyDEC, the construction of this matrix is straightforward. Once a
simplicial complex \texttt{sc} has been constructed, the following 3
lines construct the matrix in the system above :
\begin{lstlisting}[firstline=40,lastline=42]
d1 = sc[1].d; star1 = sc[1].star
A = bmat([[(-mu/k)*sc[1].star, sc[1].d.T],
          [sc[1].d, None]], format='csr')
\end{lstlisting}

After computing the boundary condition in terms of flux through the
boundary edges, the linear system is adjusted for the known values and
then solved for the fluxes and pressures. Figure~\ref{fig:darcy} shows
the solution for the case of constant horizontal velocity and linear
pressure gradient.

\begin{figure}[t]
  \centering
  \includegraphics[scale=0.4,trim=0in 0in 0in 0.5in, clip]{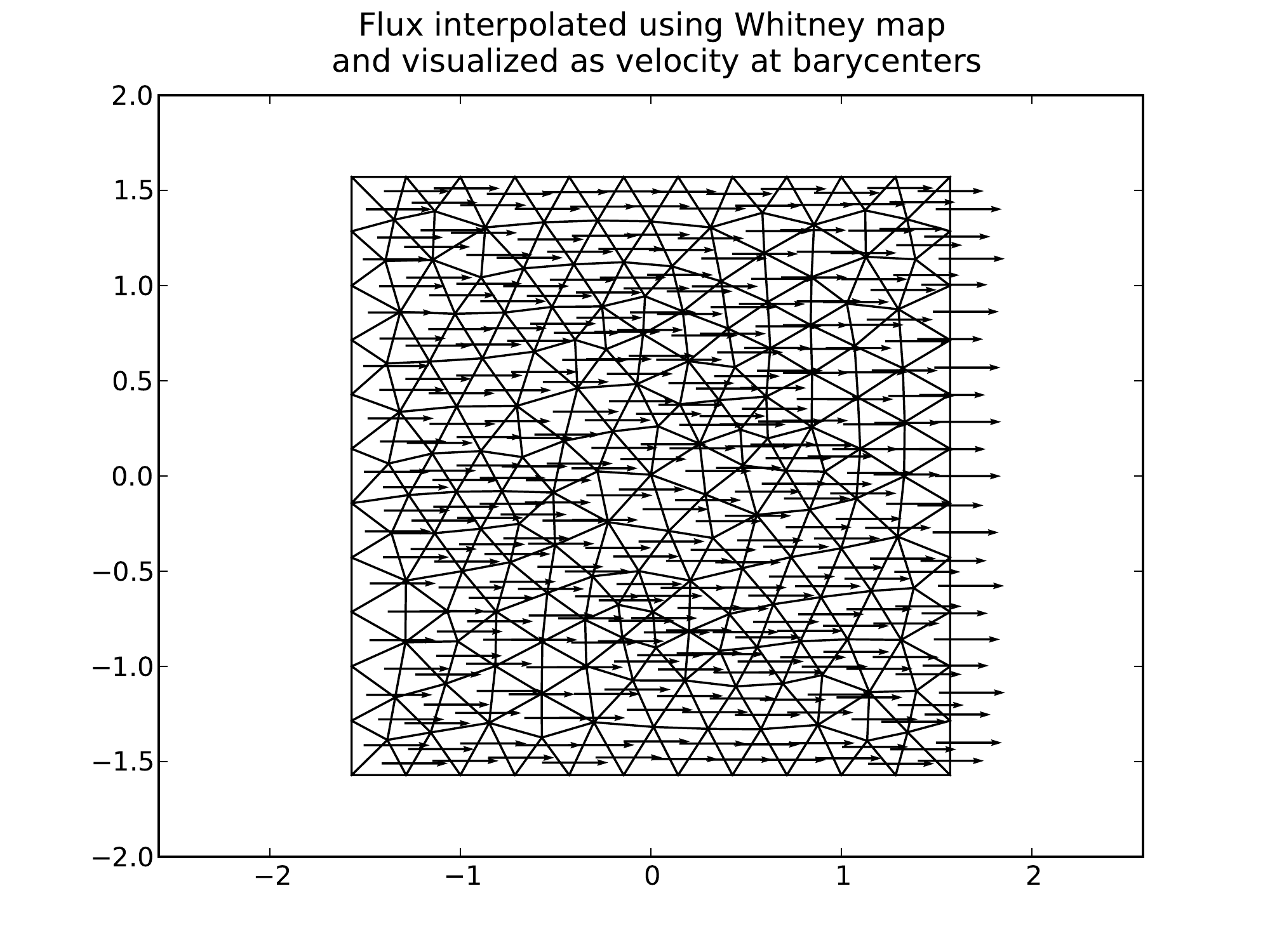}
  \includegraphics[scale=0.4]{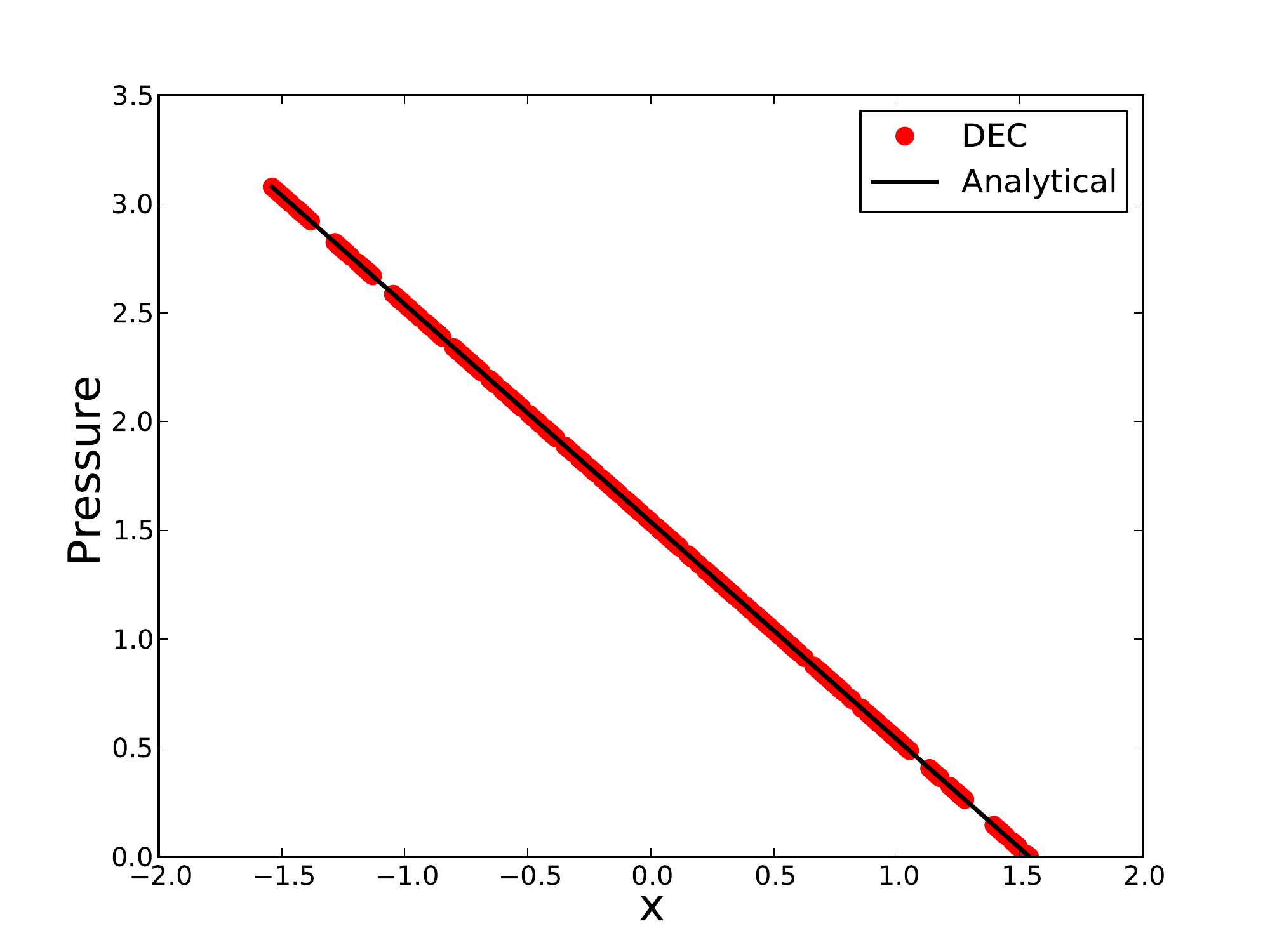}
  \caption{Darcy flow using discrete exterior calculus. The boundary
    condition is that fluid is coming in from left and leaving from
    right with velocity 1. The velocity inside should be constant and
    pressure should be linear. The flux and pressure are computed in a
    mixed formulation. The flux is taken to be a primal 1-cochain
    associated with primal edges, and the pressure is a dual 0-cochain
    on dual vertices, which are circumcenters of the triangles. The
    velocity is obtained by Whitney interpolation of the flux, which
    is sampled at the barycenters. See Section~\ref{subsec:darcy} for
    more details.}
\label{fig:darcy}
\end{figure}

\subsection{Cohomology basis using Hodge decomposition}
\label{subsec:cohomology}

The Hodge Decomposition Theorem~\cite[page 539]{AbMaRa1988} states
that for a compact boundaryless smooth manifold $M$, for any $p$-form
$\omega \in \Omega^p(M)$, there exists an $\alpha\in\Omega^{p-1}(M)$,
$\beta \in \Omega^{p+1}(M)$, and a harmonic form $h \in \Omega^p(M)$
such that $\omega = d \alpha + \delta \beta + h$. Here harmonic means
that $\Delta h = 0$, where $\Delta$ is the Laplace-deRham operator
$\d\delta +\delta \d$. Moreover $\d\alpha$, $\delta \beta$ and $h$ are
mutually $L^2$-orthogonal, which makes them uniquely determined. In
case of a manifold with boundary, the decomposition is similar, with
some additional boundary conditions. See~\cite{AbMaRa1988} for
details.

The Hodge-deRham theorem~\cite{AbMaRa1988}, relates the analytical
concept of harmonic forms with the topological concept of
cohomology. For any topological space, the cohomology groups or vector
spaces of various dimension capture essential topological information
about the space~\cite{Munkres1984}. For the manifold $M$ above, the
$p$-dimensional cohomology group with real coefficients, which is a
finite-dimensional space, is denoted $H^p(M;\R)$ or just $H^p(M)$. For
example, for a torus, $H^1$ has dimension 2. For a square with 4 holes
used in this example, which does have boundaries, $H^1$ has dimension
4. The elements of $H^p(M)$ are equivalence classes of closed forms
(those whose $\d$ is 0). Two closed forms are equivalent if their
difference is exact (that is, is $\d$ of some form). While the
representatives of 1-homology spaces can be visualized as loops around
holes, handles, and tunnels, those of 1-cohomology should be
visualized as fields. If the space of harmonic forms is denoted
$\mathcal{H}^p(M)$, then the Hodge-deRham theorem says that it is
isomorphic, as a vector space, to the $p$-th cohomology space $H^p(M)$
in the case of a closed manifold. See~\cite{Jost2005} for
details. Again, the case of $M$ with boundary requires some
adjustments in the definitions, as given in~\cite{AbMaRa1988}.

For finite dimensional spaces, Hodge decomposition follows from very
elementary linear algebra. If $U$, $V$ and $W$ are finite-dimensional
inner product vector spaces and $A : U \to V$ and $B: V \to W$ are
linear maps such that $B \circ A = 0$ then middle vector space $V$
splits into 3 orthogonal components, which are $\im A$, $\im B^T$, and
$\ker A^T \cap \ker B$. In this example, we find a basis for $H^1$ for
a square. This is done by finding a basis of harmonic 1-cochains.
Thus given a 1-cochain $\omega$, its discrete Hodge decomposition
exists and is $\omega = \dd_0 \alpha + \dcodiff_2 \beta + h$.  In this
example, the cochains $\alpha$ and $\beta$ are obtained by solving the
linear systems
% \begin{align*}
% \dcodiff_1 \dd_0 \alpha &= \dcodiff_1 \omega\, , \\
% \dd_1 \dcodiff_2 \beta &= \dd_1 \omega \, .
% \end{align*}
$\dcodiff_1 \dd_0 \alpha = \dcodiff_1 \omega$ and $\dd_1 \dcodiff_2
\beta = \dd_1 \omega$.  The harmonic component can then be computed
by subtraction.

In the example code, the main function is the one that computes the
Hodge decomposition of a given cochain \texttt{omega}. First empty
cochains for \texttt{alpha} and \texttt{beta} are created:
\begin{lstlisting}[firstline=30,lastline=33]
    sc = omega.complex    
    p = omega.k
    alpha = sc.get_cochain(p - 1)
    beta  = sc.get_cochain(p + 1)    
\end{lstlisting} Now
the solution for \texttt{alpha} and \texttt{beta} closely follows the
above equations for $\alpha$ and $\beta$:
\begin{lstlisting}[firstline=36,lastline=38]
    A = delta(d(sc.get_cochain_basis(p - 1))).v
    b = delta(omega).v
    alpha.v = cg( A, b, tol=1e-8 )[0]
\end{lstlisting}
\begin{lstlisting}[firstline=41,lastline=43]
    A = d(delta(sc.get_cochain_basis(p + 1))).v
    b = d(omega).v
    beta.v = cg( A, b, tol=1e-8 )[0]
\end{lstlisting} Even
though the matrices \texttt{A} above are singular, the solutions
exist, and since conjugate gradient is used, the presence of the
nontrivial kernels does not pose any problems~\cite{BoLe2005}.

\begin{figure}[t]
  \centering
  \includegraphics[height=1.5in,trim=0.75in 0.75in 0.75in 0.75in,
  clip]{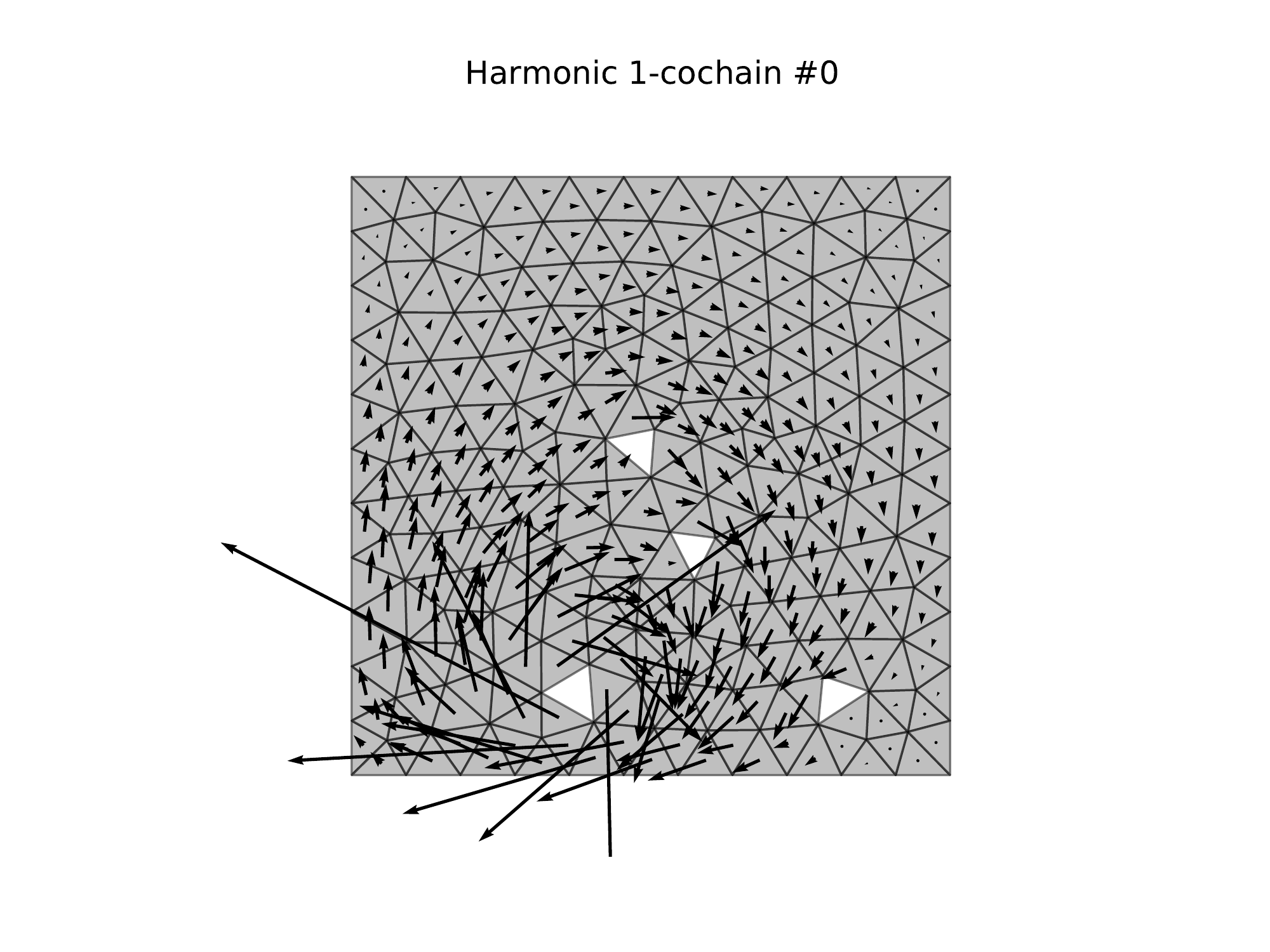}
  \includegraphics[height=1.5in,trim=0.75in 0.75in 0.75in 0.75in,
  clip]{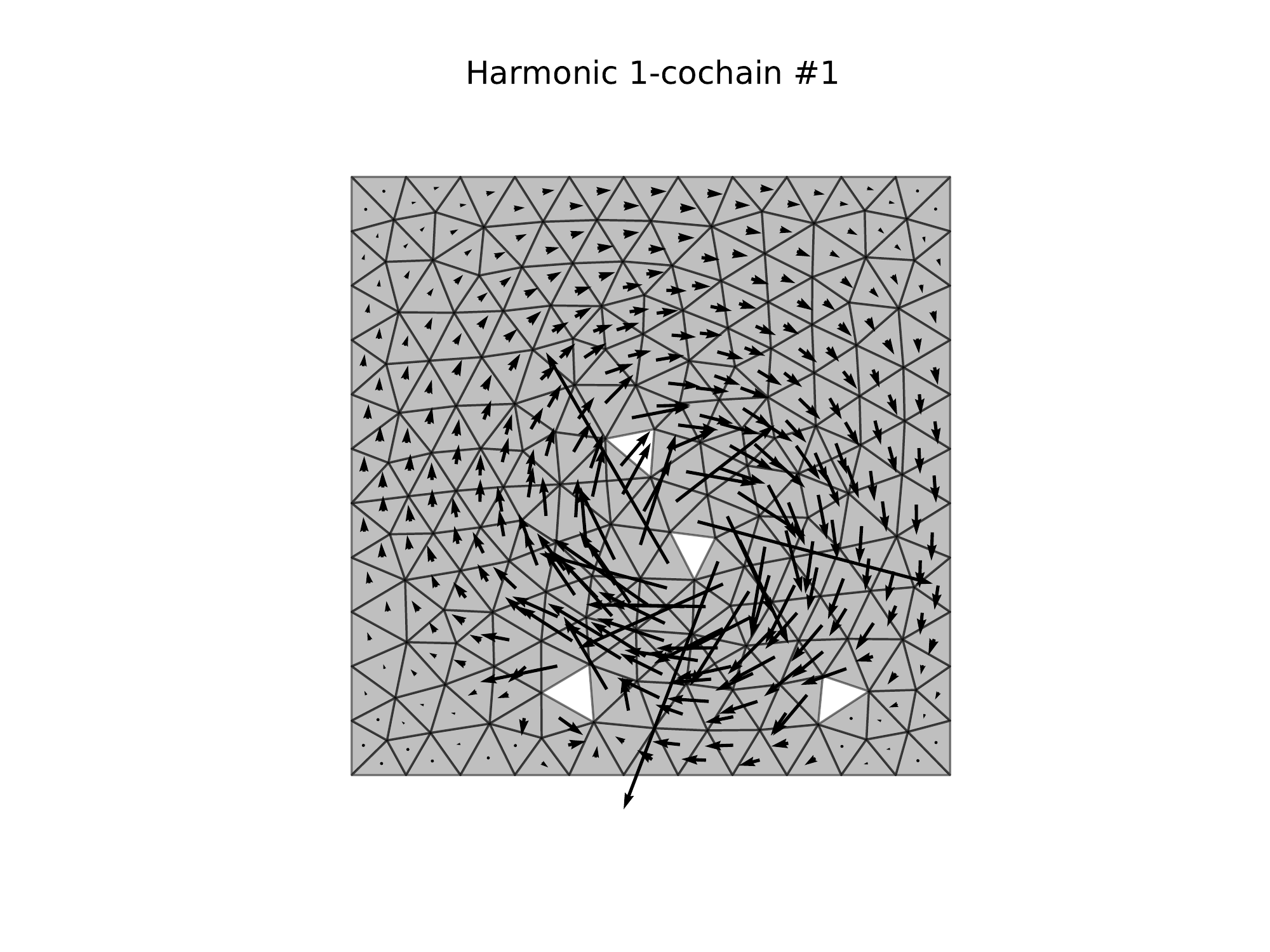} \\
  \includegraphics[height=1.5in,trim=0.75in 0.75in 0.75in 0.75in,
  clip]{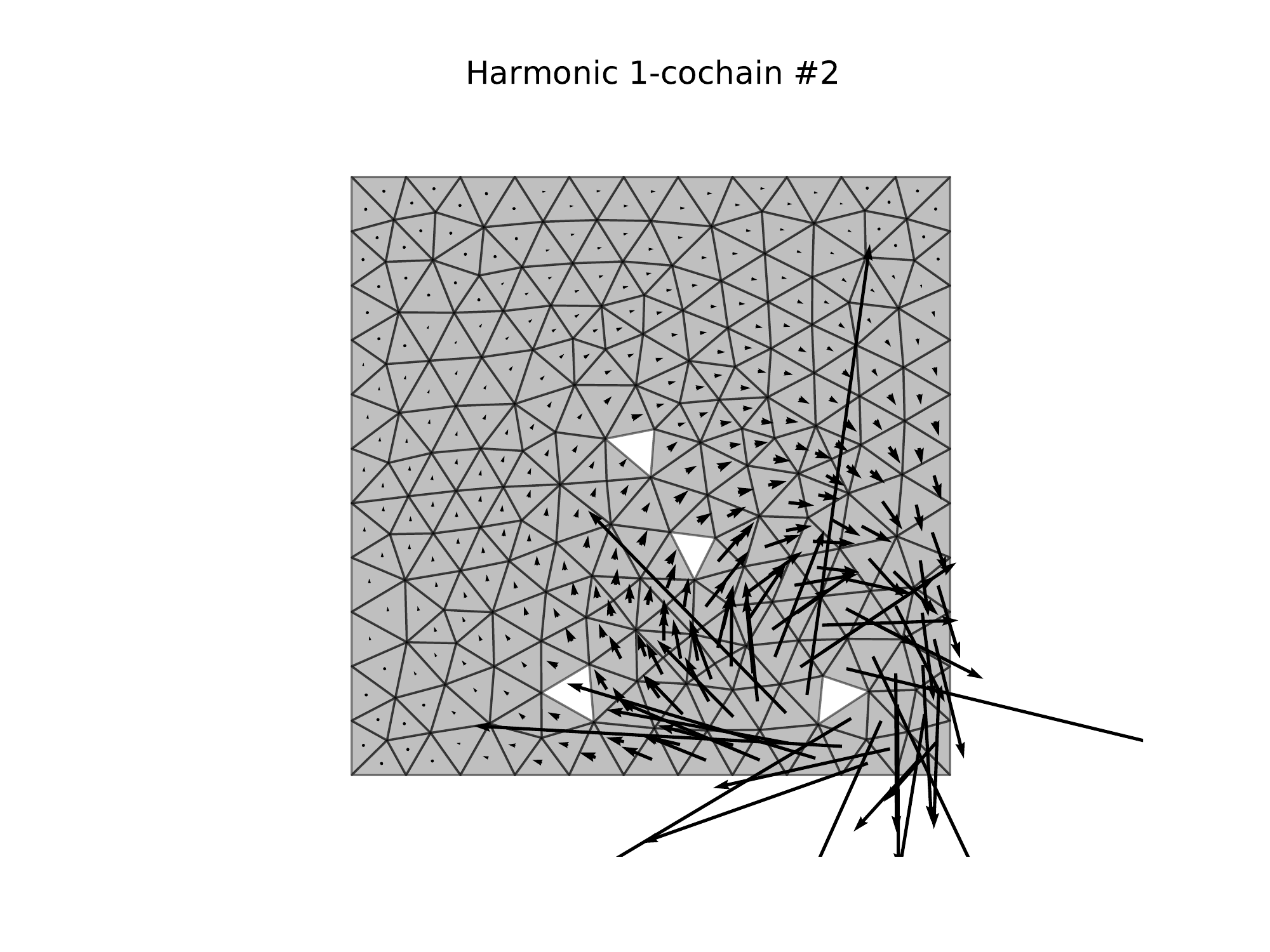}
  \includegraphics[height=1.5in,trim=0.75in 0.75in 0.75in 0.75in,
  clip]{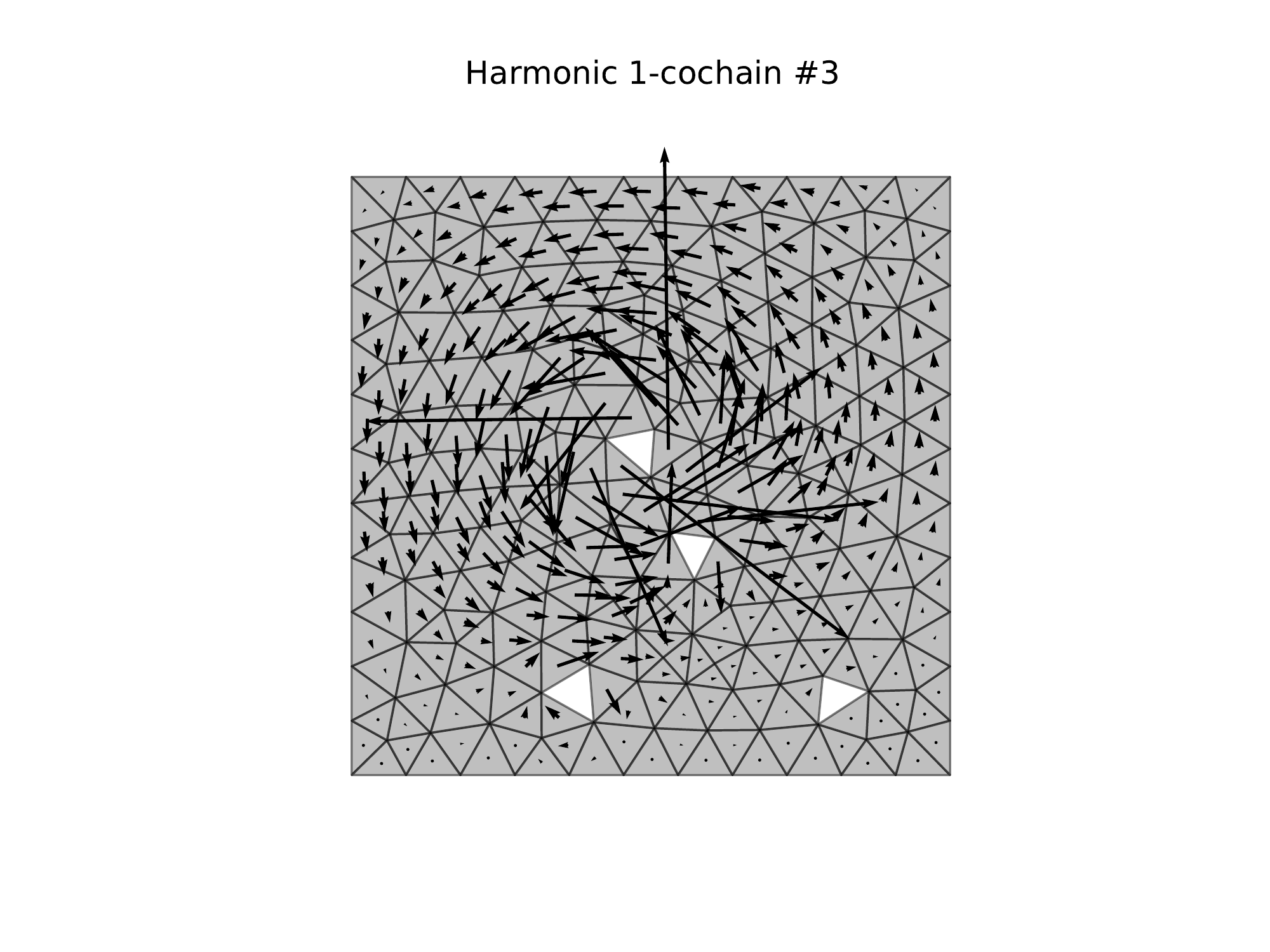}
  \caption{Four harmonic cochains form a basis for the first
    cohomology space $H^1$ for a mesh with four holes. Each cochain is
    visualized above as a vector field by interpolating it from the
    edge values using Whitney interpolation. See
    Section~\ref{subsec:cohomology} for details.}
  \label{fig:hodge_harmonic_basis}
\end{figure}

The harmonic $1$-forms shown in Figure~\ref{fig:hodge_harmonic_basis}
are obtained by decomposing random $1$-forms and retaining their
harmonic components.  Since the initial basis has no particular
spatial structure, an ad hoc orthogonalization procedure is then
applied.  For each basis vector, the algorithm identifies the
component with the maximum magnitude and applies Householder
transforms to force the other vectors to zero at that same component.

\subsection{Sensor network coverage}
\label{subsec:sensor}

As discussed in Section~\ref{sec:rips_complex}, sensor network
coverage gaps can be identified with coordinate-free methods based on
topological properties of the Rips complex. This is for an idealized
abstraction of a sensor network. The following example constructs a
\texttt{rips\_complex} object from a set of 300 points randomly
distributed over the unit square, as illustrated in top left in
Figure~\ref{fig:sensor_network}.  Recall that the Rips complex is
constructed by adding an edge between each pair each pair of points
within a given radius.  Top right of Figure~\ref{fig:sensor_network}
illustrates the edges of the Rips complex produced by a cut-off radius
of 0.15.  The triangles of the Rips complex, illustrated in bottom
left of Figure~\ref{fig:sensor_network}, represent triplets of
vertices that form a clique in the edge graph of the Rips complex. The
Rips complex is created by the following two lines of code
\begin{lstlisting}[firstline=15,lastline=16]
pts = read_array('300pts.mtx')  # 300 random points in 2D
rc = rips_complex( pts, 0.15 )
\end{lstlisting}

The sensor network is tested for coverage holes by inspecting the
kernel of the matrix $\Delta_1 = \boundary_1^T \boundary_1 +
\boundary_2 \boundary_2^T$ \cite{SiGh2007}. Specifically,
null-vectors of $\Delta_1$, which are called harmonic 1-cochains (by
analogy with the definition of harmonic cochains used in the previous
subsection), reveal the presence of holes in the sensor network.  In
this example we explore the kernel of $\Delta_1$ by generating a
random 1-cochain \texttt{x} and extracting its harmonic part using a
discrete hodge decomposition as outlined in the previous
subsection. If the harmonic component of \texttt{x} is (numerically)
zero then we may conclude with high confidence that $\Delta_1$ is
nonsingular and that no holes are present. However, in this case the
hodge decomposition of \texttt{x} produces a nonzero harmonic
component \texttt{h}. Indeed, plotting \texttt{h} on the edges of the
Rips complex localizes the coverage hole, as the bottom right of
Figure~\ref{fig:sensor_network} demonstrates.

To set up the linear systems, the boundary matrices are obtained from
the Rips complex \texttt{rc} created above:
\begin{lstlisting}[firstline=20,lastline=22]
cmplx = rc.chain_complex()  # boundary operators [ b0, b1, b2 ]
b1 = cmplx[1].astype(float)  # edge boundary operator
b2 = cmplx[2].astype(float)  # face boundary operator
\end{lstlisting}
Then the random cochain is created and the Hodge decomposition
computed, to find the harmonic cochain which is then normalized:
\begin{lstlisting}[firstline=24,lastline=29]
x = rand(b1.shape[1]) # random 1-chain
# Decompose x using discrete Hodge decomposition
alpha = cg( b1 * b1.T, b1   * x, tol=1e-8)[0]
beta  = cg( b2.T * b2, b2.T * x, tol=1e-8)[0]
h = x - (b1.T * alpha) - (b2 * beta) # harmonic component of x
h /= abs(h).max() # normalize h
\end{lstlisting}

\begin{figure}
 \centering
 \includegraphics[width=0.3\textwidth]{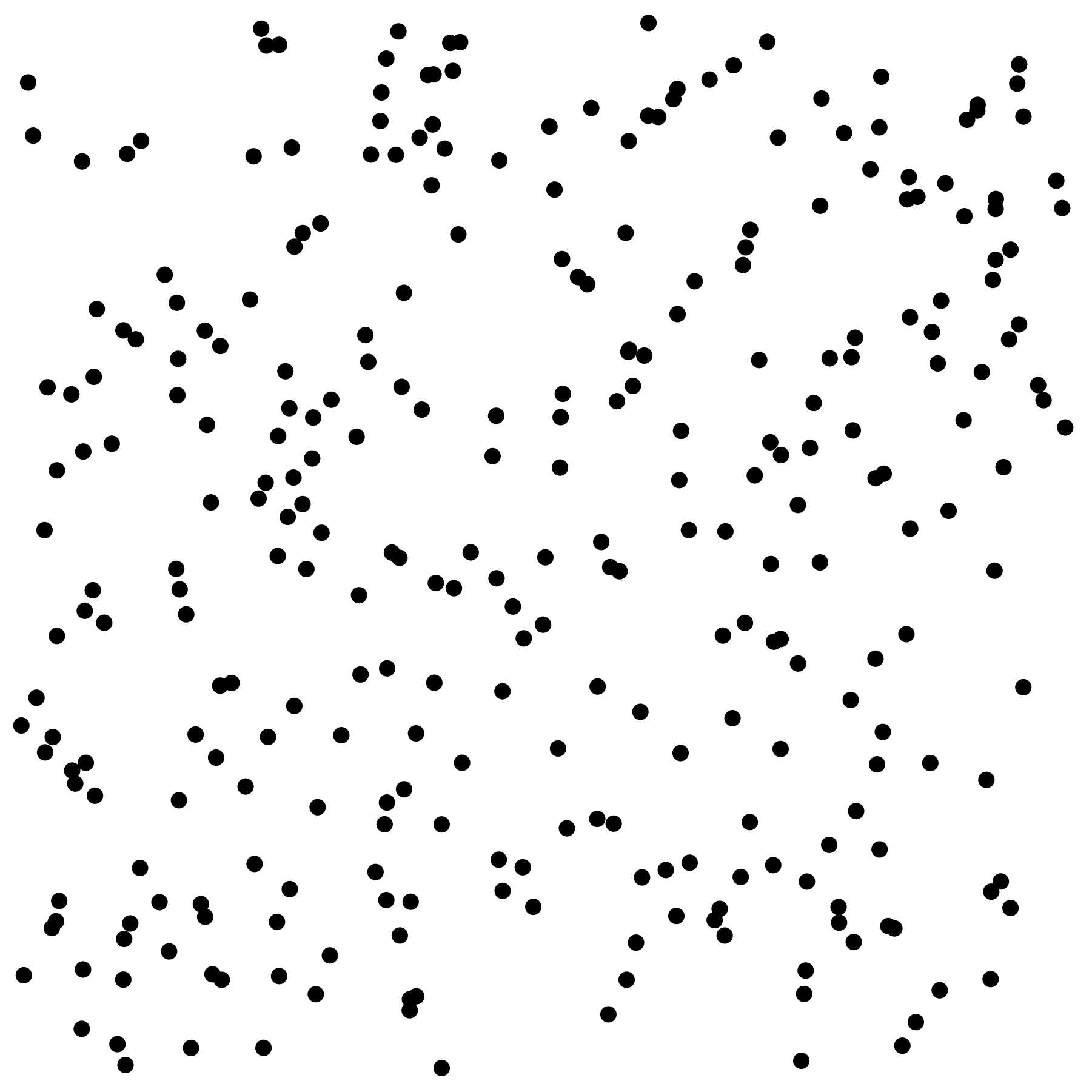}
 \includegraphics[width=0.3\textwidth]{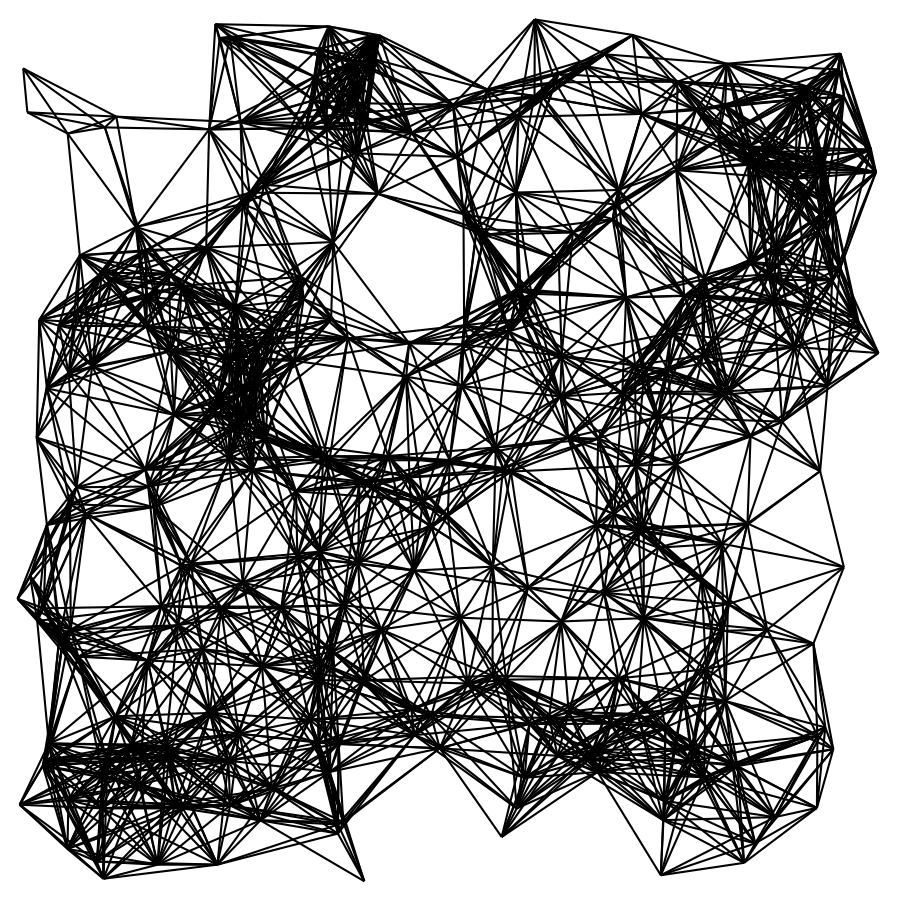}\\
 \includegraphics[width=0.3\textwidth]{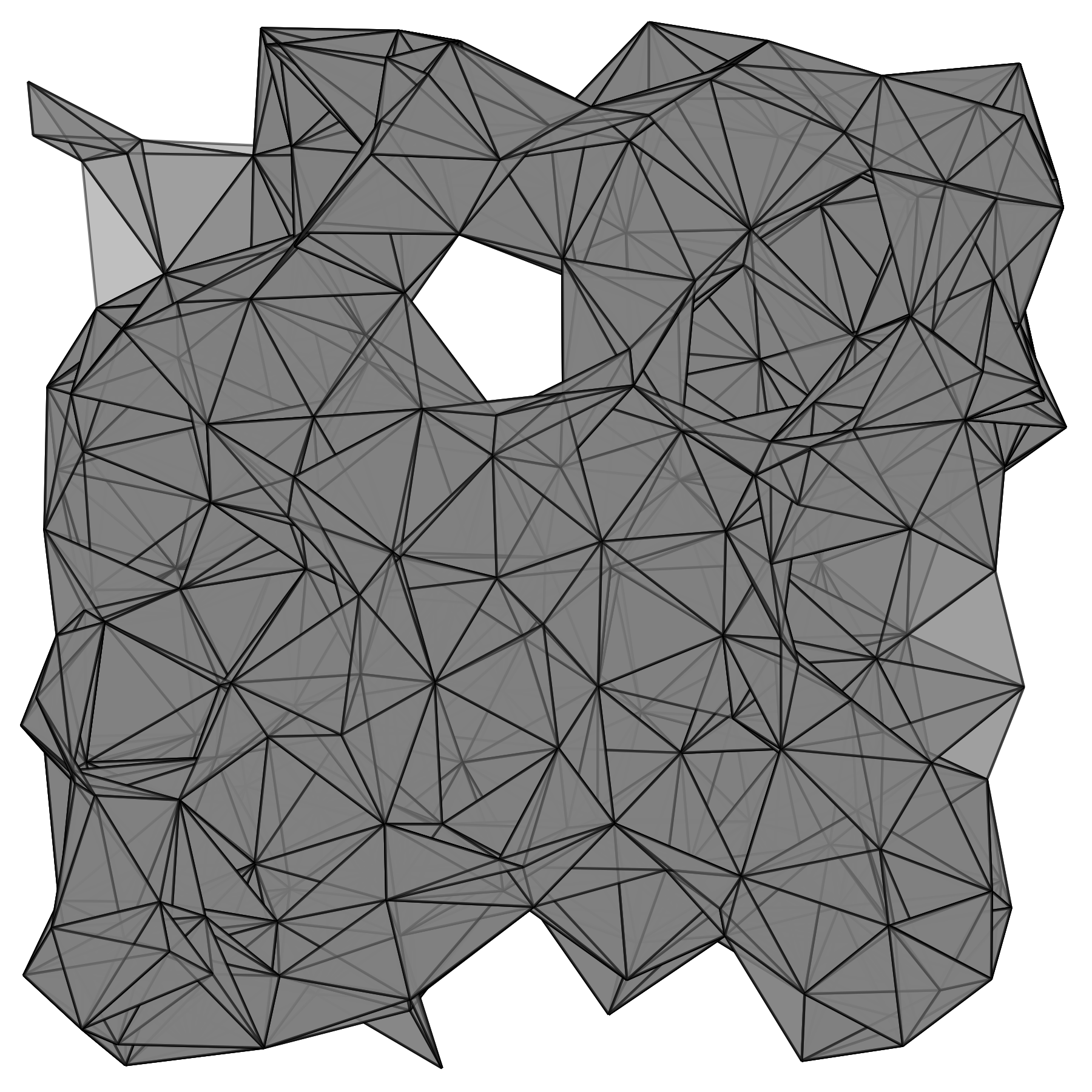}
 \includegraphics[trim=1.9in 1in 1.8in 1in, clip, scale=0.48]
 {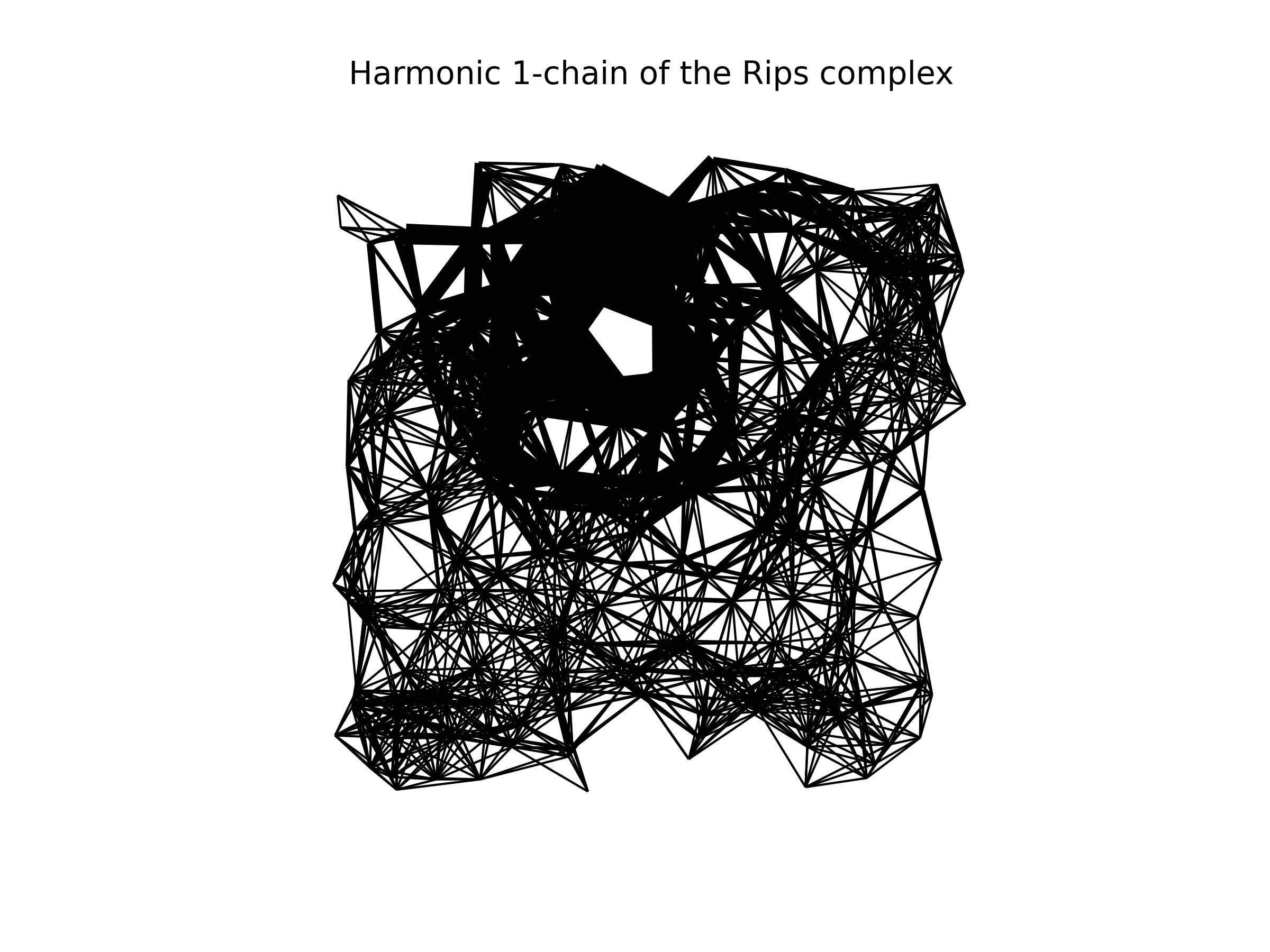}
 \caption{Hodge decomposition for finding coverage holes in an
   idealized sensor network. Top left shows a sample sensor network
   with 300 randomly distributed points. Pairs of points within a
   fixed distance of one another are connected by an edge and this is
   shown on top right. Triangles are added to the Rips complex when
   three points form a clique (complete graph). These are shown in
   bottom left. The existence of a harmonic 1-cochain indicates a
   potential hole in the sensor network coverage. In the bottom right
   figure, edge thickness reflects the magnitude of the harmonic
   cochain on each edge. See Section~\ref{subsec:sensor} for more
   details.}
 \label{fig:sensor_network}
\end{figure}

\subsection{Least squares ranking on graphs} \label{subsec:ranking}

This is a formulation for ranking alternatives based on pairwise data.
Given is a collection of alternatives or objects that have to be
ranked, by computing a ranking score that sorts them. The ranking
scores are to be computed starting from some pairwise
comparisons. Some examples of objects to be ranked are basketball
teams, movies, candidates for a job. Typically, the given data will
not have pairwise comparisons for all the possible pairs. There is no
geometry in this application, hence no exterior calculus is
involved. PyDEC however still proves useful because the full version
of this example~\cite{HiKaWa2011} uses an abstract simplicial
2-complex. Thus PyDEC is useful in forming the complex and for
determining its boundary matrix. In this simplified example, only an
abstract simplicial 1-complex is needed.

Form a simple graph $G$, with the objects to be ranked being the
vertices and with an edge between any two which have pairwise
comparison data given. If there are $n$ objects, possibly only a
sparse subset out of all possible $O(n^2)$ pairs may have comparison
data associated with them. Here we'll only require that the graph be
connected. This condition can be dropped with the consequence that the
rankings of separate components become independent of each other. The
comparison values are real numbers. 

Since $G$ is a simple graph, by orienting the edges arbitrarily it
becomes an oriented 1-dimensional abstract simplicial complex. The
vector $\omega$ of pairwise comparison values is a 1-cochain since if
A is preferred over B by, say, $4$ points, then B is preferred over A
by $-4$ points.

The ranking scores $\alpha$ which are to be computed on vertices form
a 0-cochain. For any edge $e = (u,v)$ from vertex $u$ to vertex $v$,
the difference of vertex values $\alpha(v) - \alpha(u)$ should match
$\omega(e)$ as much as possible, for example, in a least squares
sense. This idea is from~\cite{Leake1976} who proposed it as a method
for ranking football teams. By including the 3-cliques as triangles,
$G$ becomes a 2-dimensional simplicial complex. This was used
in~\cite{JiLiYaYe2011} to extend this ranking
idea. In~\cite{JiLiYaYe2011} the computation of the scores $\alpha$ is
interpreted as one part of the Hodge decomposition of $\omega$. See
Section~\ref{subsec:cohomology} above for a basic discussion of Hodge
decomposition where it is used for computing harmonic cochains on a
mesh. Here we will just compute the ranking score $\alpha$. This is
done by solving the least squares problem
%\begin{equation}\label{eq:lstsqr}
$\boundary_1^T \, \alpha \simeq \omega$.
%\end{equation}

The graph in this example is used for ranking basketball teams, using
real data for a small subset of American Men's college basketball
games from 2010-2011 season. Each team is a node in the graph and has
been given a number as a name. An edge between two teams indicates
that one of more games have been played between them. The score
difference from these games becomes the input 1-cochain $\omega$, with
one value on each edge. If multiple games were played by a pair the
score differences were added to create this data. The data is stored
as a matrix in which the first two columns are the teams and the third
column is the value of the 1-cochain on that edge.

% \[
% \begin{bmatrix}
% \phantom{-}8&\phantom{-}1&-9\\
% \phantom{-}5&\phantom{-}10&\phantom{-}8\\
% \phantom{-}9&\phantom{-}3&-23\\
% \phantom{-}6&\phantom{-}9&\phantom{-}13\\
% \phantom{-}2&\phantom{-}9&\phantom{-}23\\
% \phantom{-}8&\phantom{-}7&\phantom{-}3\\
% \phantom{-}\vdots & \phantom{-}\vdots & \phantom{-}\vdots
% \end{bmatrix}
% \]
Once this data is loaded from file, an abstract simplicial complex is
created from the first two columns which form the edges of the
graph. The loading and complex creation is done by the following few
lines of code
\begin{lstlisting}[firstline=17,lastline=21]
data = loadtxt('data.txt').astype(int)
edges = data[:,:2]

# Create abstract simplicial complex from edges
asc = abstract_simplicial_complex([edges])
\end{lstlisting}
In PyDEC, the simplices that are given as input to construct a complex
are preserved as is. Lower dimensional simplices that are derived from
them are stored and oriented in sorted order. Thus in the above data,
the edge between node 8 and 1 will be oriented from 8 to 1. The above
example data may mean, for example, that team labelled 1 lost to team
labelled 8 by 9 points.

The 1-cochain $\omega$ is now extracted from the data array and the
boundary matrix needed is obtained from the complex and the least
squares problem solved. All this is accomplished in the following
lines
\begin{lstlisting}[firstline=23,lastline=25]
omega = data[:,-1] # pairwise comparisons
B1 = asc.chain_complex()[1] # boundary matrix
alpha = lsqr(B1.T, omega)[0] # solve least squares problem
\end{lstlisting} The
resulting alpha values computed are given below.
\[
\begin{array}{lccccccccccc}
\text{Team number}  & 0 &1 &2 &3 &4 &5 &6 &7 &8 &9 &10\\
\alpha \text{ value} & 14.5&2.0&0.0&7.2&5.9&9.0&2.3&23.6&11.0&23.8&21.7
\end{array}
\]
The team with $\alpha=0$ is the worst team according to these rankings
and the one with the largest $\alpha$ value (23.8 here) is the best
team. Note that the score difference from the game between 8 and 1
happens to be exactly the difference in $\alpha$ values between
them. This won't always be true. See for example teams 3 and 9. Such
discrepancies come from having a residual in the least squares
solution, and that is a direct result of the presence of cycles in the
graph. In fact it may even happen that team A beats B, which beats C,
which in turn beats A. Thus no assignment of $\alpha$ values will
resolve this inconsistency. This is where a second least squares
problem and hence a Hodge decomposition plays a role. The second
problem is not considered in this example. See~\cite{HiKaWa2011} for
details on the second least squares problem and the role of Hodge
decomposition and harmonic cochains in the ranking context.
\begin{figure}[t]
  \centering
  \includegraphics[scale=0.4, trim=1in 0in 1in 0in, clip]
  {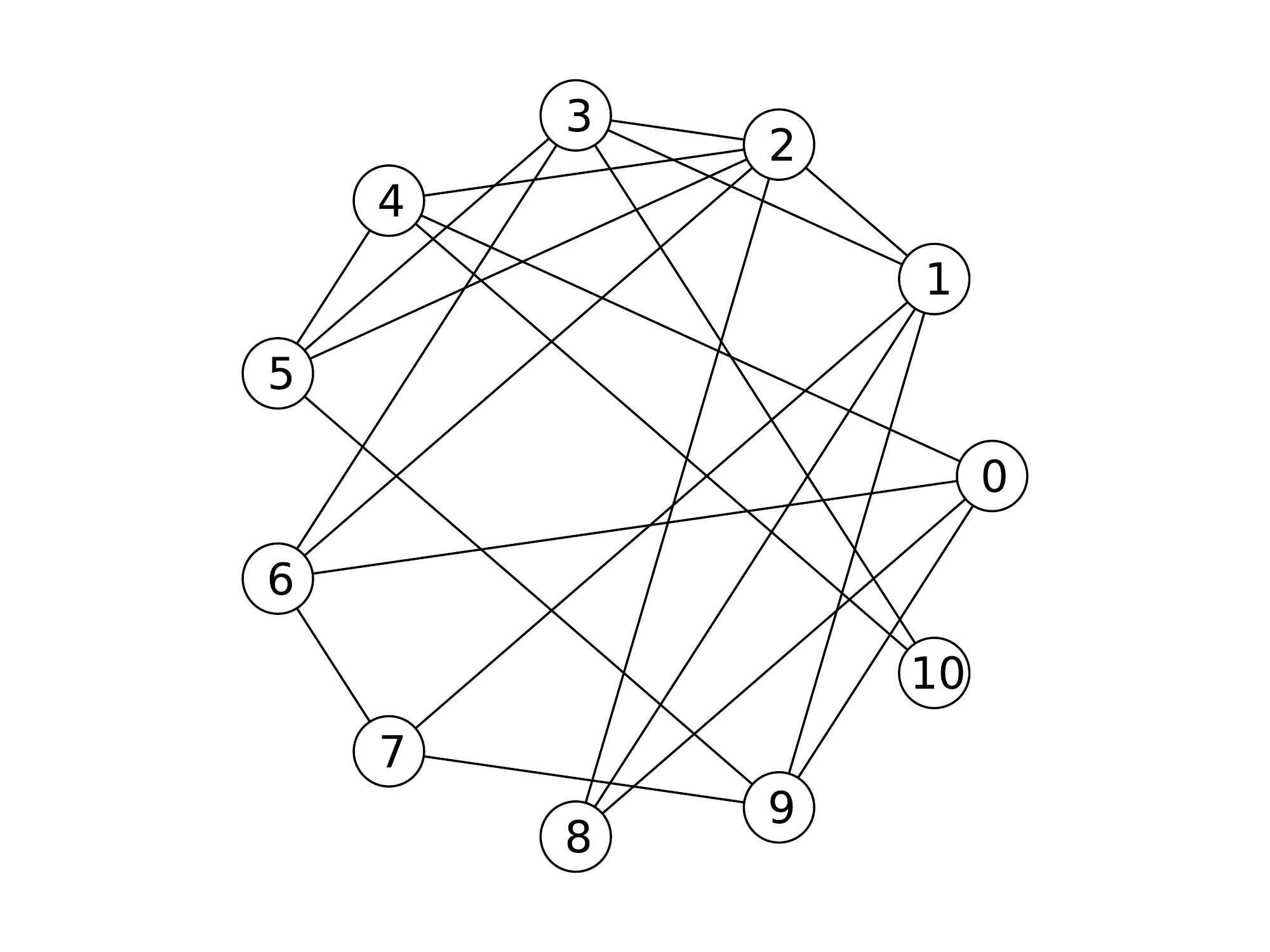}
  \caption{A typical graph of a subset of basketball games. Each node
    is a team, labelled by a number. An edge represents one or more
    games played by the two teams connected by it. The actual graph
    used in the example in Section~\ref{subsec:ranking} is a complete
    graph.}
  \label{fig:ranking}
\end{figure}

\section{Conclusions}

PyDEC is intended to be a tool for solving elliptic PDEs formulated in
terms of differential forms and for exploring computational topology
problems. It has been used for numerical experiments for Darcy
flow~\cite{HiNaCh2011}, computation of harmonic cochains on two and
three dimensional meshes~\cite{HiKaWaWa2011}, and least squares
ranking on graphs~\cite{HiKaWa2011}. It has also proved valuable in
computational topology work~\cite{DeHiKr2010,DuHi2011}, for creating
complexes and computing boundary matrices. The design goals for PyDEC
have been efficiency and ability to express mathematical formulations
easily. Section~\ref{sec:examples} which described some examples
should give an idea of how close the PyDEC code is to the mathematical
formulation of the problems considered. Many packages exist and are
being created for numerical PDE solutions using differential
forms. There are also many excellent computational topology packages.
PyDEC can handle a large variety of complexes and provides
implementations of discrete exterior calculus and lowest order finite
element exterior calculus using Whitney forms. These qualities make it
a convenient tool to explore the interrelationships between topology,
geometry, and numerical PDEs.

\section*{Acknowledgment}
This work was funded in part by NSF CAREER Award Grant DMS-0645604.

\bibliographystyle{acmdoi} 
\bibliography{hirani}

\begin{thebibliography}{10}

\bibitem{AbMaRa1988}
{\sc Abraham, R., Marsden, J.~E., and Ratiu, T.}
\newblock {\em Manifolds, Tensor Analysis, and Applications}, second~ed.
\newblock Springer--Verlag, New York, 1988.

\bibitem{ArFaWi2006}
{\sc Arnold, D.~N., Falk, R.~S., and Winther, R.}
\newblock Finite element exterior calculus, homological techniques, and
  applications.
\newblock In {\em Acta Numerica}, A.~Iserles, Ed., vol.~15. Cambridge
  University Press, 2006, pp.~1--155.

\bibitem{ArFaWi2009}
{\sc Arnold, D.~N., Falk, R.~S., and Winther, R.}
\newblock Geometric decompositions and local bases for spaces of finite element
  differential forms.
\newblock {\em Comput. Methods Appl. Mech. Engrg. 198}, 21-26 (2009),
  1660--1672.
\newblock \href {http://dx.doi.org/10.1016/j.cma.2008.12.017} {\path{doi:
  10.1016/j.cma.2008.12.017}}.

\bibitem{ArFaWi2010}
{\sc Arnold, D.~N., Falk, R.~S., and Winther, R.}
\newblock Finite element exterior calculus: from {H}odge theory to numerical
  stability.
\newblock {\em Bull. Amer. Math. Soc. (N.S.) 47}, 2 (2010), 281--354.
\newblock \href {http://dx.doi.org/10.1090/S0273-0979-10-01278-4} {\path{doi:
  10.1090/S0273-0979-10-01278-4}}.

\bibitem{BaHaKa07}
{\sc Bangerth, W., Hartmann, R., and Kanschat, G.}
\newblock {deal.II} -- a general purpose object oriented finite element
  library.
\newblock {\em ACM Trans. Math. Softw. 33}, 4 (2007), 24/1--24/27.

\bibitem{BaDo1993}
{\sc Bank, R., and Douglas, C.}
\newblock {Sparse matrix multiplication package (SMMP)}.
\newblock {\em Advances in Computational Mathematics 1}, 1 (1993), 127--137.

\bibitem{BeHi2008a}
{\sc Bell, N., and Hirani, A.~N.}
\newblock {PyDEC}: A {P}ython library for {D}iscrete {E}xterior {C}alculus
  [online].
\newblock Software made available on Google Code website.

\bibitem{Bell2008}
{\sc Bell, W.~N.}
\newblock {\em Algebraic Multigrid for Discrete Differential Forms}.
\newblock PhD thesis, University of Illinois at Urbana-Champaign, Urbana,
  Illinois, 2008.

\bibitem{BoLe2005}
{\sc Bochev, P., and Lehoucq, R.~B.}
\newblock On the finite element solution of the pure {N}eumann problem.
\newblock {\em SIAM Review 47}, 1 (2005), 50--66.
\newblock \href {http://dx.doi.org/10.1137/S0036144503426074} {\path{doi:
  10.1137/S0036144503426074}}.

\bibitem{BoHy2006}
{\sc Bochev, P.~B., and Hyman, J.~M.}
\newblock Principles of mimetic discretizations of differential operators.
\newblock In {\em Compatible Spatial Discretizations}, D.~N. Arnold, P.~B.
  Bochev, R.~B. Lehoucq, R.~A. Nicolaides, and M.~Shashkov, Eds., vol.~142 of
  {\em The IMA Volumes in Mathematics and its Applications}. Springer, Berlin,
  2006, pp.~89--119.

\bibitem{BoRo2002}
{\sc Bochev, P.~B., and Robinson, A.~C.}
\newblock Matching algorithms with physics: exact sequences of finite element
  spaces.
\newblock In {\em Collected Lectures on Preservation of Stability Under
  Discretization}, D.~Estep and S.~Tavener, Eds. Society for Industrial and
  Applied Mathematics (SIAM), 2002, ch.~8, pp.~145--166.

\bibitem{BoFeGaPe1999}
{\sc Boffi, D., Fernandes, P., Gastaldi, L., and Perugia, I.}
\newblock Computational models of electromagnetic resonators: analysis of edge
  element approximation.
\newblock {\em SIAM J. Numer. Anal. 36}, 4 (1999), 1264--1290 (electronic).

\bibitem{Bossavit1988a}
{\sc Bossavit, A.}
\newblock Whitney forms : {A} class of finite elements for three-dimensional
  computations in electromagnetism.
\newblock {\em IEE Proceedings 135, Part A}, 8 (November 1988), 493--500.

\bibitem{BoTu1982}
{\sc Bott, R., and Tu, L.~W.}
\newblock {\em Differential Forms in Algebraic Topology}.
\newblock Springer--Verlag, New York, 1982.

\bibitem{CaRiWh2005}
{\sc Castillo, P., Rieben, R., and White, D.}
\newblock {FEMSTER}: An object oriented class library of high-order discrete
  differential forms.
\newblock {\em {ACM} Transactions on Mathematical Software 31}, 4 (Dec. 2005),
  425--457.

\bibitem{Rham1955}
{\sc de~Rham, G.}
\newblock {\em Vari\'et\'es diff\'erentiables. {F}ormes, courants, formes
  harmoniques}.
\newblock Actualit\'es Sci. Ind., no. 1222 = Publ. Inst. Math. Univ. Nancago
  III. Hermann et Cie, Paris, 1955.

\bibitem{SiGh2007}
{\sc de~Silva, V., and Ghrist, R.}
\newblock Homological sensor networks.
\newblock {\em Notices of the American Mathematical Society 54}, 1 (2007),
  10--17.

\bibitem{DeDz2007}
{\sc Demlow, A., and Dziuk, G.}
\newblock An adaptive finite element method for the {L}aplace--{B}eltrami
  operator on implicitly defined surfaces.
\newblock {\em SIAM Journal on Numerical Analysis 45}, 1 (2007), 421--442.
\newblock \href {http://dx.doi.org/10.1137/050642873} {\path{doi:
  10.1137/050642873}}.

\bibitem{DeHiLeMa2005}
{\sc Desbrun, M., Hirani, A.~N., Leok, M., and Marsden, J.~E.}
\newblock Discrete exterior calculus, August 2005.
\newblock \href {http://arxiv.org/abs/math.DG/0508341}
  {\path{arXiv:math.DG/0508341}}.

\bibitem{DeHiKr2010}
{\sc Dey, T.~K., Hirani, A.~N., and Krishnamoorthy, B.}
\newblock Optimal homologous cycles, total unimodularity, and linear
  programming.
\newblock In {\em STOC '10: Proceedings of the 42nd ACM Symposium on Theory of
  Computing\/} (New York, NY, USA, June 6--8 2010), ACM, pp.~221--230.
\newblock \href {http://dx.doi.org/10.1145/1806689.1806721} {\path{doi:
  10.1145/1806689.1806721}}.

\bibitem{Dodziuk1976}
{\sc Dodziuk, J.}
\newblock Finite-difference approach to the {H}odge theory of harmonic forms.
\newblock {\em Amer. J. Math. 98}, 1 (1976), 79--104.

\bibitem{DuHi2011}
{\sc Dunfield, N.~M., and Hirani, A.~N.}
\newblock The least spanning area of a knot and the optimal bounding chain
  problem.
\newblock In {\em Proceedings of the 27th annual ACM symposium on Computational
  geometry\/} (New York, NY, USA, 2011), SoCG '11, ACM, pp.~135--144.
\newblock \href {http://dx.doi.org/10.1145/1998196.1998218} {\path{doi:
  10.1145/1998196.1998218}}.

\bibitem{Eckmann1945}
{\sc Eckmann, B.}
\newblock Harmonische {F}unktionen und {R}andwertaufgaben in einem {K}omplex.
\newblock {\em Comment. Math. Helv. 17\/} (1945), 240--255.
\newblock \href {http://dx.doi.org/10.1007/BF02566245} {\path{doi:
  10.1007/BF02566245}}.

\bibitem{EdLeZo2002}
{\sc Edelsbrunner, H., Letscher, D., and Zomorodian, A.}
\newblock Topological persistence and simplification.
\newblock {\em Discrete and Computational Geometry 28}, 4 (November 2002),
  511--533.
\newblock \href {http://dx.doi.org/10.1007/s00454-002-2885-2} {\path{doi:
  10.1007/s00454-002-2885-2}}.

\bibitem{Frankel2004}
{\sc Frankel, T.}
\newblock {\em The Geometry of Physics}, second~ed.
\newblock Cambridge University Press, Cambridge, 2004.
\newblock An introduction.

\bibitem{GiBa2010a}
{\sc Gillette, A., and Bajaj, C.}
\newblock Dual formulations of mixed finite element methods, 2010.
\newblock \href {http://arxiv.org/abs/1012.3929v3} {\path{arXiv:1012.3929v3}}.

\bibitem{GiBa2010}
{\sc Gillette, A., and Bajaj, C.}
\newblock A generalization for stable mixed finite elements.
\newblock In {\em SPM '10: Proceedings of the 14th ACM Symposium on Solid and
  Physical Modeling\/} (New York, NY, USA, 2010), ACM, pp.~41--50.
\newblock \href {http://dx.doi.org/10.1145/1839778.1839785} {\path{doi:
  10.1145/1839778.1839785}}.

\bibitem{GrHi1999}
{\sc Gradinaru, V., and Hiptmair, R.}
\newblock Whitney elements on pyramids.
\newblock {\em Electronic Transactions on Numerical Analysis 8\/} (1999),
  154--168.

\bibitem{HiPoWa2006}
{\sc Hildebrandt, K., Polthier, K., and Wardetzky, M.}
\newblock On the convergence of metric and geometric properties of polyhedral
  surfaces.
\newblock {\em Geom. Dedicata 123\/} (2006), 89--112.
\newblock \href {http://dx.doi.org/10.1007/s10711-006-9109-5} {\path{doi:
  10.1007/s10711-006-9109-5}}.

\bibitem{Hiptmair2002a}
{\sc Hiptmair, R.}
\newblock Finite elements in computational electromagnetism.
\newblock In {\em Acta Numerica}, A.~Iserles, Ed., vol.~11. Cambridge
  University Press, 2002, pp.~237--339.

\bibitem{Hirani2003}
{\sc Hirani, A.~N.}
\newblock {\em Discrete Exterior Calculus}.
\newblock PhD thesis, California Institute of Technology, May 2003.

\bibitem{HiKaWaWa2011}
{\sc Hirani, A.~N., Kalyanaraman, K., Wang, H., and Watts, S.}
\newblock Cohomologous harmonic cochains, 2011.
\newblock \href {http://arxiv.org/abs/1012.2835} {\path{arXiv:1012.2835}}.

\bibitem{HiKaWa2011}
{\sc Hirani, A.~N., Kalyanaraman, K., and Watts, S.}
\newblock Least squares ranking on graphs, 2011.
\newblock \href {http://arxiv.org/abs/1011.1716} {\path{arXiv:1011.1716}}.

\bibitem{HiNaCh2011}
{\sc Hirani, A.~N., Nakshatrala, K.~B., and Chaudhry, J.~H.}
\newblock Numerical method for {D}arcy flow derived using {D}iscrete {E}xterior
  {C}alculus, 2011.
\newblock \href {http://arxiv.org/abs/0810.3434} {\path{arXiv:0810.3434}}.

\bibitem{HoSt2011}
{\sc Holst, M., and Stern, A.}
\newblock Geometric variational crimes: Hilbert complexes, finite element
  exterior calculus, and problems on hypersurfaces, May 2011.
\newblock \href {http://arxiv.org/abs/1005.4455} {\path{arXiv:1005.4455}}.

\bibitem{HySh1997a}
{\sc Hyman, J.~M., and Shashkov, M.}
\newblock Natural discretizations for the divergence, gradient, and curl on
  logically rectangular grids.
\newblock {\em Comput. Math. Appl. 33}, 4 (1997), 81--104.
\newblock \href {http://dx.doi.org/10.1016/S0898-1221(97)00009-6} {\path{doi:
  10.1016/S0898-1221(97)00009-6}}.

\bibitem{JiLiYaYe2011}
{\sc Jiang, X., Lim, L.-H., Yao, Y., and Ye, Y.}
\newblock Statistical ranking and combinatorial hodge theory.
\newblock {\em Mathematical Programming 127\/} (2011), 203--244.
\newblock \href {http://dx.doi.org/10.1007/s10107-010-0419-x} {\path{doi:
  10.1007/s10107-010-0419-x}}.

\bibitem{Jost2005}
{\sc Jost, J.}
\newblock {\em Riemannian Geometry and Geometric Analysis}, fourth~ed.
\newblock Universitext. Springer-Verlag, Berlin, 2005.
\newblock \href {http://dx.doi.org/10.1007/3-540-28890-2} {\path{doi:
  10.1007/3-540-28890-2}}.

\bibitem{KaMiKoMr2004}
{\sc Kaczynski, T., Mischaikow, K., and Mrozek, M.}
\newblock {\em Computational homology}, vol.~157 of {\em Applied Mathematical
  Sciences}.
\newblock Springer-Verlag, New York, 2004.

\bibitem{Leake1976}
{\sc Leake, R.~J.}
\newblock A method for ranking teams: With an application to college football.
\newblock In {\em Management Science in Sports}, R.~E. Machol and S.~P. Ladany,
  Eds., vol.~4 of {\em TIMS Studies in the Management Sciences}. North-Holland
  Publishing Company, 1976, pp.~27--46.

\bibitem{LoMa2008}
{\sc Logg, A., and Mardal, K.-A.}
\newblock A symbolic engine for finite element exterior calculus.
\newblock Talk at European Finite Element Fair, May 2008.

\bibitem{LoMaWe2012}
{\sc Logg, A., Mardal, K.-A., Wells, G.~N., et~al.}
\newblock {\em Automated Solution of Differential Equations by the Finite
  Element Method}.
\newblock Springer, 2012.

\bibitem{LoWe2010}
{\sc Logg, A., and Wells, G.~N.}
\newblock {DOLFIN}: Automated finite element computing.
\newblock {\em ACM Trans. Math. Softw. 37\/} (April 2010), 20:1--20:28.
\newblock \href {http://dx.doi.org/http://doi.acm.org/10.1145/1731022.1731030}
  {\path{doi: http://doi.acm.org/10.1145/1731022.1731030}}.

\bibitem{Morita2001}
{\sc Morita, S.}
\newblock {\em Geometry of Differential Forms}, vol.~201 of {\em Translations
  of Mathematical Monographs}.
\newblock American Mathematical Society, Providence, RI, 2001.

\bibitem{Munkres1984}
{\sc Munkres, J.~R.}
\newblock {\em Elements of Algebraic Topology}.
\newblock Addison--Wesley Publishing Company, Menlo Park, 1984.

\bibitem{NiTr2006}
{\sc Nicolaides, R.~A., and Trapp, K.~A.}
\newblock Covolume discretization of differential forms.
\newblock In {\em Compatible Spatial Discretizations}, D.~N. Arnold, P.~B.
  Bochev, R.~B. Lehoucq, R.~A. Nicolaides, and M.~Shashkov, Eds., vol.~142 of
  {\em The IMA Volumes in Mathematics and its Applications}. Springer, New
  York, 2006, pp.~161--171.

\bibitem{Oliphant2007}
{\sc Oliphant, T.~E.}
\newblock Python for scientific computing.
\newblock {\em Computing in Science \& Engineering 9}, 3 (2007), 10--20.

\bibitem{Saad2003}
{\sc Saad, Y.}
\newblock {\em Iterative Methods for Sparse Linear Systems}, second~ed.
\newblock Society for Industrial and Applied Mathematics, Philadelphia, PA,
  2003.

\bibitem{Sen2003}
{\sc Sen, S.}
\newblock A cubic {W}hitney and further developments in geometric
  discretisation.
\newblock {\em arXiv:hep-th/0307166v2\/} (2003), Online.

\bibitem{Shashkov1996}
{\sc Shashkov, M.}
\newblock {\em Conservative finite-difference methods on general grids}.
\newblock CRC Press, Boca Raton, FL, 1996.

\bibitem{WaCoVa2011}
{\sc van~der Walt, S., Colbert, S., and Varoquaux, G.}
\newblock The {NumPy} array: A structure for efficient numerical computation.
\newblock {\em Computing in Science Engineering 13}, 2 (2011), 22--30.
\newblock \href {http://dx.doi.org/10.1109/MCSE.2011.37} {\path{doi:
  10.1109/MCSE.2011.37}}.

\bibitem{Whitney1957}
{\sc Whitney, H.}
\newblock {\em Geometric Integration Theory}.
\newblock Princeton University Press, Princeton, N. J., 1957.

\bibitem{Wilson2007}
{\sc Wilson, S.}
\newblock Cochain algebra on manifolds and convergence under refinement.
\newblock {\em Topology and its Applications 154}, 9 (May 2007), 1898--1920.
\newblock \href {http://dx.doi.org/10.1016/j.topol.2007.01.017} {\path{doi:
  10.1016/j.topol.2007.01.017}}.

\bibitem{Wilson2008}
{\sc Wilson, S.~O.}
\newblock Conformal cochains.
\newblock {\em Transactions of the American Mathematical Society 360\/} (2008),
  5247--5264.

\bibitem{ZoCa2005}
{\sc Zomorodian, A., and Carlsson, G.}
\newblock Computing persistent homology.
\newblock {\em Discrete and Computational Geometry 33}, 2 (February 2005),
  249--274.
\newblock \href {http://dx.doi.org/10.1007/s00454-004-1146-y} {\path{doi:
  10.1007/s00454-004-1146-y}}.

\end{thebibliography}

\section*{Appendix}
{\bf Proof of Proposition~\ref{prop:whtny_ip}:}
By definition $ M_p(i,j) = \int_{\abs{K}} \aInnerproduct {\whitney
  \cochainBasis{p}{i}} {\whitney \cochainBasis{p}{j}} \mu $.  The
integrand is nonzero only in $\ClSt(\sigma^p_i) \cap
\ClSt(\sigma^p_i)$ since the Whitney form corresponding to a simplex
is zero outside the star of that simplex \cite{Dodziuk1976}.  Thus
  \[
  M_p(i,j) = \underset{\sigma^n \hasface \sigma^p_i, \sigma^p_j}
  {\sum_{\sigma_n}}\int_{\sigma_n} 
  \aInnerproduct
      {\whitney \cochainBasis{p}{i}\bigr|_{\sigma^n}}
      {\whitney \cochainBasis{p}{j}\bigr|_{\sigma^n}} \mu \, .
      \]
By the definition of the Whitney map, $M_p(i,j)$ is 
%\begin{multline*}
\[
  (p!)^2 \underset{\sigma^n \hasface \sigma^p_i, \sigma^p_j}
  {\sum_{\sigma^n}} 
  \sum_{k,l=0}^p(-1)^{k+l}\\*
  \int_{\sigma^n}
  \aInnerproduct{\d\mu_{i_0}\wedge\dots\widehat{\d\mu_{i_k}}\dots\wedge
    \d\mu_{i_p}}
  {\d\mu_{j_0}\wedge\dots\widehat{\d\mu_{j_l}}\dots\wedge
    \d\mu_{j_p}} \mu_{i_k}\mu_{j_l} \, \mu\, .
\]
%\end{multline*}
But differentials of barycentric coordinates are constant in
$\sigma^n$.  Thus the term
\[
\aInnerproduct{\d\mu_{i_0}\wedge\dots\widehat{\d\mu_{i_k}}\dots\wedge
  \d\mu_{i_p}}
{\d\mu_{j_0}\wedge\dots\widehat{\d\mu_{j_l}}\dots\wedge
  \d\mu_{j_p}}
\]
comes out of the integral.  Recalling
Definition~\ref{def:monomials_ip} of inner product of forms,
\[
\aInnerproduct{\d\mu_{i_0}\wedge\dots\widehat{\d\mu_{i_k}}\dots\wedge
  \d\mu_{i_p}}
{\d\mu_{j_0}\wedge\dots\widehat{\d\mu_{j_l}}\dots\wedge
  \d\mu_{j_p}}
\]
is given by
\[
\det
\begin{bmatrix}
  \aInnerproduct{\d \mu_{i_0}}{\d \mu_{j_0}} & \dots & 
  \widehat{\aInnerproduct{\d \mu_{i_0}}{\d \mu_{j_l}}} &
  \dots & \aInnerproduct{\d \mu_{i_0}}{\d \mu_{j_p}}\\
  \vdots & & \vdots & & \vdots\\
  \widehat{\aInnerproduct{\d \mu_{i_k}}{\d \mu_{j_0}}} & 
  \dots & 
  \widehat{\aInnerproduct{\d \mu_{i_k}}{\d \mu_{j_l}}} & 
  \dots &
  \widehat{\aInnerproduct{\d \mu_{i_k}}{\d \mu_{j_l}}}\\
  \vdots & & \vdots & & \vdots\\
  \aInnerproduct{\d \mu_{i_p}}{\d \mu_{j_0}} & \dots &
  \widehat{\aInnerproduct{\d \mu_{i_p}}{\d \mu_{j_l}}} &
  \dots & \aInnerproduct{\d \mu_{i_p}}{\d \mu_{j_p}}
\end{bmatrix} \, ,
\]
which completes the proof.

\end{document}